\documentclass[a4paper]{amsart}
\pdfoutput=1

\usepackage[english]{babel}
\usepackage[utf8]{inputenc}
\usepackage{amsmath, amsthm, amssymb, stmaryrd, xcolor, enumitem}
\usepackage{bm}

\definecolor{linkred}{rgb}{0.6,0,0}
\usepackage[
	pdftitle={The r-ELSV Formula via Localisation on the Moduli Space of Stable Maps with Divisible Ramification},
	pdfauthor={Oliver Leigh},
	ocgcolorlinks,
	linkcolor=black,
	citecolor=linkred,
	urlcolor=linkred]
{hyperref}

\usepackage[all]{xy} 

\usepackage{tikz}
\usepackage{tikz-cd}
\usetikzlibrary{%
  matrix,%
  calc,%
  arrows,%
  positioning,
  backgrounds
}

\newcounter{mainTheoremCounter}
\newtheoremstyle{mainTheorem}{}{}{\addtolength{\leftskip}{0em}\itshape}{}{\bfseries}{}{.5em}{\thmname{#1} {#2}}
\theoremstyle{mainTheorem}
\newtheorem{mainTheorem}[mainTheoremCounter]{Theorem}

\newcounter{mainDefinitionCounter}
\newtheoremstyle{mainDefinition}{}{}{}{}{\bfseries}{}{.5em}{\thmname{#1} {#2}}
\theoremstyle{mainDefinition}
\newtheorem{mainDefinition}[mainDefinitionCounter]{Definition}

\newcounter{mainRemarkCounter}
\newtheoremstyle{mainRemark}{}{}{}{}{\bfseries}{}{.5em}{\thmname{#1} {#2}}
\theoremstyle{mainRemark}
\newtheorem{mainRemark}[mainRemarkCounter]{Remark}

\newcounter{appendixTheoremCounter}
\newtheoremstyle{appendixTheorem}{}{}{\addtolength{\leftskip}{0em}\itshape}{}{\bfseries}{}{.5em}{\thmname{#1} {#2}}
\theoremstyle{appendixTheorem}

\theoremstyle{plain}  

\newtheorem{theorem}{Theorem}[subsection]
\newtheorem{lemma}[theorem]{Lemma}

\newtheorem{corollary}[theorem]{Corollary}

\theoremstyle{definition} 

\newtheorem{definition}[theorem]{Definition}

\newtheorem{re}[theorem]{}
\newtheorem{remark}[theorem]{Remark}

\newcommand{\num}{\tag{\addtocounter{equation}{1}\arabic{equation}}}

\newcommand{\SmNeg}{\mbox{\scriptsize-}}

\newcommand{\C}{\mathbb{C}}
\newcommand{\E}{\mathbb{E}}
\newcommand{\F}{\mathbb{F}}
\newcommand{\Z}{\mathbb{Z}}
\newcommand{\bbL}{\mathbb{L}}
\newcommand{\eL}{\bbL^{\hspace{-0.2em}\mathtt{e}}}
\newcommand{\eE}{\E^{\mathtt{e}}}
\newcommand{\eF}{\mathbb{F}^{\mathtt{e}}}
\newcommand{\eT}{\mathbb{T}^{\mathtt{e}}}
\newcommand{\ephi}{\phi^{\mathtt{e}}}

\renewcommand{\P}{\mathbb{P}}

\newcommand{\calB}{\mathcal{B}}
\newcommand{\calC}{\mathcal{C}}

\newcommand{\calE}{\mathcal{E}}
\newcommand{\calF}{\mathcal{F}}

\newcommand{\calH}{\mathcal{H}}

\newcommand{\calP}{\mathcal{P}}
\newcommand{\calR}{\mathcal{R}}

\newcommand{\calT}{\mathcal{T}}
\newcommand{\calU}{\mathcal{U}}
\newcommand{\calV}{\mathcal{V}}
\newcommand{\calW}{\mathcal{W}}
\newcommand{\calX}{\mathcal{X}}
\newcommand{\calY}{\mathcal{Y}}
\newcommand{\calZ}{\mathcal{Z}}

\newcommand{\frakM}{\mathfrak{M}}
\newcommand{\frakD}{\mathfrak{D}}
\newcommand{\Tot}{\mathsf{Tot}}

\newcommand{\fixedSM}{\mathtt{G}}
\newcommand{\fixedIncSM}{\bm{k}}
\newcommand{\fixedRSM}{\mathtt{G}^r}
\newcommand{\fixedIncRSM}{{\bm{k}^r}}
\newcommand{\fixedDR}{\mathtt{F}}
\newcommand{\fixedIncDR}{\bm{j}}
\newcommand{\nonsimpDR}{\widehat{\mathtt{F}}}

\renewcommand{\part}{\lambda}

\newcommand{\fix}{\mathtt{fix}}
\newcommand{\mov}{\mathtt{mov}}
\newcommand{\fixmov}{\mathtt{f.m.}}

\newcommand{\ForMap}{\bm{\mathtt{q}}} 

\newcommand{\ForMapSM}{\bm{\mathtt{p}}} 

\newcommand{\ForMapRSMtoSM}{\bm{\upsilon}}
\newcommand{\ForMapRSMtoSMfixed}{\bm{\mathtt{u}}}

\newcommand{\ForMapDRtoRSM}{\bm{\nu}}
\newcommand{\ForMapDRtoRSMfixed}{\bm{\mathtt{v}}}

\renewcommand{\O}{\mathcal{O}}
\newcommand{\Spec}{\mathrm{Spec}}
\newcommand{\sheafHom}{\mathcal{H}om}
\newcommand{\Sym}{\mathrm{Sym}}

\newcommand{\M}{\mathcal{M}}
\newcommand{\Mbar}{\overline{\mathcal{M}}}
\renewcommand{\L}{\mathcal{L}}

\newcommand{\smallfamily}[3]{\begin{array}{c}\xymatrix@=1em{{#1}\ar[d]^{#3} \\ {#2}}\end{array}}
\newcommand{\specialcell}[1]{\ifmeasuring@#1\else\omit$\displaystyle#1$\ignorespaces\fi}

\newcommand{\notationConvention}{
\textbf{Notation Conventions for Main Spaces}:
We will use the following simplifying notation for the key spaces:
\begin{enumerate}
\item$\M := \Mbar_g (X,\part)$ from definition \ref{relative_SM_definition} with $\bm{\pi}: \calC \rightarrow \M$ the universal curve.
\item$\M^r := \Mbar^{r}_g (X,\part)$ from definition \ref{RSM_definition} with $\bm{\pi}^r: \calC^r \rightarrow \M^r$ the universal ($r$-twisted) curve.
\item$\M^{1/r} := \Mbar^{_{1/r}}_g (X,\part)$ from definition \ref {M1/r_main_definition} with $\bm{\pi}^{1/r}: \calC^{1/r} \rightarrow \M^{1/r}$ the universal ($r$-twisted) curve.
\item$\frakM := \frakM_{g,l}$ and $\frakM^r := \frakM^{r}_{g,l}$ from defintion \ref{r-twisted_curves_definition}.
\end{enumerate}
The natural forgetful morphisms are denoted by
\begin{align*}
\xymatrix{
\M^{\frac{1}{r}}\ar[r]^{\ForMapDRtoRSM} &\M^r \ar[r]^{\ForMapRSMtoSM}& \M.
}
\end{align*}

We will also denote by $\calR$, $\L$ and  $\bm{\delta}:\O_{\calC}\rightarrow \calR$ the bundles and morphism on $\calC^{1/r}$ pulled back from $\calC^r$.  Lastly, $\bm{\sigma}:\O_{\calC}\rightarrow \L$ is the morphism pulled back from $\mathsf{Tot} \bm{\pi}_*\L$.
}

\begin{document}
\setlength{\parindent}{0cm}

\title[$r$-ELSV via Localisation using Stable Maps with Divisible Ramification]{The $r$-ELSV Formula via Localisation on the Moduli Space of Stable Maps with Divisible Ramification}

\author{Oliver Leigh}
\address{Oliver Leigh, Matematiska institutionen, Stockholms universitet, 106 91 Stockholm, Sweden}
\curraddr{}
\email{oliver.leigh@math.su.se}
\thanks{}

\subjclass[2010]{14D23, 14H10, 14N35}

\keywords{virtual localisation; ELSV formula; moduli spaces; stable maps; twisted curves; spin structures; Hurwitz numbers}
\date{}
\dedicatory{}

\begin{abstract}
The moduli space of stable maps with divisible ramification uses $r$-th roots of a canonical ramification section to parametrise stable maps whose ramification orders are divisible by a fixed integer $r$. In this article, a virtual fundamental class is constructed while letting domain curves have a positive genus; hence removing the restriction of the domain curves being genus zero. We apply the techniques of virtual localisation and obtain the $r$-ELSV formula as an intersection of the virtual class with a pullback via the branch morphism. 
\end{abstract}

\maketitle

\vspace{-0.7cm}
\section*{Introduction}
For a smooth curve $X$, the moduli space of stable maps with divisible ramification was introduced in \cite{Leigh_Ram} as a natural compactification of the sub-moduli space
\[
\M^{_{1/r}}_g(X, d) = \Big\{~ \big[f:C\rightarrow X \big] \in \M_g(X, d)~\Big|~  \mbox{$R_f = r\cdot D$ for some $D\in\mathrm{Div}(C)$} ~\Big\} /\sim 
\]
 of $\M_{g}(X,d)$ where $R_f$ denotes the ramification divisor of $f$. \\

The compactification of \cite{Leigh_Ram} agrees with the extended concept of ramification for stable maps introduced in \cite{FantechiPand} and \cite{GraVakil}. This extended concept is based on the observation that, for smooth curves, the ramification divisor arises from the differential map  $df: f^*\Omega_{X} \rightarrow \Omega_C$. When the domains are nodal, the differential map can be combined with the natural morphism $\Omega_C \rightarrow \omega_C$ to obtain a morphism 
\begin{align}
\delta: \O_C \longrightarrow \omega_C \otimes f^* \omega^\vee_{X}.  \label{delta_defn_1}
\end{align}
The extended concept of ramification is determined by $\delta$ and it is shown in  \cite{Leigh_Ram} that the concept of \textit{divisibility of ramification order by $r$} is equivalent to requiring $r$-th roots of $\delta$.\\

In order to control the $r$-th roots, it is convenient to consider a moduli space with slightly more information. Hence, following \cite{Leigh_Ram}, we will refer to the following moduli space as \textit{the moduli space of stable maps with divisible ramification}. \\

\vspace{-0.2cm}
\begin{mainDefinition}\label{mainModuliDefinition} For $g \geq 0$ and $d>0$ denote by $\Mbar^{_{1/r}}_g(X, d)$ the moduli stack that parameterises $(f: C\rightarrow X,\,\, L,\,\, e:L^{\otimes r}\overset{\sim}{\rightarrow}\omega_C \otimes f^*\omega_{X}^\vee ,\,\, \sigma:\O_C\rightarrow \L)$ where:
				\begin{enumerate}
					\item $C$ is a $r$-prestable curve of genus $g$ (a stack such that the coarse space $\overline{C}$ is a prestable curve of genus $g$, where points corresponding to nodes of $\overline{C}$ are balanced $r$-orbifold points, and ${C}^{\mathrm{sm}} \cong \overline{C}^{\mathrm{sm}}$).
					\item $f$ is a morphism such that the induced morphism $\overline{f}:\overline{C}\rightarrow X$ on the coarse space is a stable map parametrised by $\Mbar_{g}(X,d)$.
					\item  $L$ is a line bundle on $C$ and $ e: L^{\otimes r} \overset{\sim}{\rightarrow} \omega_C \otimes f^* \omega^\vee_{X}$  is an isomorphism.
					\item $\sigma: \O_C \rightarrow L$ is a section such that $e(\sigma^r) = \delta$, where $\delta$ is defined in (\ref{delta_defn_1}).\\
				\end{enumerate}
					
\end{mainDefinition}

\begin{mainRemark}
This article will be primarily concerned with the relative version of the space from definition \ref{mainModuliDefinition}. We denote this by $\Mbar^{_{1/r}}_g(X, \part)$ where $\part$ is an \textit{ordered} partition. However, we will leave the technicalities of the relative version until section \ref {background_section}.\\
\end{mainRemark}

$\Mbar^{_{1/r}}_g(X, \part)$ is a proper Deligne-Mumford stack which is non-empty only when $r$ divides $2g-2+l(\part)+|\part|(1-2g_X)$ \cite[Thm. A]{Leigh_Ram}. The morphism to the ``normal'' moduli space of stable maps (i.e. forget the $r$-th root and $r$-twisted structures), 
\begin{align} \label{main_forgetul_morphism_intro}
\Mbar^{_{1/r}}_g(X, \part) \longrightarrow \Mbar^{}_g(X, \part),
\end{align}
is both flat and of relative dimension $0$ onto its image. \\

The moduli space $\Mbar^{_{1/r}}_g(X, \part)$ is often not equidimensional and requires a virtual fundamental class to perform intersection theory. Also, in the case when $X=\P^1$ the natural $\C^*$-action on $\P^1$ induces an action on the moduli space. A virtual class in the $\C^*$-equivariant setting is a powerful computational tool. In \cite{Leigh_Ram} a virtual class is constructed for the case $g=0$. The following theorem extends this result to the equivariant setting and to include the case when $g>0$.\\

\begin{mainTheorem} \label{POT_theorem}
 $\Mbar^{_{1/r}}_g(\P^1, \part)$ has a natural $\C^*$-equivariant perfect obstruction theory giving a virtual fundamental class of dimension $\frac{1}{r}(2g-2+l(\part)+|\part|)$.\\
\end{mainTheorem}

For the rest of the introduction we set $m :=m(\P^1,g,\part) := \frac{1}{r}(2g-2+l(\part)+|\part|)$. It is shown in \cite{Leigh_Ram} that there is a natural morphism of stacks
\[
\mathtt{br} : \Mbar^{_{1/r}}_g(\P^1, \part) \longrightarrow \Sym^{m} \P^1 
\]
which is compatible with the branch morphism $\bm{br}$ of \cite{FantechiPand} via the $r$-th diagonal morphism $\Delta$ which is defined by $\sum_i x_i \mapsto \sum_i rx_i$:
\[
\xymatrix@R=1.25em{
 \Mbar^{_{1/r}}_g(\P^1,\part) \ar[r]^{\hspace{0.1em}\mathtt{br}}\ar[d] &\Sym^{m} \P^1 \ar[d]^{\Delta}\\
 \Mbar_g(\P^1,\part) \ar[r]^{\hspace{0.1em}\bm{br}} &\Sym^{rm} \P^1 
}
\]
Specifically, $\mathtt{br}$ takes a moduli point to its branch divisor divided by $r$.\\

The methods of \cite{Kontsevich, GraVakil_Loc} give an explicit description of the $\C^*$-fixed loci of $\Mbar^{}_g(\P^1, \part)$. Furthermore, we will identify the $\C^*$-fixed loci of $\Mbar^{_{1/r}}_g(\P^1, \part)$ as pullbacks of the $\C^*$-fixed loci of $\Mbar^{}_g(\P^1, \part)$ via the forgetful morphism of (\ref{main_forgetul_morphism_intro}). \\

The $\C^*$-fixed loci of $\Mbar^{_{1/r}}_g(\P^1, \part)$ can be characterised by their image under $\mathtt{br}$. Following \cite{GraVakil_Loc} we call the fixed loci corresponding to the point $[m\cdot(0)] \in \Sym^{m}\P^1$ the \textit{simple fixed locus} and denote this locus by $\fixedDR$. This contains the $\C^*$-fixed maps where there is no degeneration of the target at $\infty$. \\

\begin{mainTheorem}\label{VL_theorem}
The virtual localisation formula of \cite{GraberPand,ChangKiemLi} can be applied to $\Mbar^{_{1/r}}_g(\P^1,\part)$. The simple fixed locus $\fixedDR$ is smooth of dimension $3g-3+l(\part)$ and there is a morphism of degree $ (\part_1\cdots \part_{l(\part)})^{-1}$ to the moduli space of $r$-spin curves 
\[
\bm{b}: \fixedDR \longrightarrow \Mbar^{\frac{1}{r}, \mathtt{a}}_{g,l(\part)}
\]
where $\mathtt{a}= (a_1,\ldots, a_{l(\part)})$ is a vector of reverse remainders with $a_i\in \{0,\ldots, r-1\}$ defined by $\part_i = \left\lfloor \frac{\part_i}{r} \right\rfloor r + (r-1 -a_i)$.\\

Moreover, the contribution to the localisation formula of $\big[\Mbar^{_{1/r}}_g(\P^1,\part)\big]^{\mathrm{vir}}$ by the virtual normal bundle of the simple fixed locus is: 
\[
\frac{1}{e\left(\mathrm{N}^{\mathrm{vir}}_{\fixedDR}\right)}
=
 r^{l+ 2g-2}
\left(\prod_{i =1}^{l} \frac{\part_i \left(\frac{\part_i}{r}\right)^{\left\lfloor \frac{\part_i}{r} \right\rfloor}}{\left\lfloor \frac{\part_i}{r} \right\rfloor!} \right)
 \frac{c_{\frac{r}{t}}\big(-\bm{b}^*R\bm{\rho}_* \L \big)}{\prod_{i=1}^{l}\left(1 - \frac{\part_i }{t}\, \bm{b}^*\psi_i\right)}
 \left(\frac{t}{r}\right)^{3g-3+l-m}
\]
where $\bm{\rho}$ and $\L$ are the universal curve and $r$-th root of $\Mbar^{1/r, \mathtt{a}}_{g,l}$, while $l:=l(\part)$   and  \,$t$ is the generator of the $\C^*$-equivariant Chow ring of a point.\\
\end{mainTheorem}

We can use the branch-type morphism and the virtual fundamental class to define the following natural Hurwitz-type formula
\[
H^r_{g,\part} :=\int_{\big[\Mbar^{_{1/r}}_g(\P^1, \part)\big]^{\mathrm{vir}}} \mathtt{br}^*[\, p_1+ \cdots + p_{m}] \label{hurwitz_integral}
\]
where $m=\frac{1}{r}(2g-2+l(\part)+|\part|)$ and $p_i\in\P^1$. We can choose an equivariant lift of the class $[\,p_1+ \cdots + p_{m}]$  that corresponds to the point $[m\cdot(0)] \in\Sym^m\P^1$. With this choice the class $\mathtt{br}^*[\,p_1+ \cdots + p_{m}]$ will vanish on the non-simple fixed loci. Hence we can apply the localisation formula of theorem \ref{VL_theorem}. Taking the non-equivariant limit of the resulting intersection gives the following formula. \\

\begin{mainTheorem}\label{rHurwitz_theorem}
There is an equality
\[
H^r_{g,\part}
=
 m!\, r^{m+ l+ 2g-2} 
 \left(
 \prod_{i=1}^{l} \frac{\left(\frac{\part_i}{r}\right)^{\lfloor \frac{\part_i}{r} \rfloor}}{\lfloor \frac{\part_i}{r} \rfloor!}
 \right)
~ \int_{\Mbar^{\frac{1}{r}, \mathtt{a}}_{g,l}} 
\frac{c(-R\bm{\rho}_* \mathcal{L})}{\prod_{j=1}^{l}(1 - \frac{\part_i}{r}\psi_j)}
\]
where $\bm{\rho}$ and $\L$ are the universal curve and $r$-th root of $\Mbar^{\frac{1}{r}, \mathtt{a}}_{g,l}$, while $\mathtt{a}= (a_1,\ldots, a_{l})$ is a vector with $a_i\in \{0,\ldots, r-1\}$ defined by $\part_i = \left\lfloor \frac{\part_i}{r} \right\rfloor r + (r-1 -a_i)$  and $l=l(\part)$.\\
\end{mainTheorem}

The intersection formula in theorem \ref{rHurwitz_theorem} is known as the $r$-ELSV formula. It appeared in \cite{SSZ_rELSV} as a conjecture relating to the stationary Gromov-Witten theory of $\P^1$
\begin{align*}
\num \label{Zvonkines_rELSV_Formula}(r!)^m \int&_{\big[\Mbar_{g,m}(\P^1,\part)\big]^{\mathrm{vir}}}~ \psi_1^r \mathrm{ev}_1^*[\mathrm{pt}] \cdots  \psi_m^r \mathrm{ev}_m^*[\mathrm{pt}]\\
&=
 m!\, r^{m+l+ 2g-2} 
 \left(
 \prod_{i=1}^{l} \frac{\left(\frac{\part_i}{r}\right)^{\lfloor \frac{\part_i}{r} \rfloor}}{\lfloor \frac{\part_i}{r} \rfloor!}
 \right)
~ \int_{\Mbar^{\frac{1}{r}, \mathtt{a}}_{g,l(\part)}} 
\frac{c(-R\bm{\rho}_* \mathcal{L})}{\prod_{j=1}^{l}(1 - \frac{\part_i}{r}\psi_j)}. 
\end{align*}
This formula has since been proved using the methods of Chekhov-Eynard-Orantin topological recursion in \cite{BKLPS,DKPS_Loop}. In fact, the proof in \cite{DKPS_Loop} is for a generalised version, conjectured in \cite{KLPS}, involving a $q$-orbifold $\P^1$. A version of this formula for one-part double Hurwitz numbers was recently considered in   \cite{DL_Double}.\\

In the sequel to this article \cite{Leigh_Deg} we will provide a geometric proof of the equality from (\ref{Zvonkines_rELSV_Formula}) using the methods of degenerated targets from \cite{li2} and the Gromov-Witten/Hurwitz correspondence from \cite{OkounkovPand_Comp}. \\

\hypertarget{notation_main}{\textbf{Notation}} 
Unless otherwise stated, the following notation will be used in this article:
\begin{itemize}[leftmargin=2em]
\item All stacks, schemes and varieties are  over $\C$.
\item $X$ is a smooth one dimensional variety. 
\item $\part$ is an \textit{ordered} partition with length $l:=l(\part)$ and $|\part|>0$. 
\item $m:= m(X,g,\part) := \frac{1}{r}(2g-2+l+|\part|(1-2g_X))$.
\item If $a\in \Z_{\geq0}$ then  $\left<\tfrac{a}{r} \right>$ is the remainder after dividing $a$ by $r$ (i.e. $a = \left\lfloor \tfrac{a}{r} \right\rfloor r + \left< \tfrac{a}{r} \right>$).
\end{itemize}

\tableofcontents
\phantomsection
\addcontentsline{toc}{section}{Contents}

\vspace{-0.7cm}
\section*{Acknowledgements} 
The author wishes to thank Dan Petersen, Jim Bryan and Paul Norbury for useful conversations which contributed to this article and Olof Bergvall for providing feedback on an early version of this manuscript.  \\

\section{Background}\label{background_section}

\subsection{Review of the Moduli Space of Relative Stable Maps}

Let $\part$ be an \textit{ordered} partition and fix $x\in X$. Relative stable maps parameterise maps where the pre-image of $x$ lies in the smooth locus of $C$ and where the map has monodromy above $x$  locally given by $\part$. To obtain a proper space we follow \cite{li1} and allow the target to degenerate in a controlled manner by allowing $X$ to sprout a chain of $\P^1$'s.\\

\begin{re}[\textit{Degenerated Targets and Relative Stable maps}]

For a smooth curve X, we can define the $i$-th degeneration $X[i]$ inductively from $X[0]:=X$ by:
\begin{enumerate}
\item $X[i+1]$ is given by the union $X[i] \cup \P^1$ meeting at a node $N_{i+1}$.
\item The node $N_{1}$ is at $x \in X$. For $i>0$ the node $N_{i+1}$ is in the $i$th component of $X[i+1]$, i.e. the node is not in $X[i-1] \subset X[i+1]$.
\end{enumerate}
Then a \textit{degenerated target} is a pair $(T, t)$ where $T=X[i]$ for some $i\geq 0$ and $t$ is a geometric point in the smooth locus of $i$th component of $T$.  \\

A genus $g$ stable map to $X$ relative to $(\part, x)$ is given by 
\[
\Big(~ h:C \longrightarrow T,~ p: T \longrightarrow X,~ q_1,\ldots,q_{l} , t   ~\Big)
\]
where $(C, q_i)$ is a $l$-marked genus $g$ prestable curve, $(T, t)$ is a degenerated target, $h$ is a morphism sending $q_i$ to $t$ and $p$ is a morphism sending $t$ to $x$ such that:
\begin{enumerate}
\item There is an equality of divisors on $C$ given by $h^{-1}(t) = \sum \part_i q_i$.
\item We have $p|_{X}$ is an isomorphism and $p|_{T\setminus X}:T\setminus X\rightarrow \{x\}$ is constant. 
\item The pre-image of each node $N$ of $T$ is a union of nodes of $C$. At any such node $N'$ of $C$, the two branches of $N'$ map to the two branches of $N$, and their orders of branching are the same.
\item The data has finitely many automorphisms (an automorphism is a a pair of isomorphisms $a:C\rightarrow C$ and $b:T\rightarrow T$ taking $q_i$ to $q_i$ and $t$ to $t$ such that $h\circ a = b\circ h$ and $p = p\circ b$). \\
\end{enumerate}
\end{re}

\begin{definition}[\textit{Moduli Space of Relative Stable Maps}]\label{relative_SM_definition} The moduli stack of genus $g$ stable maps relative to $(\part,x)$ is denoted $\Mbar_g(X,\part)$. It is the groupoid with:
\begin{enumerate}
\item Objects over a scheme $S$ given by:
\[
\xi  = \left(~
\hspace{-0.6em}\begin{array}{c}\xymatrix@=1em{{C}\ar[d]_{\pi} \\ {S} \ar@/_/[u]_{q_i}}\end{array}\hspace{-0.6em}
,~~ \hspace{-0.6em}\begin{array}{c}\xymatrix@=1em{{T}\ar[d]_{\pi'} \\ {S} \ar@/_/[u]_{t}}\end{array}\hspace{-0.6em}
,~~h: C \rightarrow T
,~~ p :T \rightarrow X
~\right)
\]
where $\pi$ and $\pi'$ are proper flat morphisms, $h$ is a morphism over $S$ and for each geometric point $z\in S$ we have $\xi_z$ is a genus $g$ stable map relative to $(\part,x)$. Furthermore, we require that in a neighbourhood of a node of $C_z$ mapping to a singularity of $T_z$ we can choose \'{e}tale-local coordinates on $S$, $C$ and $T$ with charts of the form $\Spec\,R$, $\Spec\,R[u,v]/(uv-a)$ and $\Spec\,R[x,y]/(xy-b)$ respectively such that the map is of the form $x \mapsto \alpha u^k$ and $y \mapsto \alpha v^k$ with $\alpha$ and $\beta$ units. 
\item
Morphisms $\xi_1 \rightarrow \xi_2$ between two appropriately labelled objects are given by pairs of cartesian diagrams
\[
\begin{array}{c}
\xymatrix@=1.5em{
 C_1\ar[d]_{\pi_1}\ar[r]_{a} &  C_2\ar[d]^{\pi_2}\\
S_1 \ar[r]^{a'} & S_2
}\end{array}
\hspace{1.5cm}
\begin{array}{c}
\xymatrix@=1.5em{
T_1\ar[d]_{\pi'_1}\ar[r]_{b} &  T_2\ar[d]^{\pi'_2}\\
S_1 \ar[r]^{b'} & S_2
}\end{array}
\]
that are compatible with the other data (i.e. we have $a\circ q_{1,i} = q_{2,i} \circ a'$, $b\circ t_{1} = t_{2} \circ b'$, $b\circ h_1 = h_2\circ a$ and $p_1 = p_2\circ b$).\\
\end{enumerate}
\end{definition}

\begin{re}[\textit{Universal Objects of $\Mbar_{g}(X,\part)$}]\label{universal_objects_of_moduli_relative_stable_maps}
Use the notation $\M := \Mbar_{g}(X,\part)$. There is a smooth Artin stack $\calT$ which parametrises the degenerated targets and a morphism which forgets the map data
\[
\ForMapSM = (\ForMapSM_{\frakM}, \ForMapSM_{\calT} ): \M \rightarrow \frakM_{g,l}\times  \calT. 
\]

We have universal curves $\bm{\pi}: \calC \rightarrow \M$ and $\bm{\pi}_\calT: \calC_{\calT} \rightarrow \calT$ along with universal maps $\bm{h}:\calC\rightarrow \calC_{\calT}$ and $\bm{p}:\calC_{\calT}\rightarrow X$ fitting into the following commuting diagram:
\[\xymatrix@R=1.5em{
\calC \ar[r]_{\bm{h}} \ar[d]_{\bm{\pi}}  &  \ar[d]^{\bm{\pi}_{\calT} } \calC_{\calT} \ar[r]_{\bm{p}} & X\\
\M \ar[r]^{\ForMapSM_{\calT}} & \calT & 
}\]
There also  universal sections defining the marked points given by $\bm{q}_i:\M\rightarrow \calC$ and $\bm{t}:\calT\rightarrow \calC_{\calT}.$ Moreover for convenience we define  $\bm{f}:= \bm{p}\circ \bm{h}$ to be the universal map $\bm{f}:\calC \rightarrow X$. 
\end{re}

\begin{definition}[\textit{Ramification Bundle}]\label{Ram_bun_def}
Using the notation from \ref{universal_objects_of_moduli_relative_stable_maps}, there is a natural bundle defined on the universal curve $\calC$ which we call the \textit{ramification bundle} and denote by
\[
\calR := \omega_{\bm{\pi}}^{\log} \otimes \bm{f}^*(\omega_{X}^{\log})^\vee, 
\] 
where we have also denoted $\omega_{\bm{\pi}}^{\log} = \omega_{\bm{\pi}}\big(\mbox{$\sum$} \bm{q}_i\big)$ and $\omega_{X}^{\log} = \omega_{X}(x)$. For a family $\xi\in \Mbar_{g}(X,\part)$ with $f:=\bm{f}_\xi$ we will use the notation $R_f:= \calR_\xi$. \\
\end{definition}

\begin{re}[\textit{Canonical Ramification Section}]\label{Can_ram_sec}
Using the notation from \ref{universal_objects_of_moduli_relative_stable_maps}, it is shown in \cite{Leigh_Ram} that (following choices of global sections defining the divisors $\bm{q}_i \in \calC$, $\bm{t} \in \calC_{\calT}$ and $x\in X$)  there is a morphism 
\[
\bm{\delta}: \O_{\calC} \longrightarrow \calR
\]
called the \textit{canonical ramification section} which has the following properties at each geometric point  $\xi = \big( h:(C,q_i) \longrightarrow (T,t) ,~ p: (T,t) \longrightarrow (X,x) \big)$ in  $\M$ with $f= p\circ h$ and $\delta = \bm{\delta}_\xi$:
\begin{itemize}[leftmargin=1.25em]
\item[]
  Let $B\subset X$ be the closed subset containing only $x$ and the nodes of $X$. If $D:=f^{-1}(B) \subset C$, then:
\begin{enumerate}
\item $\delta$ restricted to $C\setminus D$ is the natural morphism $\O_{C\setminus D} \rightarrow \omega_{C\setminus D}\otimes (f|_{C\setminus D})^* \omega_{X\setminus B}^\vee$  of \cite[Lemma 8]{FantechiPand}.
\item $\delta$ is an isomorphism locally at $D$. 
\end{enumerate}
\end{itemize}
In particular, $\delta$ describes the ramification behaviour of $f$ away from the pre-images of nodes of $X$ and away from the relative divisor.  \\
\end{re}

\begin{theorem}[\textit{Branch Morphism} \cite{FantechiPand}]\label{FP_branch_morphism_thm} There is a morphism of stacks 
\[
\bm{br} : \Mbar^{}_g(X, \part) \longrightarrow \Sym^{rm} X 
\]
defined at each geometric point in $\xi \in \Mbar_{g}(X,\part)$ with $C:=\calC_\xi$ and $f:= \bm{f}_{\xi}$ to be
\[
\mathrm{Div}\Big(Rf_*\big[ \O_{C} \overset{\delta_{\xi}}{\longrightarrow}  \calR_{\xi} \big] \Big) 
\]
where $\mathrm{Div}$ is Mumford's divisor associated to the determinant of a perfect complex whose homology is not supported on any points of depth 0 \cite[\S 5.3]{GIT}. \\

On the substack $\M^{}_g(X, \part)$ the morphism $\bm{br}$ takes a point $\xi$ to the branch divisor of $\bm{f}_\xi$ minus the sub-divisor supported at $x\in X$. \\
\end{theorem}

\begin{remark}
The construction of $\bm{br}$ in \cite{FantechiPand} is given for the space of absolute stable maps $\Mbar^{}_g(X, d)$. However the proofs all still work when applied to case of $\Mbar^{}_g(X, \part)$ using the canonical ramification section $\bm{\delta}$ of \ref{Can_ram_sec}. \\
\end{remark}

\begin{re}[\textit{A Natural $\C^*$-action on $\P^1$}] We define (for now and the rest of this article) an action of $\C^*$ on $\P^1$, by identifying $\P^1 := \P(\C^2)$ and letting $\C^*$ act on $\C^2$ with weights $0,1$. Explicitly, for $c\in \C^*$ and $(x_0,x_1)\in \C^2$ we define the action by 
\begin{align}
c\cdot (x_0, x_1) = ( x_0, c x_1).
\end{align}
This canonically induces a $\C^*$-action on $\P^1$ which we use throughout the rest of this article. The $\C^*$-fixed points on $\P^1$ of this action are identified as
\[
0 := [0:1] 
\hspace{1cm}
\mbox{and}
\hspace{1cm}
\infty := [1:0].
\]
The weights of the tangent spaces to $\P^1$ at these points are $-1$ and $1$ respectively. 
\end{re}

\begin{re}[\textit{A Natural $\C^*$-action on $\Mbar_g(\P^1,\part)$}]  \label{C*-action_relative_stable_maps_def}
We define  (for now and the rest of this article) an action of $\C^*$ on $\Mbar_g(\P^1,\part)$ by acting on the image of the map $c\cdot [f] = [c\cdot f]$ using the $\C^*$-action on $\P^1$. More precisely, the action is defined by the morphism of stacks
\begin{align*}
\C^*\times \Mbar_g(\P^1,\part) \longrightarrow  \Mbar_g(\P^1,\part)
\end{align*}
which maps a family over a scheme $S$
\[
\left(~\hspace{-0.5em}\begin{array}{c}\xymatrix@=1em{{S}\ar[d]_{a} \\ {\C^*} }\end{array}\hspace{-0.6em},~ \xi~\right) 
=
 \left(~
\hspace{-0.5em}\begin{array}{c}\xymatrix@=1em{{S}\ar[d]_{a} \\ {\C^*} }\end{array}\hspace{-0.6em}
,\hspace{-0.5em}\begin{array}{c}\xymatrix@=1em{{C}\ar[d]_{\pi} \\ {S} \ar@/_/[u]_{q_i}}\end{array}\hspace{-0.6em}
,\hspace{-0.5em} \begin{array}{c}\xymatrix@=1em{{T}\ar[d]_{\pi'} \\ {S} \ar@/_/[u]_{t}}\end{array}\hspace{-0.6em}
,~ h: C \rightarrow T
,~ p :T \rightarrow \P^1
~\right)
\]
to the family of relative stable maps 
\[
 \left(~
\hspace{-0.5em}\begin{array}{c}\xymatrix@=1em{{C}\ar[d]_{\pi} \\ {S} \ar@/_/[u]_{q_i}}\end{array}\hspace{-0.6em}
,\hspace{-0.5em} \begin{array}{c}\xymatrix@=1em{{T}\ar[d]_{\pi'} \\ {S} \ar@/_/[u]_{t}}\end{array}\hspace{-0.6em}
,~ h: C \rightarrow T
,~ \big((a\circ\pi')\cdot p\big) :T \rightarrow \P^1
~\right).
\]\\
\end{re}

\vspace{-0.25cm}
\begin{re}[\textit{Fixed Loci of the $\C^*$-action on $\Mbar_g(\P^1,\part)$}] \label{fixed_locus_SM} Following the methods of \cite{Kontsevich,GraberPand, GraVakil_Loc} we identify the fixed locus as containing maps $f=p\circ h: C\rightarrow \P^1$ which are étale everywhere except possibly above $0$ and $\infty$. This means that if $B\subseteq C$ is an irreducible component of $C$ where $f|_B$ is non-constant of degree $d$, then $B\cong \P^1$ and $f|_B$ is of the form $[x_0:x_1] \mapsto [x_0^d :x_1^d]$. \\

Moreover, the stack-theoretic image of $\M^{\C^*}$ under $\bm{br}$ is reduced and equal to $rm+1$ of points given by
\[
[(rm-n)\cdot(0) + n\cdot(\infty)] \in \Sym^{rm}\P^1,
\hspace{0.5cm}
\mbox{for $n=0,\ldots, rm$. }
\]\\
\end{re}

\vspace{-0.25cm}
\begin{re}\textit{Simple Fixed Loci $\Mbar_g(\P^1,\part)$}] \label{simple_fixed_locus}
As observed in \cite{FantechiPand,GraVakil_Loc} there is a unique connected component of $\Mbar_g(\P^1,\part)^{\C^*}$ where the branch divisor is supported only at $0\in \P^1$. This is also the unique connected component of $\Mbar_g(\P^1,\part)^{\C^*}$ where the target is not degenerated, hence it is called the \textit{simple locus} in \cite{GraVakil_Loc}. We denote it by 
\[
\fixedSM \subset \Mbar_g(\P^1,\part)^{\C^*}.
\]
A geometric point in $\fixedSM$ is a morphism $f$ with domain curve that is the union of
\begin{enumerate}
\item A single genus $g$ stable curve, which we denote $C_v$, with $f|_{C_v}: C_v \rightarrow \{0\}$;
\item $l$ copies of $\P^1$ denoted $C_i$ where $C_i\cap C_v$ is a single node and $C_i\cap C_j =\emptyset$ for $i\neq j$. Also, $f|_{C_i}: C_i \rightarrow \P^1$ is a $\C^*$-fixed Galois cover of degree $\part_i$.
\end{enumerate}
We call the nodes arising from an intersection between $C_v$ and a $C_i$ \textit{flag nodes}.\\

It is shown in \cite{GraberPand,GraVakil_Loc} that $\fixedSM$ is smooth and there is an isomorphism 
\begin{align}
\fixedSM \cong \Mbar_{g,l} \times \calP_{\part_1} \times \cdots \times \calP_{\part_{l}} \label{simple_fixed_locus_iso}
\end{align}
where $\Mbar_{g,l}$ is the moduli space of stable curves and $\calP_{d} \cong \calB \mu_d$ parametrises $\C^*$-fixed Galois covers of degree $d$. \\
\end{re}

\vspace{0.1cm}
\begin{remark}
The fixed locus studied in \cite{GraberPand} was for the moduli space of absolute stable maps $\Mbar_g(\P^1,d)$ and their version of the isomorphism (\ref{simple_fixed_locus_iso}) has different automorphisms. In this article we follow \cite{GraVakil_Loc} by using the moduli space of relative stable maps $\Mbar_g(\P^1,\part)$ which has labeled (i.e. ordered) fixed points in the pre-image of $\infty\in\P^1$. 
\end{remark}

\begin{re}[\textit{Normalisation of the Simple Locus along the Flag Nodes}] \label{normalisation_along_flag_nodes_stable_maps}
The simple fixed locus $\fixedSM$ of $\M$ has universal curve $\calC_{\fixedSM}$ which can be partially normalised along the universal flag nodes to give $\bm{n}_{\fixedSM} : \widetilde{\calC}_{\fixedSM}\rightarrow \calC_{\fixedSM}$. Moreover, $\widetilde{\calC}_{\fixedSM}$ can be decomposed into a disjoint union of closed sub-stacks:
\[
\widetilde{\calC}_{\fixedSM} = \calC_{\fixedSM,v} \sqcup \bigsqcup_{i=1}^l \calC_{\fixedSM,i}. 
\]
These have geometric points $(\zeta, s)$ such that $\zeta\in \fixedSM$ as described in \ref{simple_fixed_locus} and, using the same notation, if $(\zeta, s)\in \calC_{\fixedSM,v}$ then $s\in C_v$, otherwise if  $(\zeta, s)\in \calC_{\fixedSM,i}$ then $s\in C_i$.\\

Moreover, using the isomorphism (\ref{simple_fixed_locus_iso}) we have the following  cartesian diagrams:
\begin{align}  \label{normalisation_along_flag_nodes_stable_maps_diagrams}
\begin{array}{c}
\xymatrix{
\calC_{\fixedSM,v} \ar[d] \ar[r] & \calC_{\Mbar_{g,l}} \ar[d] \\
\fixedSM \ar[r] & \Mbar_{g,l}
}
\end{array}
\hspace{1cm}
\begin{array}{c}
\xymatrix{
\calC_{\fixedSM,i} \ar[d] \ar[r] & \calC_{\calP_{\part_{l}}}\ar[d] \ar[r] & \P^1 \ar[d] \\
\fixedSM \ar[r] &  \calP_{\part_{l}}\ar[r] & \bullet
}
\end{array}
\end{align}\\
\end{re}

\subsection{\texorpdfstring{$r$}{r}-Twisted Curves}

\begin{definition}[\textit{The Moduli Space of \texorpdfstring{$r$}{r}-Twisted Prestable Curves}]\label{r-twisted_curves_definition}

The moduli stack of $l$-marked genus $g$ $r$-twisted prestable curves is denoted by $\frakM^{r}_{g,l}$.  It is the groupoid  with:
\begin{enumerate}
\item Objects over a scheme $S$ given by:
\[
\xi = \left(~
\hspace{-0.5em}\smallfamily{C}{S}{\pi}\hspace{-0.5em}
,~~\left( \hspace{-0.5em} \begin{array}{c}\xymatrix@=1em{C   \\ S \ar[u]^{s_{i}} }\end{array} \right)_{\hspace{-0.25em}i\in\{1,\ldots,l\}}
~\right)
\]
where
\begin{enumerate}
\item $\pi$ is a proper flat morphism from a tame stack to a scheme,
\item each $s_i$ is a section of $\pi$ that maps to the smooth locus of $C$,
\item the fibres of $\pi$ are purely one dimensional with at worst nodal singularities,
\item the smooth locus $C^{\mathrm{sm}}$ is an algebraic space,
\item the coarse space $\overline{\pi}: \overline{C} \rightarrow S$ with sections $\overline{s}_i$ is a genus $g$, $l$-pointed prestable curve
	\[
	\big(~\overline{C},~~ \overline{\pi}:\overline{C} \rightarrow S,~~ (\,\overline{s}_i: S\rightarrow \overline{C})_{{i\in\{1,\ldots,l\}}} ~\big)
	\]
\item the local picture at the nodes is given by $\left[ U/ \mu_r\right] \rightarrow T$, where 
	\begin{itemize}
	\item $T = \Spec A$, $U = \Spec A[z,w]/(zw-t)$ for some $t\in A$, and the action of $\mu_r$ is given by $(z,w)\mapsto (\xi_r z, \xi_r^{-1} w)$.
	\end{itemize}
\end{enumerate}
\item Morphisms $\xi_1 \rightarrow \xi_2$ between two appropriately labelled objects are given by an equivalence class of cartesian diagrams
\[
\begin{array}{c}
\xymatrix@R=1.5em{
 C_1\ar[d]_{\pi_1}\ar[r]_{a} &  C_2\ar[d]^{\pi_2}\\
S_1 \ar[r]^{a'} & S_2
}\end{array}
\]
that are compatible with the other data (i.e. $a\circ s_{1,i} = s_{2,i}\circ a'$) and where equivalence is given by base-preserving natural transformations. \\
\end{enumerate}
\end{definition}

\begin{remark}
The fact that taking equivalence classes up to base-preserving natural transformations gives a well-defined $1$-category is due to \cite[Prop. 4.2.2]{AbraVist}. 
\end{remark}

\begin{theorem}[ \textit{Smoothness of $\frakM^{r}_{g,l}$   \cite{AbraJarvis, Twisted_and_Admissible, Olsson, Chiodo_StableTwisted}} ]
The stack $\frakM^{r}_{g,l}$ is a smooth proper Artin stack of dimension $3g-3+l$.\\
\end{theorem}

\begin{remark}
If the twisted curves parametrised in $\frakM^{r}_{g,l}$ are also required to have stable coarse space and if $3g-3+l\geq0$, the space becomes a Deligne-Mumford stack. \\
\end{remark}

\begin{re}[\textit{Forgetting Twisted Structure}]\label{forgetting_twisted_structure_intro_re}
There is a natural morphism of stacks 
\[
\frakM_{g,l}^r\longrightarrow \frakM_{g,l}
\]
which forgets the $r$-Twisted structure. The automorphism group of a $r$-twisted curve which fixes the underlying coarse curve is $ \mu_r^{\oplus n}$ where $n$ is number of nodes \cite[Prop. 7.1.1]{Twisted_and_Admissible}. So for $r\neq 1$ the moduli space $\frakM_{g,l}^r$ has stacky structure that $\frakM_{g,l}$ does not and the morphism $\frakM_{g,l}^r\rightarrow \frakM_{g,l}$ is  not an isomorphism. However, it is flat and surjective of degree 1.\\
\end{re}

\begin{re}[\textit{Universal Curves of $r$-Twisted Curves}] \label{universal_twisted_curve_frakMr_re}
Unlike $\frakM_{g,l}$, the stack $\frakM^r_{g,l}$ has two universal curves:
\begin{enumerate}
\item[(a)] A universal $r$-twisted curve  $\bm{\pi}_{\frakM^{r}}:\calC_{\frakM^{r}}\rightarrow \frakM^{r}_{g,l} $ which is of Deligne-Mumford type and where $\calC_{\frakM^{r}}$ parametrises  objects
\[
\xi = \left(~
\hspace{-0.5em}\raisebox{0.3 em}{$\begin{array}{c}\xymatrix@=1em{C  \ar[d]_{\pi} & ~\Theta\ar@{_{(}->}[l]_{\raisebox{0.25 em}{\mbox{\scriptsize$\vartheta$}}}\ar[dl]^{\gamma_\Theta} \\ S &  }\end{array}$} \hspace{-0.5em} 
,~~\left( \hspace{-0.5em} \begin{array}{c}\xymatrix@=1em{C   \\ S \ar[u]^{s_{i}} }\end{array} \right)_{\hspace{-0.25em}i\in\{1,\ldots,l\}}
~\right) 
\]
where $(\pi, s_i)$ is an object in $\frakM^r_{g,l}$, $\gamma_\Theta$ is an étale gerbe and $\vartheta$ is a closed sub-stack.

\item[(b)] A universal coarse curve  $\bm{\overline{\pi}}_{\frakM^{r}}:\overline{\calC}_{\frakM^{r}}\rightarrow \frakM^{r}_{g,l} $ which is representable and where $\overline{\calC}_{\frakM^{r}}$ parametrises objects $(\pi, s_i)$ in $\frakM^r_{g,l}$ as well as a section of the associated coarse morphism $\overline{\pi}$. \\
\end{enumerate}
\end{re}

\subsection{\texorpdfstring{$r$}{r}-Twisted Curves with Roots of Line Bundles}

\begin{definition}[\textit{The Moduli Space of $r$-Twisted Prestable Curves with Roots of Line Bundles}]\label{D_moduli_curves_with_bundles_definition}
Let $b\in r \Z$  and  denote by $\frakD^{{1/r,b}}_{g,l}$ the moduli stack which has
\begin{enumerate}
\item Objects over a scheme $S$ given by:
\[
		\xi= \Big(~
		\zeta
		,~~ F 
		,~~ L,
		~~ e: L^r \overset{\sim}{\rightarrow}F
		~\Big)
\]
where 
\begin{enumerate}
\item $\zeta = \big(\pi:C\rightarrow S ,\, s_i:S\rightarrow C\big)$ is a family over $S$ in $\frakM^{r}_{g,l}$, 
\item $F$ is a $\pi$-relative line bundle on $C$ of degree $b$,
\item $L$ is a $\pi$-relative line bundle on $C$ of degree $\frac{b}{r}$ and
\item $e: L^r \overset{\sim}{\longrightarrow} F$ is an isomorphism. 
\end{enumerate}
\item Morphisms $\xi_1 \rightarrow \xi_2$ between two appropriately labelled objects are given by triples 
\[
		 \Big(~
		a: \zeta_1 \rightarrow \zeta_2
		,~~ b':F_1\overset{\sim}{\rightarrow} a^*F_2, ~~ b:L_1\overset{\sim}{\rightarrow} a^*L_2
		~\Big)
\]
where $a$ is a morphism in $\frakM^{r}_{g,l}$, and where $b'$ and $b$ are isomorphisms compatible with the other data (i.e. if $\phi: a^*(L^r)\overset{\sim}{\rightarrow} (a^*L)^r $ is the canonical isomorphism then $(a^*e_2) \circ \phi \circ (b^{\otimes r})  = b'\circ e_1 $.)\\
\end{enumerate}
\end{definition}

\begin{re}[\textit{Partial-Normalisation at Separating Nodes \cite[\S2.3]{ChiodoZvonkine}}] \label{normalisation_separating_nodes}
Let $\xi\in \frakD^{{1/r,b}}_{g,l}$ be a geometric point  given by 
\[
		\xi= \Big(~
		C
		,~~ s_i
		,~~ F 
		,~~ L,
		~~ e: L^r \overset{\sim}{\rightarrow}F
		~\Big)
\]
such that $C$ has a connecting node and  $C_1 \sqcup C_2$ is the partial normalisation at that node. Also, let $\iota_i: C_i\hookrightarrow C$ be the inclusions and $\widehat{\gamma}_i: C_i \rightarrow \widehat{C}_i$ be the morphisms locally forgetting the $r$-orbifold structure at the connecting node. Lastly, let  $z_i\in \widehat{C}_i$ correspond to the pre-images of the nodes.  \\

Then for $i=1,2$ the sheaves $\widehat{\gamma}_{i*} \iota_i^*L$ and $\widehat{\gamma}_{i*} \iota_i^*F$ are locally free on $\widehat{C}_i$ and there is a unique pair  $b_1,b_2\in \{ 0, \ldots, r-1\}$ such that $b_1 + b_2 \equiv 0 ~(\mathrm{mod}\,r) $ with 
\[
\Big(\widehat{\gamma}_{1*} \iota_1^*L  \Big)^r \cong \widehat{\gamma}_{1*} \iota_1^* F(-b_1 z_1)
\hspace{0.7cm}
\mbox{and}
\hspace{0.7cm}
\Big(\widehat{\gamma}_{2*} \iota_2^*L  \Big)^r \cong \widehat{\gamma}_{2*} \iota_2^*F(-b_2 z_2).
\]\\
\end{re}

\vspace{-0.6cm}
\begin{remark}
The construction of \ref{normalisation_separating_nodes} also works on families with a separating node defined by an étale gerbe.\\
\end{remark}

\subsection{Stable Maps with Divisible Ramification}
\begin{definition}[\textit{Moduli Space of $r$-Stable Maps with Roots of Ramification}] \label{RSM_definition}
Recall $rm = 2g-2+ l(\part)+|\part|(1-2g_X)$ and consider the natural morphisms
\begin{enumerate}
\item $\Mbar_g(X,\part)\rightarrow \frakD^{1,rm}_{g,l}$ 
defined by taking a family of relative stable maps $\xi$ with $\pi:C\rightarrow S$ and $q_i:S\rightarrow C$ to $(\pi, q_i, \calR_{\xi})$. 
\item $\frakD^{_{1/r,rm}}_{g,l} \rightarrow \frakD^{_{1,rm}}_{g,l}$ defined by forgetting the $r$-twisted the $r$-th root structure.
\end{enumerate}
Using these morphisms, define the moduli stack $ \Mbar^r_g(\P^1,\part)$ by the following cartesian diagram:
\[
\xymatrix@R=1.5em{
 \Mbar^r_g(\P^1,\part) \ar[r] \ar[d]&  \Mbar_g(\P^1,\part)\ar[d] \\
\frakD^{\frac{1}{r},rm}_{g,l} \ar[r]& \frakD^{1,rm}_{g,l}}
\]
A geometric point in $\Mbar^{r}_g(\P^1,\mu)$ is of the form
\[
\xi = \Big(\, (C,q_i)
,~ (T,t)
,~ h: C \rightarrow T
,~ p :T \rightarrow \P^1
,~ L
,~ e:L^r\overset{\sim}{\rightarrow} R_f
~\Big)
\]
where $f:=g\circ h$ and  $R_f:= \omega^{\log}_{C} \otimes(p\circ h)^*(\omega_{\P^1}^{\log})^\vee$  and
\begin{enumerate}
\item $\Big((\overline{C},\overline{q}_i),~ (T,t) ,~ \overline{h}: \overline{C} \rightarrow T,~ p :T \rightarrow \P^1 \Big) \in \Mbar_g(\P^1,\mu)$,
\item $\Big((C,q_i),~R_f,~ L
,~ e:L^r\overset{\sim}{\rightarrow} R_f\Big) \in \frakD^{1/r,rm}_{g,l}$.\\
\end{enumerate}
\end{definition}

\begin{re}[\textit{Simplifying Notation for Main Spaces}]\label{simplifying_notation_main_in_text}
For the rest of this article we will use the following simplifying notation for the key spaces:
\begin{enumerate}
\item$\M := \Mbar_g (X,\part)$ with universal curve $\bm{\pi}: \calC \rightarrow \M$.
\item$\M^r := \Mbar^{r}_g (X,\part)$ with universal ($r$-twisted) curve $\bm{\pi}^r: \calC^r \rightarrow \M^r$.
\item$\frakM := \frakM_{g,l}$ and $\frakM^r := \frakM^{r}_{g,l}$.
\end{enumerate}

We will also denote by $\calR$ and $\bm{\delta}:\O_{\calC}\rightarrow \calR$ the pullback to $\calC$ of the corresponding bundle and morphism on $\calC_{\M}$.\\
\end{re}

\begin{definition}[\textit{Stacks Parametrising Sections of Sheaves}]\label{totalspace_definition}
Let $\calF$ be a line bundle on $\calC^r$. Then we have the following associated stacks which will be used in this article:
\begin{enumerate}

\item $\mathsf{Tot}\,\bm{\pi}^r_*\calF := \Spec_{\M^r}\, \Big(\Sym^\bullet \,R^1\bm{\pi}^r_*(\calF^\vee \otimes \omega_{\bm{\pi}^r} ) \Big)$ is the moduli stack with:
\begin{enumerate}
\item Objects $
					\Big(~
					\xi
					,~~
					\sigma: \O_{\calC_{\xi}} \longrightarrow \calF_{\xi} ~\Big)
$
 where $\xi$ is an object of $\M^r$ and $\calF_{\xi}$ is $\calF$ pulled back to $\calC_{\xi}$.
 \item Morphisms $(\xi_1 ,\sigma_1)\rightarrow (\xi_2 ,\sigma_2)$ are morphisms $a: \xi_1\rightarrow \xi_2$ such that if $\phi: a^* \calF_{\xi_2} \overset{\sim}{\rightarrow} \calF_{\xi_1}$ is the induced isomorphism then $a^*\sigma_2 = \phi \circ \sigma_1$. \\
\end{enumerate}

\item $\mathsf{Tot}\,\calF := \Spec_{ \calC^r}\, \Big(\Sym^\bullet \, \calF^\vee \Big)$
 is the moduli stack with:
 \begin{enumerate}
  \item Objects:
\[
					\Big(~
					\xi
					,~~\vartheta: \Theta \hookrightarrow C
					,~~\gamma_\Theta: \Theta \rightarrow S
					,~~
					\zeta: \O_{S} \longrightarrow {\gamma_{\Theta}}_*\vartheta^*\calF_{\xi} ~\Big)
\]
where $\xi$ is an object of $ \M^r$ over $S$ along with the extra data for an object in $\calC^r$ of a closed sub-stack $\vartheta: \Theta \hookrightarrow C$ and étale gerbe $\gamma_\Theta: \Theta \rightarrow S$ (see \ref{universal_twisted_curve_frakMr_re} for more details).
 \item Morphisms $(\xi_1, \vartheta_1, \gamma_{\Theta_1} ,\zeta_1)\rightarrow (\xi_2, \vartheta_2, \gamma_{\Theta_2} ,\zeta_2)$ are morphisms from $\calC^r$ $b:(\xi_1, \vartheta_1, \gamma_{\Theta_1}) \rightarrow (\xi_2, \vartheta_2, \gamma_{\Theta_2})$ such that if the induced isomorphism  is  $\psi: b^*({\gamma_{\Theta_2}}_*\vartheta_2^*\calF_{\xi_2}) \overset{\sim}{\rightarrow} {\gamma_{\Theta_1}}_*\vartheta_1^*\calF_{\xi_1}$ then $b^*\zeta_2 = \psi \circ \zeta_1$.\\
\end{enumerate}
\end{enumerate}
\end{definition}

\begin{remark}
Note that while we use the notation $\mathsf{Tot}\,\bm{\pi}_*\calF$, it is often the case that $\bm{\pi}_*\calF$ is not locally free. This space is called an \textit{abelian cone} in \cite{Intrinsic}. \\
\end{remark}

\begin{re}[\textit{$r$-th Power Maps over $\M^r$}]\label{rth_power_maps_on_M}
Natural examples of definition \ref{totalspace_definition} arise when the line-bundle $\calF$ is the universal $r$-root $\L$ and its $r$-th tensor power $\L^r$.\\

Moreover, in these cases, there are natural morphisms which arise from taking the $r$-th power of the section. 
We call these morphisms the \textit{$r$-th power maps} and denote the related morphism by the following commutative diagrams:
\begin{align*}
\begin{array}{c}
\xymatrix{
\mathsf{Tot}\,\bm{\pi}_*\L \ar[r]^{\bm{\tau}}\ar[d]^{\bm{\beta}} &\mathsf{Tot}\,\bm{\pi}_*\L^r \ar[d]^{\bm{\alpha}} \\
\M^r \ar@{=}[r]&\M^r
} 
\end{array}
\hspace{0.5cm}
\hspace{0.5cm}
\begin{array}{c}
\xymatrix{
\mathsf{Tot}\,\L \ar[r]^{\check{\bm{\tau}}}\ar[d]^{\check{\bm{\beta}}} &\mathsf{Tot}\,\L^r \ar[d]^{\check{\bm{\alpha}}} \\
\calC^r \ar@{=}[r]&\calC^r
} 
\end{array}\num \label{rth_power_map_commuting_diagrams} \\
\end{align*}
\end{re}

\begin{remark}
Comparing the two $r$-th power maps we see that that $\check{\bm{\tau}}$ is a map between total spaces of line bundles on $\calC^r$. In fact, it is an $r$-fold cover ramified at the zero section. However, in general, $\mathsf{Tot}\,\bm{\pi}_*\L$ is not the total space of a bundle and hence $\bm{\tau}$ is more complicated than $\check{\bm{\tau}}$. 
\end{remark}

\begin{definition}[\textit{Moduli Space of Stable Maps with Divisible Ramification}]\label{M1/r_main_definition}
The canonical ramification section $\bm{\delta} :\O_{\calC} \rightarrow \calR$  and the universal $r$th root $\bm{e}:\L^r\overset{\sim}{\rightarrow} \calR$ define a natural inclusion
\begin{align*}
\begin{array}{cccc}
\bm{i'} : & \M^r & \longrightarrow & \mathsf{Tot}\,\bm{\pi}_*\L^r \\
& \xi &\longmapsto & \big(~\xi, \,\bm{e}_\xi^{-1}(\bm{\delta}_\xi)\,\big).
\end{array}
\end{align*}
The moduli space of stable maps with divisible ramification  $\M^{1/r}$ (also denoted $\M^{_{1/r}}_g(\P^1, \part)$) is defined by following cartesian diagram which also defines  $\ForMapDRtoRSM$ and $\bm{i}$:
\begin{align}\label{M1/r_defining_diagram}
\begin{array}{c}
\xymatrix@R=1.5em{
		\M^{\frac{1}{r}} \ar[r]^{\bm{i}\hspace{0.5em}}\ar[d]_{\ForMapDRtoRSM} &\mathsf{Tot}\,\bm{\pi}_*\L \ar[d]^{\bm{\tau}}\\
		\M^{r}\ar[r]_{\bm{i'}\hspace{0.5em}} &\mathsf{Tot}\,\bm{\pi}_*\L^r
} 
\end{array}  
\end{align}\\
\end{definition}

\begin{re}[\textit{Universal Objects of $\M^{1/r}$}]
The universal objects of $\M^r$ pullback via the morphism $\bm{\nu} :\M^{1/r}\rightarrow \M^r$ to give universal objects on $\M^{1/r}$. For example we have a universal ($r$-twisted) curve $\bm{\pi}^{1/r}:\calC^{1/r} \rightarrow \M^{1/r}$. The universal $r$th root section $\bm{\sigma}$ pulls back from $\mathsf{Tot}\,\bm{\pi}_*\L$.\\
\end{re}

\begin{theorem}[\textit{Branch-Type morphism for $\M^{1/r}$} \cite{Leigh_Ram}]\label{branch_type_morphism_re}
There is a morphism of stacks
\[
\mathtt{br} :\M^{\frac{1}{r}}  \longrightarrow \Sym^{m} X 
\]
defined at each geometric point in $\xi \in \M^{1/r}$ with $C:=\calC^{1/r}_\xi$ and $f:= \bm{f}^{1/r}_{\xi}$ to be
\[
\mathrm{Div}\Big(Rf_*\big[ \O_{C} \overset{\bm{\sigma}_{\xi}}{\longrightarrow}  \L_{\xi} \big] \Big). 
\]
It commutes with the branch morphism of \cite{FantechiPand} via the diagram
\[
\xymatrix@R=1.5em{
\M^{\frac{1}{r}} \ar[r]^{\mathtt{br}\hspace{0.9em}}\ar[d] &\Sym^{m} X \ar[d]^{\Delta}\\
 \M \ar[r]^{\bm{br}\hspace{0.9em}} & \Sym^{rm} X  
}
\]
where $\Delta$ is defined by $\sum_i x_i \mapsto \sum_i rx_i$.\\
\end{theorem}

\section{\texorpdfstring{$\C^*$}{C*}-Action on Stable Maps with Divisible Ramification}\label{C*-section}

\notationConvention\\

\subsection{Natural \texorpdfstring{$\C^*$}{C*}-Actions and Equivariant Morphisms}

\begin{re}[\textit{$\C^*$-Action on $\M^r$}] \label{C*-action_on_M_def}
For $a\in \C^*$ there is an isomorphism $m_a: \P^1\rightarrow \P^1$ defining multiplication by $a$ and for any morphism $f:C\rightarrow \P^1$ we have $a\cdot f = m_a \circ f$. This gives rise to canonical isomorphisms
\[
f^*(\omega_{\P^1}^{\log})^\vee \overset{\sim}\longleftarrow f^*m_a^*(\omega_{\P^1}^{\log})^\vee \overset{\sim}\longrightarrow (a\cdot f)^*(\omega_{\P^1}^{\log})^\vee
\]
which (using the notation from \ref{Ram_bun_def}) give an isomorphism which we denote by
\[
\Phi_{a}: R_f \overset{\sim}\longrightarrow R_{(a\cdot f)}.
\]

We use this canonical morphism to define a natural $\C^*$-action on $\M^r$,  given by the morphism of stacks
$\C^*\times \M^r \longrightarrow  \M^r$ that maps a moduli point 
\[
 \left(\,
a
,\, \Big(\, (C,q_i)
,~ (T,t)
,~ h: C \rightarrow T
,~ p :T \rightarrow \P^1
,~ L
,~ e:L^r\overset{\sim}{\rightarrow} R_f
\,\Big)\right)
\]
to the moduli point
\[
\Big(~
 (C,q_i)
,~ (T,t)
,~ h
,~ a\cdot p
,~ L
,~ \Phi_a\circ e
~\Big).
\]
This can be defined on families in the same way as in \ref{C*-action_relative_stable_maps_def}. \\
\end{re}

\begin{remark}
This action has the property that the natural forgetful morphism $\M^r \rightarrow \M$ is $\C^*$-equivariant. \\
\end{remark}

\begin{re}[\textit{$\C^*$-Action on $\M^{\frac{1}{r}}$ and Related Spaces}]\label{action_on_M1/r_and_related_spaces}

Since the $\C^*$-action defined in \ref{C*-action_on_M_def} does not affect the bundle $\L$, it extends immediately to the the spaces $\M^{1/r}$,  $\mathsf{Tot}\,\bm{\pi}_*\L$, $\mathsf{Tot}\,\bm{\pi}_*\L^r$ and their universal curves by leaving the extra data unaffected. Moreover, we give the spaces $\frakM$, $\frakM^r$ and $\calT$ the trivial action.\\
 
The natural inclusions $\bm{i'}: \M^r\hookrightarrow \mathsf{Tot}\,\bm{\pi}_*\L^r$ and $\bm{i}: \M^{1/r}\hookrightarrow \mathsf{Tot}\,\bm{\pi}_*\L$ appearing in definition \ref{M1/r_main_definition} are equivariant since the the following diagram is commutative:
\[
\xymatrix@R=0.8em{
&R_f \ar[dd]^{\Phi_a}& \\
\O_C \ar[ru]^{\delta_f} \ar[rd]_{\delta_{(a\cdot f)}} & & L^r \ar[lu]_{e} \ar[ld]^{\Phi_a \circ e}\\
&R_{(a\cdot f)} &
}
\]\\
\end{re}

\subsection{Basic Properties of the Fixed Locus}

\begin{re}[\textit{The Simple and Non-Simple Fixed Loci}] \label{simple_non-simple_fixed_loci_of_M1/r}
By the universal property of the fixed locus the commuting diagram from theorem \ref{branch_type_morphism_re} restricts to the commuting diagram
\[
\xymatrix@R=1.5em{
(\M^{\frac{1}{r}})^{\C^*} \ar[r]^{\widetilde{\mathtt{br}}\hspace{0.75em}}\ar[d] &\Sym^{m} X \ar[d]^{\Delta}\\
\M^{\C^*} \ar[r]^{\widetilde{\bm{br}}\hspace{0.75em}}& \Sym^{rm} X  
}
\]
where $\widetilde{\mathtt{br}}$ and $\widetilde{\bm{br}}$ are the respective restrictions of $\mathtt{br}$ and $\bm{br}$ to the fixed loci.\\

We know from \ref{fixed_locus_SM} that the stack-theoretic image of $\M^{\C^*}$ under $\bm{br}$ is reduced and equal to a finite number of points. Since $\Delta$ is a closed immersion we must also have that the image of $(\M^{1/r})^{\C^*}$ under $\mathtt{br}$ is reduced and equal to a finite number of points. We can identify the possible points via $\mathrm{Im}(\widetilde{\bm{br}}) \cap \mathrm{Im}(\Delta)$ to be
\[
h_n:= [(m-n)\cdot(0) + k\cdot(\infty)] \in \Sym^m\P^1
\]
 for $n\in\{ 0,1,\ldots, m\}$, giving a decomposition 
 \[
(\M^{\frac{1}{r}})^{\C^*} = \bigsqcup_{n=0}^{m} \left( (\M^{\frac{1}{r}})^{\C^*} \cap \mathtt{br}^{-1}(h_n) \right).
\]

Following \ref{simple_fixed_locus} we split this into two cases:
\begin{enumerate}
\item \textit{The Simple Fixed Locus}: $\fixedDR := (\M^{\frac{1}{r}})^{\C^*} \cap \mathtt{br}^{-1}(h_0)$. This is the case where there is \textit{no} degeneration at $\infty \in \P^1$. 
\item \textit{The Non-Simple Fixed Loci}: $\nonsimpDR_n := (\M^{\frac{1}{r}})^{\C^*} \cap \mathtt{br}^{-1}(h_n)$ for $n\in\{ 1,\ldots, m\}$. This is the case where there \textit{is} degeneration at $\infty \in \P^1$. \\
\end{enumerate}
\end{re}

\begin{re}[\textit{The Simple Fixed Locus of \texorpdfstring{$\M^r$}{Mr}}] \label{simple_fixed_locus_of_M^r}
By the universal property of the fixed locus the forgetful morphism $\ForMapRSMtoSM: \M^r\rightarrow \M$ restricts to a morphism $\widetilde{\ForMapRSMtoSM}: (\M^r)^{\C^*}\rightarrow \M^{\C^*}$ on the fixed loci. Since this is surjective and the stack theoretic image of $(\M)^{\C^*}$ under $\bm{br}$ is reduced, we must have the following equality of stack theoretic images
\[
\mathrm{Im}(br\circ\ForMapRSMtoSM) = \mathrm{Im}(br) \subset \Sym^{rm}\, \P^1.
\] 
Hence we have a decomposition of $(\M^r)^{\C^*}$  similar to that described in \ref{simple_non-simple_fixed_loci_of_M1/r} and we define the \textit{simple fixed locus} $\fixedRSM\subset (\M^r)^{\C^*}$  to be
\[
\fixedRSM := (\M^r)^{\C^*} \cap (\mathtt{br}\circ\ForMapRSMtoSM) ^{-1} \big([ rm\cdot(0)] \big). 
\]\\
\end{re}

\vspace{-0.75cm}
\begin{lemma}[\textit{The Simple Fixed Loci as Pullbacks}] \label{simple_fixed_locus_of_M1/r_Pullbacks}
The simple fixed loci $\fixedRSM$ and $\fixedDR$ fit into the following diagram where both squares are cartesian and the vertical arrows are closed immersions:
\begin{align}
\begin{array}{c}
\xymatrix@R=1.5em{
\fixedDR\ar[r]^{\ForMapDRtoRSMfixed} \ar[d]^{\fixedIncDR}& \fixedRSM \ar[r]^{\ForMapRSMtoSMfixed}\ar[d]^{\fixedIncRSM}&  \fixedSM\ar[d]^{\fixedIncSM}\\
\M^{\frac{1}{r}}\ar[r]_{\ForMapDRtoRSM} &\M^r \ar[r]_{\ForMapRSMtoSM}& \M
} \label{simple_fixed_locus_of_M1/r_Pullbacks_diagram}
\end{array}
\end{align}
\end{lemma}
\begin{proof}
Following \ref{simple_non-simple_fixed_loci_of_M1/r} and \ref{simple_fixed_locus_of_M^r}, both $\fixedDR$ and $\fixedRSM$ can be constructed using pullbacks of $[rm\cdot(0)] \in\Sym\,\P^1$ under the compositions involving the branch morphism and forgetful morphisms. The pullback construction allows us to construct the desired diagram. For example we have the following commuting diagram 
\[
\xymatrix@R=1.5em{
\fixedRSM \ar@{-->}[r] \ar@/^0.75pc/[rr] \ar[d] & \fixedSM \ar[r] \ar[d]&\big\{[rm\cdot(0)]\big\} \ar[d] \\
(\M^r)^{\C^*} \ar[r] &\M^{\C^*} \ar[r] & \Sym\,\P^1
}
\]
where the right square is cartesian. The dashed arrow arises from the properties of the cartesian square. We have a similar diagram for $\fixedDR$. Hence we have a diagram of the desired form, but we must show that the squares are cartesian.  \\

If we define $\calX := \M^r \times_{\M} \fixedSM$, then we claim that $\fixedRSM \subseteq \calX$ is a sub-stack. To see this we consider use the rightmost square of (\ref{simple_fixed_locus_of_M1/r_Pullbacks_diagram}) and combine it with the cartesian square defining $\calX$:
\[
\xymatrix@R=1.5em{
\raisebox{0.1em}{$\fixedRSM$} \ar@{^(->}[rd]_{\fixedIncRSM} \ar@/^0.75pc/[rr]^{\ForMapRSMtoSMfixed} \ar@{-->}[r]&  \raisebox{0.1em}{$\calX$} \ar[r] \ar@{^(->}[d]& \raisebox{0.1em}{$\fixedSM$} \ar@{^(->}[d]^{\fixedIncSM} \\ 
  &\M^r \ar[r] & \M
}
\]
Hence we must have that $\fixedRSM \rightarrow \calX$ is an immersion. In a similar way we have that $\fixedDR$ is a sub-stack of  $\calY := \M^{1/r} \times_{\M^r} \fixedRSM$. \\

We now show that every family in $\calX$ is fixed under the $\C^*$ action. Recall the moduli space from definition  \ref{D_moduli_curves_with_bundles_definition} and define the notation $\frakD^n:= \frakD^{_{1/n,rm}}_{g,l}$ where $\C^*$ acts with the trivial action. Then from definition \ref{RSM_definition} we have a natural equivariant forgetful map $\M^r \longrightarrow \frakD^{r}$ which gives the equivariant isomorphism $\calX \cong \frakD^{r} \times_{\frakD^{1}} \fixedSM$. This shows that $\calX$ has trivial $\C^*$ action and hence $ \calX \subseteq (\M^r)^{\C^*}$. \\

We can similarly express $\calY$ as a cartesian product of spaces with trivial action by using the total-space stack over $\frakD^{r}$ defined by:
\[
\mathsf{Tot}_{\frakD}\,{\bm{\pi}_{\frakD}}_*\L_{\frakD} := \Spec_{\frakD}\, \Big(\Sym^\bullet \,R^1{\bm{\pi}_{\frakD}}_*(\L_{\frakD}^\vee \otimes \omega_{\bm{\pi}_{\frakD}} ) \Big)
\]
 which contains objects $
					\big(~
					\xi
					,~~
					\sigma: \O_{(\calC_{\frakD})_{\xi}} \longrightarrow (\L_{\frakD})_{\xi} ~\big)
$ where $\xi$ is an object of $\frakD^{r}$, $\bm{\pi}_{\frakD}: \calC_{\frakD} \rightarrow \frakD^{r}$ is the universal ($r$-twisted) curve and  $\L_{\frakD}$ is the universal $r$-th root. 
\end{proof}
\mbox{}

\vspace{-0.2cm}
\begin{lemma} \label{fixed_locus_etale_morphism}
The forgetful morphism $\ForMapDRtoRSMfixed: \fixedDR \rightarrow \fixedRSM$ from lemma \ref{simple_fixed_locus_of_M1/r_Pullbacks} is étale of degree $r^{l}$. 
\end{lemma}
\begin{proof}
We will show $\ForMapDRtoRSMfixed$ is étale in corollary \ref{cotangent_F1/r->F_is_zero}. We assume that $3g-3+l\geq 0$ since the special cases $(g,l)=(0,1)$ and $(0,2)$ are simpler and only require minor modifications. We calculate the degree at the fibre of a geometric point $\xi\in\fixedRSM$ determined by the data
\[
f: C \longrightarrow \P^1
,~ q_i 
,~ L
~~\mbox{and}~
~ e:L^r\overset{\sim}{\longrightarrow} R_f. 
\]
The pre-image is supported on finite collection determined by the $r$-th roots $\sigma$ of $\delta$. Denote by $C_v \subset C$ be the union of irreducible components of $C$ mapping to $0\in\P^1$ and denote by $C_i\cong \P^1[r]$ the irreducible components labelled by $q_i$ where $f$ is non-constant and the image of the $r$-orbifold point is $0$.\\

The map $f$ is constant on $C_v$ so we have $\delta|_{C_v} = 0$ and hence we must also have $\sigma|_{C_v} = 0$ as well. Thus the number of possible $r$-th roots $\sigma$ of $\delta$ is determined by the number of possibilities for $\sigma|_{C_i}$. \\

We choose coordinates on $C_i$ given by $C_i\setminus C_v \cong \Spec\,\C[y]$ and at the $r$-orbifold point by  $V:=\left[ (\Spec\,\C[z])/ \mu_r\right]$ where the action of $\mu_r$ is given by $z\mapsto \xi_r z$. Then $(R_f)|_{C_i}$ has local $\mu_r$-equivariant generator $z^{-r\part_i}$ at the orbifold point and trivial generator on $C_i\setminus C_v$. The restriction $\delta|_{C_i}$ is then given on these open sets by $z^{r\part_i}\cdot(z^{-r\part_i})$ and $1$ respectively. \\

The $r$-root bundle $L$ is similarly given by local generators $z^{-\part_i}$ and $1$, while the $r$-root section $\sigma|_{C_i}$ is determined by polynomials $\zeta_1(z)\in\C[z]$ and $\zeta_2(y)\in\C[y]$ compatible with change of coordinates such that $\zeta_1^r = z^{r\part_i}$ and $\zeta_2^r = 1$. There are exactly $r$ different choices for these, determined by the roots of unity. Hence we have a total of $r^l$ different choices for $\sigma$. \\

We must now examine the automorphisms and isomorphisms of objects in $\fixedDR$. We have automorphisms in $\fixedDR$ and $\fixedRSM$ arising from $\overline{C}$ and $f$, however these are unaffected by the morphism $\ForMapDRtoRSMfixed$. The other automorphisms of $\fixedRSM$ form a subgroup $\mu_r^{_{\oplus(l+1)}}$ arising from the isomorphism $L \overset{\sim}{\rightarrow} L$ defined by $r$-roots of unity and the $r$-orbifold nodes (discussed more in \ref{forgetting_twisted_structure_intro_re}). \\

The automorphisms of $\xi$ arising from $L\overset{\sim}{\rightarrow}L$ and the nodes where $r\nmid \part_i$ do not correspond to automorphisms of  $\xi'\in\ForMapDRtoRSMfixed^{-1}(\xi)$, but rather isomorphisms of objects. However, the automorphisms of $\xi$ arising from nodes where $r|\part_i$ correspond to the automorphisms of (not arising from $\overline{C}$ or $f$). Hence, $|\ForMapDRtoRSMfixed^{-1}(\xi)| = r^{l-1-\sum_{i: r|\part_i}1}$ and for $\xi'\in\ForMapDRtoRSMfixed^{-1}(\xi)$ we have $|\mathrm{Aut}\, \xi| /|\mathrm{Aut}\, \xi'| = r^{1+\sum_{i: r|\part_i}1}$. Thus the degree is $r^l$ as desired. 
\end{proof}
\mbox{}

\vspace{0.5cm}

\subsection{Flag Nodes and Partial Normalisation}

\begin{re}[\textit{Étale Gerbes for the Flag Nodes in the Simple Locus}]\label{etale_gerbes_of_simple_locus_DR}
Consider the case $3g-3+l\geq 0$. Let $\bm{\pi}_{\fixedSM}:\calC_{\fixedSM}\rightarrow \fixedSM$ and $\bm{\pi}_{\fixedDR}:\calC_{\fixedDR}\rightarrow \fixedDR$ be the universal ($r$-twisted) curves. $\bm{\pi}_{\fixedSM}$ is representable and the flag nodes of $\fixedSM$ define $l$ different sections $\fixedSM \rightarrow \calC_{\fixedSM}$ of $\bm{\pi}_{\fixedSM}$ (in the case $(g,l)=(0,2)$ there is only one flag node). These pullback via the map forgetful map $\calC_{\fixedDR}\rightarrow \calC_{\fixedSM}$ to define sub-stacks at the flag nodes. There are also corresponding sections of the universal coarse curve $\overline{\bm{\pi}}_{\fixedDR}:\overline{\calC}_{\fixedDR}\rightarrow \fixedDR$. This is summarised in the following commuting diagram where the square is cartesian and $\bm{\gamma}_{\Theta_i}$ is an étale gerbe:
\begin{align*}
\xymatrix{
\Theta_i ~\ar[d]_{\bm{\gamma}_{\Theta_i}}\ar@{^(->}[r]_{\bm{\vartheta}_i}& \calC_{\fixedDR}\ar[d]^{\bm{\gamma}_{\fixedDR}} \\ 
\fixedDR~\ar@{^(->}[r]^{\bm{z_i}} \ar@{=}[rd]&\overline{\calC}_{\fixedDR} \ar[d]^{\overline{\bm{\pi}}_{\fixedDR}}\\ 
& \fixedDR
}
\end{align*}\\ 
\end{re}

\begin{re}[\textit{Normalisation along the Flag Nodes of the Simple Locus}]\label{normalisation_along_flag_nodes_DR}
The simple fixed locus $\fixedDR$ of $\M^{1/r}$ has universal curve $\calC_{\fixedDR}$ which can be partially normalised along the étale gerbes from \ref{etale_gerbes_of_simple_locus_DR} to give $\bm{n}: \widetilde{\calC} \rightarrow \calC_{\fixedDR}$. This is the pullback via the forgetful map $\calC_{\fixedDR}\rightarrow \calC_{\fixedSM}$  of the partial normalisation $\widetilde{\calC}_{\fixedSM}$ of $\calC_{\fixedSM}$ from \ref{normalisation_along_flag_nodes_stable_maps}.\\

$\widetilde{\calC}$ can also be decomposed into a disjoint union of closed sub-stacks
$\widetilde{\calC} = \calC_{v} \sqcup \bigsqcup_{i=1}^l \calC_{i}
$ 
where $\calC_{v}$ and $\calC_{i}$ are defined from \ref{normalisation_along_flag_nodes_stable_maps} via the following cartesian diagrams:
\begin{align} \label{normalisation_along_flag_nodes_DR_diagrams_defining_components}
\begin{array}{c}
\xymatrix{
\calC_{v_{\mbox{}}} \ar@{^(->}[d]_{\bm{\iota}_{v}} \ar[r] & \calC_{\fixedSM,v_{\mbox{}}} \ar@{^(->}[d] \\
\calC_{\fixedDR} \ar[r] & \calC_{\fixedSM}
}
\end{array}
\hspace{1cm}
\begin{array}{c}
\xymatrix{
\calC_{i_{\mbox{}}} \ar@{^(->}[d]_{\bm{\iota}_{i}} \ar[r] & \calC_{\fixedSM,i_{\mbox{}}} \ar@{^(->}[d] \\
\calC_{\fixedDR} \ar[r] & \calC_{\fixedSM}
}
\end{array}
\end{align}

Moreover, combining the diagrams of (\ref{normalisation_along_flag_nodes_DR_diagrams_defining_components}) with the diagrams (\ref{normalisation_along_flag_nodes_stable_maps_diagrams}) from \ref{normalisation_along_flag_nodes_stable_maps} we have the following diagrams where the squares are cartesian:
\begin{align} \label{normalisation_along_flag_nodes_DR_diagrams_relating_to_SM}
\begin{array}{c}
\xymatrix@C=1.8em{
\calC_{v}\ar[r]^{\bm{\gamma}_v} \ar[rd]_{\bm{\pi}_v} &\overline{\calC}_{v} \ar[d]^{\overline{\bm{\pi}}_v} \ar[r] &\calC_{\fixedSM,v} \ar[d]^{\bm{\pi}_{\fixedSM,v}} \ar[r] & \calC_{\Mbar_{g,l}} \ar[d] \\
&\fixedDR \ar[r] &\fixedSM \ar[r] & \Mbar_{g,l}
}
\end{array}
\hspace{0.25cm}
\begin{array}{c}
\xymatrix@C=1.8em{
\calC_{i}\ar[r]^{\bm{\gamma}_{i}} \ar[rd]_{\bm{\pi}_{i}} &\overline{\calC}_{i} \ar[d]^{\overline{\bm{\pi}}_{i}} \ar[r] &\calC_{\fixedSM,i} \ar[d]^{\bm{\pi}_{\fixedSM,i}} \ar[r] & \calC_{\calP_{\part_i}} \ar[d]^{\bm{\pi}_{\calP_{\part_i}}} \ar[r]& \P^1\ar[d]\\
&\fixedDR \ar[r] &\fixedSM \ar[r] & {\calP_{\part_i}} \ar[r]& \bullet
}
\end{array}
\end{align}
Here the marked points of $\Mbar_{g,l}$ and $0\in \P^1$ correspond to the images of $\bm{z}_i$.\\
\end{re}

\begin{re}[\textit{Universal Bundles on the Simple Fixed Locus}] \label{universal_bundles_on_fixed_locus}
Denote by $\bm{q}_i:\fixedDR\rightarrow \calC_{\fixedDR}$, the universal sections of $\bm{\pi}_{\fixedDR}:\calC_{\fixedDR}$ which correspond to the marked points  from the relative stable map condition (see \ref{universal_objects_of_moduli_relative_stable_maps} for more details).  Recall from \ref{simple_fixed_locus} and \ref{simple_non-simple_fixed_loci_of_M1/r} that the stable maps in the simple fixed locus $\fixedDR$ are not degenerated at infinity so we have a universal relative stable map $\bm{f}_{\fixedDR}: \calC_{\fixedDR}\rightarrow \P^1$. Moreover, we have isomorphisms
\[
\calR
\cong 
\omega^{\mathrm{log}}_{\bm{\pi}_{\fixedDR}}\otimes \bm{f}_{\fixedDR}^* \big(\omega^{\mathrm{log}}_{\P^1}\big)^\vee
\cong 
\omega_{\bm{\pi}_{\fixedDR}}\otimes \O_{\calC_{\fixedDR}}\big(\mbox{$\sum_i$} (\part_i+1) \bm{q}_i \big)
\]
where the last isomorphism has used $\omega_{\P^1} \cong \O_{\P^1}\big(-2(\infty)\big)$. \\

Letting $\overline{\bm{f}}_{\fixedDR}: \overline{\calC}_{\fixedDR}\rightarrow \P^1$ be the universal relative stable map from the universal coarse curve such that $ \bm{f}_{\fixedDR} = \overline{\bm{f}}_{\fixedDR}\circ \bm{\gamma}_{\fixedDR}$. Then we have 
\[
\calR
\cong 
 \bm{\gamma}_{\fixedDR}^* \Big(\omega_{\overline{\bm{\pi}}_{\fixedDR}}\otimes \O_{\overline{\calC}_{\fixedDR}}\big(\mbox{$\sum_i$} (\part_i+1) \overline{\bm{q}}_i \big) \Big).
\]
For $3g-3+l\geq 0$, restricting $\calR$ to the normalisation components gives line bundles
\[
\calR_v := \bm{\iota}_v^* \calR 
\cong 
 \bm{\gamma}_{v}^* \omega_{\overline{\bm{\pi}}_{v}}\big(\mbox{$\sum_i$} \bm{z}_i \big) 
 \hspace{0.5cm}
\mbox{and}
  \hspace{0.5cm}
\calR_i :=  \bm{\iota}_i^* \calR
\cong 
 \bm{\gamma}_{i}^* \O_{\overline{\calC}_{i}}\big(\part_i \bm{z}_i \big).
\]
Here we have used the right diagram of (\ref{normalisation_along_flag_nodes_DR_diagrams_relating_to_SM}) which gives that $\omega_{\overline{\bm{\pi}}_{i}}$ is the pullback of $\omega_{\P^1} \cong \O_{\P^1}\big(-2(0)\big)$. Restricting $\L$, the universal $r$-th root bundle on $\fixedDR$, to the components of the normalisation gives line bundles
\[
\L_v := \bm{\iota}_v^*\L
 \hspace{0.5cm}
\mbox{and}
  \hspace{0.5cm}
\L_i := \bm{\iota}_i^*\L.
\]\\

\end{re}

\begin{lemma} \label{direct_image_of_line_bundles_at_flag_nodes}
Let $\L$ be the universal $r$-th root bundle on $\fixedDR$ and denote the Chern polynomial by $c_s(-)$. Then we have
\begin{enumerate}
\item $c_s\Big(R(\bm{\gamma}_{\Theta_i})_* \bm{\vartheta}_i^* \L^r \Big) = 1$ and 
\item $
c_s\Big( R(\bm{\gamma}_{\Theta_i})_* \bm{\vartheta}_i^* \L \Big)  = \left\{\begin{array}{ll}
1 & \mbox{if $r|\part_i$;} \\
0 & \mbox{otherwise.} \\
\end{array}\right.$
\end{enumerate}
Moreover, $R(\bm{\gamma}_{\Theta_i})_* \bm{\vartheta}_i^* \L \cong0$ if $r\nmid\part_i$.
\end{lemma}
\begin{proof}
Using \ref{universal_bundles_on_fixed_locus} and the properties that $ \bm{\gamma}_{\fixedDR}$ is flat and $ {\bm{\gamma}_{\fixedDR}}_*$ is exact gives
\[
R(\bm{\gamma}_{\Theta_i})_* \bm{\vartheta}_i^* \L^r 
\cong 
\bm{z}_i^* R{\bm{\gamma}_{\fixedDR}}_*  \bm{\gamma}_{\fixedDR}^* \Big(\omega_{\overline{\bm{\pi}}_{\fixedDR}}\otimes \O_{\overline{\calC}_{\fixedDR}}\big(\mbox{$\sum_i$} \part_i \overline{\bm{q}}_i \big) \Big)
\cong
\bm{z}_i^* \Big(\omega_{\overline{\bm{\pi}}_{\fixedDR}}\otimes \O_{\overline{\calC}_{\fixedDR}}\big(\mbox{$\sum_i$} \part_i \overline{\bm{q}}_i \big) \Big).
\]
The first result now follows from $\bm{z}_i^*\omega_{\overline{\bm{\pi}}_{\fixedDR}}\cong \O_{\fixedDR}$ and $\bm{z}_i^*\O_{\overline{\calC}_{\fixedDR}}\big(\mbox{$\sum_i$} \part_i \overline{\bm{q}}_i \big) \cong \O_{\fixedDR}$. \\

For the second result we consider the case where $r|\part_i$. Then locally at the flag node there exists a line bundle $\mathtt{L}_i$ on $\fixedDR$ such that $  \bm{\vartheta}_i^* \L \cong \bm{\gamma}_{\Theta_i}^* \mathtt{L}_i$. Then, 
$R(\bm{\gamma}_{\Theta_i})_* \bm{\vartheta}_i^* \L \cong \mathtt{L}_i$ and we also have that $\mathtt{L}_i^r \cong  R(\bm{\gamma}_{\Theta_i})_* \bm{\vartheta}_i^* \L^r  \cong \O_{\fixedDR}$. The result for this case follows from basic properties of the Chern polynomial. \\

For the case where $r\nmid\part_i$ we show that $R(\bm{\gamma}_{\Theta_i})_* \bm{\vartheta}_i^* \L \cong 0$ by showing this at every geometric point of $\fixedDR$. Indeed, in this case, for a geometric point $\xi\in \fixedDR$ we have $(\bm{\vartheta}_i^* \L)_{\xi}$ is a trivial line bundle on $(\Theta_i)_\xi \cong \calB \mu_r$ which has non-trivial weight for the $\mu_r$-action. Hence,  there are no invariants and $(R(\bm{\gamma}_{\Theta_i})_* \bm{\vartheta}_i^* \L)_\xi \cong 0$. 
\end{proof}
\mbox{}

\subsection{Forgetting \texorpdfstring{$r$}{r}-Orbifold Structure Flag Nodes}

\begin{re}[\textit{Partial-Normalisation at the Flag Nodes}] \label{normalisation_flag_nodes}

We can apply the concept from \ref{normalisation_separating_nodes} to the flag nodes of the simple fixed loci $\fixedRSM$ and $\fixedDR$. For example, consider a point in $\fixedRSM$ determined by the data
\[
f: C \rightarrow \P^1
,~ q_i 
,~ L
~~\mbox{and}~
~ e:L^r\overset{\sim}{\rightarrow} \omega^{\log}_{C} \otimes f^*(\omega_{\P^1}^{\log})^\vee. 
\]
We let $\iota: B \hookrightarrow C$ a sub-curve with a morphism $\widehat{\gamma}:B\rightarrow \widehat{B}$ locally forgetting the $r$-orbifold structure at any points corresponding to flag  nodes. In the cases $(g,l)=(0,2)$ and $3g-3+l\geq 0$ there are two cases to consider:
\begin{enumerate}
\item If $B$ is an irreducible component where $f|_B$ is non-constant then
\[
\widehat{\gamma}_*\iota^*\Big( \omega^{\log}_{C} \otimes f^*(\omega_{\P^1}^{\log})^\vee \Big) 
\cong 
\O_{\P^1}(\part_i) 
\cong \O_{\P^1}\left( \left\lfloor \frac{a}{r} \right\rfloor r + \left< \frac{a}{r} \right>\right). 
\]
and the $r$-th root becomes
\[
\Big(\widehat{\gamma}_* \iota^*L  \Big)^r \cong \O_{\P^1}\left( \left\lfloor \frac{a}{r} \right\rfloor r\right).
\]
\item If $B = f^{-1}(0)$ and if $p_1,\ldots,p_l \in \widehat{B}$ correspond to flag  nodes then
\[
\widehat{\gamma}_*\iota^*\Big( \omega^{\log}_{C} \otimes f^*(\omega_{\P^1}^{\log})^\vee \Big) 
\cong 
\omega_{\widehat{B}}\Big(\mbox{$\sum\limits_{i}$} \, p_i  \Big).
\]
and the $r$-th root becomes
\[
\Big(\widehat{\gamma}_* \iota^*L  \Big)^r
\cong 
\omega_{\widehat{B}}\Big(\mbox{$\sum\limits_{i}$}\big(1+ \left< \part_i/r \right>-r\big)p_i + \mbox{$\sum\limits_{r\mid \part_i}$}r p_i \Big).
\]\\
\end{enumerate}
\end{re}

\vspace{-0.4cm}
\begin{remark}
The discussion in \ref{normalisation_flag_nodes} also holds for families in $\fixedRSM$ and $\fixedDR$ (as was the case for \ref{normalisation_separating_nodes}) but we have used a geometric point to simplify the exposition. \\
\end{remark}

\vspace{0.01cm}
\begin{re}[\textit{Reverse Clutching-Type Morphism at the Flag Nodes}] \label{reverse_clutching_morphism}
Let $3g-3+l\geq 0$ and define the vector of \textit{reverse remainders} to be $\mathtt{a}:= \mathtt{a}(g,\part):= (a_1,\ldots, a_l)$ where $a_i\in \{0,\ldots, r-1\}$ is defined by
\[
\part_i = \left\lfloor \frac{\part_i}{r} \right\rfloor r + (r-1 -a_i). 
\]
For $l=l(\part)$ we denote by $\Mbar_{g,l}^{1/r, \mathtt{a}}$, the moduli space of $r$-spin curves twisted by $\mathtt{a}$ as defined in \cite{Chiodo_StableTwisted}. That is, the moduli space of $r$-stable curves with $r$-th roots of $\omega_C(-\sum a_i q_i)$. \\

Moreover, we define $\calP^{1/r}_{d}$ to be the moduli space parametrising pairs which consist of $\C^*$-fixed Galois covers of $\P^1$ of degree $d$ along with an $r$-th root of $\O_{\P^1}(\left\lfloor \frac{d}{r} \right\rfloor r )$.\\

Then we have a natural degree 1 morphism considered in \cite[\S2.3]{ChiodoZvonkine}:
\[
\fixedRSM \longrightarrow   \Mbar_{g,l}^{\frac{1}{r}, \mathtt{a}}\times \calP^{\frac{1}{r}}_{\part_1}\times \cdots \times \calP^{\frac{1}{r}}_{\part_l} 
\]
which is defined in the following way:
\begin{enumerate}
\item Locally normalise the curve at the flag nodes and locally forget the stack structure there. 
\item The map $f$ is taken to the (local) coarse maps associate to the restrictions. (The map is forgotten on the component where $f$ was constant.)
\item The $r$-th root $L$ on $C$ is taken to its pullbacks on the normalised components and then locally taking the invariant sections at the pre-images of the nodes. As described in \ref{normalisation_flag_nodes}. 
\end{enumerate}
Note that this defines a map to $\Mbar_{g,l}^{{1/r, \mathtt{b}}}\times \calP^{{1/r}}_{\part_1}\times \cdots \times \calP^{{1/r}}_{\part_l} $
where $\mathtt{b}=(b_1,\ldots,b_l)$ is defined by 
\[
b_i =
\left\{\begin{array}{ll}
a_i, & \mbox{if $r\nmid\part_i$}, \\
a_i- r = -1,& \mbox{if $r\mid\part_i$}.
 \end{array}
 \right.
\]
However, we can compose with the natural isomorphism $\Mbar_{g,l}^{{1/r, \mathtt{b}}} \overset{\sim}{\longrightarrow} \Mbar_{g,l}^{{1/r, \mathtt{a}}}$ which is defined via the isomorphism $L_{\mathtt{b}}\cong L_{\mathtt{a}} \big(  \sum_{r\mid \part_i} p_i \big)$.\\
\end{re}

\begin{corollary}[\textit{Degree of Morphism to $ \Mbar_{g,l}^{{1/r, \mathtt{a}}}$ From Theorem \ref{VL_theorem}}] \label{degree_of_map_to_spin_curves}
Consider the morphism  $\fixedRSM \rightarrow \Mbar_{g,l}^{_{1/r, \mathtt{a}}}\times \calP^{_{1/r}}_{\part_1}\times \cdots \times \calP^{_{1/r}}_{\part_l} $ from \ref{reverse_clutching_morphism}. The morphism $\bm{b}:\fixedDR\rightarrow \Mbar_{g,l}^{_{1/r, \mathtt{a}}}$ defined by the following composition has degree $(\part_1\cdots\part_l)^{-1}$:
\begin{align*}
\begin{tikzpicture}[baseline=(current  bounding  box.center), node distance=3.5cm, auto,
  f->/.style={->,preaction={draw=white, -,line width=3pt}},
  d/.style={double distance=1pt},
  Dash/.style={dashed,->},
  fd/.style={double distance=1pt,preaction={draw=white, -,line width=3pt}}]
  \node (01) {$\fixedDR $};
  \node [right of=01,node distance=2.5cm] (11){};
  \node [right of=11] (12){};
  \node [right of=12] (13){$  \Mbar_{g,l}^{\frac{1}{r}, \mathtt{a}}$};
  \node [below of=11, node distance=1.5cm] (21) {$\fixedRSM$};
  \node [right of=21] (22) {$\Mbar_{g,l}^{\frac{1}{r}, \mathtt{a}}\times \calP^{\frac{1}{r}}_{\part_1}\times \cdots \times \calP^{\frac{1}{r}}_{\part_l} $};
  \draw[f->]  (01) -> (13) node[pos=0.5, below]{$\bm{b}$};  
  \draw[f->]  (21) -> (22) node[pos=0.4, above]{}; 
  \draw[f->] (01) -> (21) node[pos=0.5, align=right, left]{$\ForMapDRtoRSMfixed~~$};  \draw[f->] (22) -> (13) node[pos=0.6, below]{~~$\mathrm{pr}_1$}; 
\end{tikzpicture} 
\end{align*}
\end{corollary}
\begin{proof}
This follows immediately from \ref{reverse_clutching_morphism},  lemma \ref{fixed_locus_etale_morphism} and because the projection map $\mathrm{pr}_1$ has degree $r^{-l}(\part_1\cdots\part_l)^{-1}$. 
\end{proof}
\mbox{}

\begin{lemma}\label{sheaves_on_fixed_locus_pullbacks}
Let $\mathtt{a}$ be the vector of reverse remainders from \ref{reverse_clutching_morphism} and $\bm{b}$ the morphism from corollary \ref{degree_of_map_to_spin_curves}. Also, let $\bm{\rho}$ be the universal ($r$-twisted) curve of $ \Mbar_{g,l}^{_{1/r, \mathtt{a}}}$ and $\L$ be the universal $r$-root.
\begin{enumerate}[label={(\roman*)}]
\item \label{dualizing_sheaves_iso} There are isomorphisms (using the notation from lemma \ref{simple_fixed_locus_of_M1/r_Pullbacks} and (\ref{normalisation_along_flag_nodes_DR_diagrams_relating_to_SM})):
\[
{\overline{\bm{\pi}}_{v}}_*\omega_{\overline{\bm{\pi}}_{v}}
\cong
(\ForMapRSMtoSMfixed\circ\ForMapDRtoRSMfixed)^* {\bm{\pi}_{\fixedSM,v}}_*\omega_{\bm{\pi}_{\fixedSM,v}} \cong \bm{b}^*\bm{\rho}_* \omega_{\bm{\rho}}.
\]
\item \label{dist_triangle_forgetting_log_L^r} If $\L_v$ is the restriction of the $r$-th root bundle on $\calC_{\fixedDR}$ to $\calC_v$ then there is a distinguished triangle
\[
\xymatrix{
\bm{b}^* R\bm{\rho}_* \omega_{\bm{\rho}} \ar[r] & R{\bm{\pi}_v}_* \L_v^r \ar[r] & \bigoplus\limits_{i} \O_{\fixedDR} \ar[r] & \bm{b}^* R\bm{\rho}_* \L[1].
}
\]
\item \label{dist_triangle_forgetting_log_L} Using the notation from part \ref{dist_triangle_forgetting_log_L^r}, there is a distinguished triangle
\[
\xymatrix{
\bm{b}^* R\bm{\rho}_* \L \ar[r] & R{\bm{\pi}_v}_* \L_v \ar[r] & \bigoplus\limits_{i:\, r|\part_i} \O_{\fixedDR} \ar[r] & \bm{b}^* R\bm{\rho}_* \L[1].
}
\]
\end{enumerate}
\end{lemma}
\begin{proof}
Begin by considering the following cartesian diagrams where $\overline{\bm{\rho}}$ is the universal coarse curve of $\Mbar_{g,l}^{_{1/r, \mathtt{a}}}$ and $\widehat{\bm{\pi}}_{v}$ is defined from $\bm{\pi}_{v}$ after forgetting the $r$-twisted structure at the flag nodes:  
\[
\xymatrix{
\overline{\calC}_{v} \ar[d]_{\overline{\bm{\pi}}_{v}}\ar[r] & \overline {\calC}_{\Mbar_{g,l}^{\frac{1}{r}, \mathtt{a}}}  \ar[d]^{\overline{\bm{\rho}}}  \\
\fixedDR \ar[r]^{\bm{b}\hspace{1em}}&   \Mbar_{g,l}^{\frac{1}{r}, \mathtt{a}} \\
}
\hspace{1cm}
\xymatrix{
\widehat{\calC}_{v} \ar[d]_{\widehat{\bm{\pi}}_{v}}\ar[r]_{\widehat{\bm{b}}\hspace{1em}} & {\calC}_{\Mbar_{g,l}^{\frac{1}{r}, \mathtt{a}}}  \ar[d]^{{\bm{\rho}}}  \\
\fixedDR \ar[r]^{\bm{b}\hspace{1em}}&   \Mbar_{g,l}^{\frac{1}{r}, \mathtt{a}} \\
}
\]
The left diagram gives the following isomorphism ${\overline{\bm{\pi}}_{v}}_*\omega_{\overline{\bm{\pi}}_{v}}
\cong
\bm{b}^*\overline{\bm{\rho}}_* \omega_{\overline{\bm{\rho}}}$. Moreover, since there is no $r$-twisted structure at smooth points we have the further isomorphism $\bm{b}^*\overline{\bm{\rho}}_* \omega_{\overline{\bm{\rho}}} \cong 
 \bm{b}^*\bm{\rho}_* \omega_{\bm{\rho}} $. The proof of part \ref{dualizing_sheaves_iso} is then competed by observing that the isomorphism ${\overline{\bm{\pi}}_{v}}_*\omega_{\overline{\bm{\pi}}_{v}}
\cong
(\ForMapRSMtoSMfixed\circ\ForMapDRtoRSMfixed)^* {\bm{\pi}_{\fixedSM,v}}_*\omega_{\bm{\pi}_{\fixedSM,v}}$ arises from the middle square of the left diagram from (\ref{normalisation_along_flag_nodes_DR_diagrams_relating_to_SM}).\\

For part \ref{dist_triangle_forgetting_log_L^r} the discussion in \ref{universal_bundles_on_fixed_locus} shows that ${\bm{\gamma}_{v}}_* \L_v^r \cong \omega_{\overline{\bm{\pi}}_v}^{\mathrm{log}}$ where the $\mathrm{log}$ superscript refers to the (un-twisted) markings corresponding to the flag nodes. Since $R^j{\overline{\bm{\pi}}_v}_*\omega_{\overline{\bm{\pi}}_v}^{\mathrm{log}}$ vanishes for $j\neq 0$ the result for part \ref{dist_triangle_forgetting_log_L^r} follows from part \ref{dualizing_sheaves_iso}.\\ 

For part \ref{dist_triangle_forgetting_log_L} we consider the right cartesian diagram given above. Let $\widehat{\bm{\gamma}}: \calC_v \rightarrow \widehat{\calC}_v$ be the universal morphism forgetting the $r$-twisted structure at the flag nodes. From the discussion in \ref{normalisation_flag_nodes} and at the end of \ref{reverse_clutching_morphism} we have that there is an isomorphism 
\[
\widehat{ \bm{\gamma}}_* \L_v  
\cong 
\widehat{\bm{b}}^* \L \otimes \O_{\calC_v}\big(\mbox{$\sum_{i; r\mid \part_i}$}r \bm{z}_i \Big).
\]
The result for part \ref{dist_triangle_forgetting_log_L}  now follows immediately. 
\end{proof}
\mbox{}

\vspace{0.2cm}
\begin{lemma}\label{edge_bundles_trivial_without_linearisations}
The direct images $R{\bm{\pi}_{i}}_*\L_i$ and $R{\bm{\pi}_{i}}_*\L_i^r$ are trivial bundles.
\end{lemma}
\begin{proof}
Let $\fixedRSM \longrightarrow   \Mbar_{g,l}^{_{1/r, \mathtt{a}}}\times \calP^{_{1/r}}_{\part_1}\times \cdots \times \calP^{_{1/r}}_{\part_l} $ be the \textit{reverse clutching morphism} from \ref{reverse_clutching_morphism} and define $\bm{d}: \fixedDR\rightarrow \calP^{_{1/r}}_{\part_i}$ to be the following composition
\[
\xymatrix@C=3.5em{
\fixedDR \ar[r]^{\ForMapDRtoRSM\hspace{0.25em}}& \fixedRSM \ar[r] & \Mbar_{g,l}^{\frac{1}{r}, \mathtt{a}}\times \calP^{\frac{1}{r}}_{\part_1}\times \cdots \times \calP^{\frac{1}{r}}_{\part_l} \ar[r]^{\hspace{4.25em}\mathrm{pr}_{i+1}} & \calP^{\frac{1}{r}}_{\part_i}.
}
\]
Let $\overline{\bm{\pi}}_{i}$ be the coarse space morphism associated to $\bm{\pi}_{i}$ (i.e. forget the $r$-twisted structure). Also, let $\bm{\pi}_{\calP}: {\calC}_{\calP}\rightarrow \calP^{_{1/r}}_{\part_i}$ be the universal curve and $\L_{\calP}$ be the universal $r$-th root bundle for $\calP^{_{1/r}}_{\part_i}$. Then we can form the following cartesian diagram:
\[
\xymatrix@R=1.5em{
\overline{\calC}_{i} \ar[d]_{\overline{\bm{\pi}}_{i}}\ar[r]_{\bm{d'}\hspace{0.5em}} & {\calC}_{\calP} \ar[d]^{\bm{\pi}_{\calP}}  \\
\fixedDR \ar[r]^{\bm{d}\hspace{0.5em}}&  \calP^{\frac{1}{r}}_{\part_i} \\
}
\]
If $\bm{\gamma}_{i}: \calC_i \rightarrow \overline{\calC}_i$ is the morphism forgetting the twisted structure, the above cartesian diagram gives an isomorphisms of the form $\bm{d'}^*\L_{\calP} \cong {\bm{\gamma}_i}_* \L_i$ and $\bm{d'}^*\L_{\calP}^r \cong {\bm{\gamma}_i}_* \L_i^r$. Hence, we have isomorphisms 
\[
R{\bm{\pi}_{i}}_* \L_i\cong \bm{d}^* R{\bm{\pi}_{\calP}}_* \L_{\calP}
\hspace{1cm}
\mbox{and}
\hspace{1cm}
R{\bm{\pi}_{i}}_* \L_i^r \cong \bm{d}^* R{\bm{\pi}_{\calP}}_* \L_{\calP}^r.
\]\\

\vspace{-0.5cm}
The natural forgetful morphism $\calP^{_{1/r}}_{\part_i}  \rightarrow  \calP_{\part_i}$ is the étale gerbe $\calB (\mu_r\times {\mu_{\part_i}}) \rightarrow \calB {\mu_{\part_i}}$. Combining this with the right square of the right diagram of (\ref{normalisation_along_flag_nodes_stable_maps_diagrams}) gives the following diagram where both squares are cartesian:
\[
\xymatrix@R=1.5em{
{\calC}_{\calP} \ar[d]^{\bm{\pi}_{\calP}}\ar[r]\ar@/^/[rr]^{\bm{d''}} & {\calC}_{\calP_{\part_i}} \ar[d] \ar[r] & \P^1 \ar[d]\\
 \calP^{\frac{1}{r}}_{\part_i} \ar[r]&  \calP_{\part_i}  \ar[r] & \bullet\\
}
\]
This shows that ${\calC}_{\calP}$ is a quotient stack ${\calC}_{\calP} \cong \big[\, \P^1 / (\mu_r\times {\mu_{\part_i}}) \,\big]$
where $(\mu_r\times {\mu_{\part_i}})$ acts trivially on $\P^1$.\\

The bundle $\L_{\calP}$ can now be expressed as the tensor product $\bm{d''}^*\O_{\P^1}(\lfloor \part_i/r \rfloor ) \otimes \bm{\pi}_{\calP}^*\,\calU_1$ where $\calU_1$ is line bundle on $\calP^{_{1/r}}_{\part_i}\cong \calB (\mu_r\times {\mu_{\part_i}})$ given by a trivial bundle where $(\mu_r\times {\mu_{\part_i}})$ acts with weight $(1,0)$. Hence, $R{\bm{\pi}_{\calP}}_* \L_{\calP}$ is also  trivial where $(\mu_r\times {\mu_{\part_i}})$ acts with weight $(1,0)$. The universal section $\bm{\sigma}: \O_{\calC_{\fixedDR}} \rightarrow \L$ is non-zero on fibres and hence rigidifies the $\mu_r$-action on $\fixedDR$. Thus $\bm{d}^*\calU_1 \cong \O_{\fixedDR}$ and  $\bm{d}^* R{\bm{\pi}_{\calP}}_* \L_{\calP}$ is trivial. \\

The bundle $\L_{\calP}^r$ is given by $\bm{d''}^*\O_{\P^1}(\part_i)$ with trivial $(\mu_r\times {\mu_{\part_i}})$-action, so $R{\bm{\pi}_{\calP}}_* \L_{\calP}^r$ is a trivial bundle with trivial $(\mu_r\times {\mu_{\part_i}})$-action. 
\end{proof}


\section{Proving Theorem \ref{VL_theorem}: Localisation Formula} \label{localisation_formula_section}

In order to be consistent with the exposition of \cite{GraberPand}, we will use the (derived) dual description of perfect obstruction theories for this section only. In this language, an equivariant perfect obstruction theory for a Deligne-Mumford stack $\calX$ is a morphism $\eT_{\calX} \rightarrow \eF_{\calX}$ in $\mathrm{D}_b^{\mathtt{e}}(\calX)$ satisfying properties dual to those given in \cite{Intrinsic}. A review of equivariant perfect obstruction theories is given in section \ref{background_equivariant_POT_subsection}. \\

\subsection{Virtual Localisation Formula}

\begin{theorem}[\textit{Virtual Localisation Formula \cite{GraberPand,ChangKiemLi}}] \label{localisation_theorem_GraberPand} 
Let $\calX$ be a Deligne-Mumford stack with a $\C^*$-action and let $\eT_{\calX}\rightarrow \eF_{\calX}$ be an equivariant perfect obstruction theory. Then
\[
\left[\calX\right]^{\mathrm{vir}} =  \sum_n (\bm{\zeta}_n)_* \frac{ \left[\calX_n\right]^{\mathrm{vir}}}{e(\mathrm{N}_n^{\mathrm{vir}})}
\]
in $A_*^{\C^*}(\calX) \otimes \mathbb{Q}[t,\frac{1}{t}]$ where:
\begin{enumerate}
\item $t$ is the generator of the $\C^*$-equivariant ring of a point.
\item The sum ranges over the connected components of the $\C^*$-fixed locus $\calX^{\C^*}$ with inclusions $\bm{\zeta}_n: \calX_n \rightarrow \calX$. 
\item $ \left[\calX_n\right]^{\mathrm{vir}}$ arises from the $\C^*$-fixed part $(L\bm{\zeta}_n^*\eF_{\calX})^{\mathtt{fix}}$ which is a perfect obstruction theory for $\calX_n$. 
\item $\mathrm{N}_n^{\mathrm{vir}}$ is the $\C^*$-moving part $(L\bm{\zeta}_n^*\eF_{\calX})^{\mathtt{mov}}$ of the perfect obstruction theory and $e(-)$ is the equivariant Euler class. \\
\end{enumerate}
\end{theorem}

\begin{remark}
Theorem \ref{localisation_theorem_GraberPand} was originally due to \cite{GraberPand}, however it required the existence of a $\C^*$-equivariant embedding of $\calX$ into a smooth Deligne-Mumford stack. The requirement of this condition was removed in \cite{ChangKiemLi} and the localisation method was extended to include other, more general, concepts. \\
\end{remark}

\begin{re}[\textit{Perfect Obstruction Theory}]\label{dual_pot_re_vir_loc}
It will be shown in corollary \ref{POT_for_M1/r_main_corollary} that there is an equivariant perfect obstruction theory $\eT_{\M^{1/r}} \rightarrow \eF_{\M^{1/r}}$ for ${\M^{1/r}}$ which fits into the distinguished triangle 
\begin{align*}
\xymatrix@=3em{
 \eF_{\ForMapDRtoRSM} \ar[r] & \eF_{\M^{1/r}}  \ar[r] & L\ForMapDRtoRSM^*\eF_{\M^r}\ar[r] & \eF_{\ForMapDRtoRSM}[1]
 } 
\end{align*}
and is compatible with the distinguished triangle of equivariant tangent complexes arising from the morphism $\ForMapDRtoRSM: \M^{1/r}\rightarrow \M^r$. Here, $\eF_{\M^r}$ is the perfect obstruction theory for $\M^r$ to be given in \ref{equi_POT_for_M_re}.  $\eF_{\ForMapDRtoRSM}$ is the equivariant perfect relative obstruction theory for $\ForMapDRtoRSM$  first constructed in \cite{Leigh_Ram}; it is described further in \ref{equivariant_pot_nu_re} and \ref{equivariant_pot_nu_re_alt}. \\
\end{re}

\begin{re}[\textit{Strategy for Localisation Calculation}]
The strategy for computing the localisation formula will be to use the distinguished triangle from \ref{dual_pot_re_vir_loc} and calculate the fixed and moving parts on each term. Indeed, pulling back via $\fixedIncDR:\fixedDR\hookrightarrow \M^{1/r}$ and taking either the fixed or the moving part gives another distinguished triangle
\[
\xymatrix@C=1.25em{
 \big(L\fixedIncDR^* \eF_{\ForMapDRtoRSM}\big)^{\fixmov} \ar[r] &  \big(L\fixedIncDR^* \eF_{\M^{1/r}} \big)^{\fixmov}  \ar[r] & \big(L\fixedIncDR^*L\ForMapDRtoRSM^*\eF_{\M^r}\big)^{\fixmov}   \ar[r]& \big(L\fixedIncDR^* \eF_{\ForMapDRtoRSM}\big)^{\fixmov}  [1].
}
\]
This shows that we can compute $ (L\fixedIncDR^* \eE_{\M^{1/r}} )^{\fixmov}$ by computing $(L\fixedIncDR^*L\ForMapDRtoRSM^*\eE_{\M^r})^{\fixmov}$ and $(L\fixedIncDR^* \eE_{\ForMapDRtoRSM})^{\fixmov} $ individually. We will compute $(L\fixedIncDR^*L\ForMapDRtoRSM^*\eE_{\M^r})^{\fixmov}$ in section \ref{analysis_fix_mov_Mr_subsection} and $(L\fixedIncDR^* \eE_{\ForMapDRtoRSM})^{\fixmov} $  in section  \ref{analysis_fix_mov_nu_subsection}. \\
\end{re}

\subsection{Analysis of the Contributions From \texorpdfstring{$\M^r$}{Mr}} \label{analysis_fix_mov_Mr_subsection}

\begin{lemma}\label{fixed_moving_parts_pullback_of_RSM}
The fixed and moving parts of $L\fixedIncDR^*L\ForMapDRtoRSM^*\eF_{\M^r}$ have:
\begin{enumerate}
\item $\big(L\fixedIncDR^*L\ForMapDRtoRSM^*\eF_{\M^r})^{\fix} \cong F_0$ for some locally free sheaf $F_0$ of rank $3g-3+l$.
\item The Euler class of the moving part is the following:
\begin{enumerate}
\item If $g=0$ and $l=1$ then 
\begin{flalign*}
\hspace{5.5em}
e\Big(\big(L\fixedIncDR^* L\ForMapDRtoRSM^*{\eF_{\M^r}}\big)^{\mov}\Big) &= t^{\part_1-1} \frac{\part_1!}{\part_1^{\part_1-1}}  &
\end{flalign*}
\item If $g=0$ and $l=2$ then 
\begin{flalign*}
\hspace{5.5em}
e\Big(\big(L\fixedIncDR^* L\ForMapDRtoRSM^*{\eF_{\M^r}}\big)^{\mov}\Big) &= \frac{ t^{\part_1+\part_2}}{r}\, \frac{\part_1!}{\part_1^{\part_1+1}} \frac{\part_2!}{\part_2^{\part_2+1}} \left( \part_1 + \part_2 \right)&
\end{flalign*}
\item If $3g-3+l\geq0$ then 
\begin{flalign*}
\hspace{5.5em}
e\Big(\big(L\fixedIncDR^* L\ForMapDRtoRSM^*{\eF_{\M}}\big)^{\mov}\Big) &= 
\frac{r^{-l} \, t^{1-g+|\part|}}{c_{\frac{1}{t}}\big(- \bm{b}^*\bm{\rho}_* \omega_{\bm{\rho}}\big)} \prod_{i=1}^{l} \frac{\part_i!}{\part_i^{\part_i+1}}\left(1 - \frac{\part_i }{t}\psi_{i}\right)&
\end{flalign*}
where $c_{s}$ is the Chern polynomial, $\bm{b}: \fixedDR \longrightarrow \Mbar^{_{1/r, \mathtt{a}}}_{g,l(\part)}$ is the morphism from corollary \ref{degree_of_map_to_spin_curves} and $\bm{\rho}$ is the universal curve for $\Mbar^{_{1/r, \mathtt{a}}}_{g,l(\part)}$.\\
\end{enumerate}
\end{enumerate}
\end{lemma}

\vspace{-0.4cm}
\begin{remark}
This calculation is essentially the same as in \cite[\S 4]{GraberPand} with changes arising from the relative condition described in \cite[\S 3.7]{GraVakil_Loc}. Additional differences in the case at hand arise from the deformation space of an ($r$-twisted) flag node  being an $r$-fold cover of the deformation space of the associated coarse flag node.\\
\end{remark}

\vspace{-0.4cm}
\begin{proof}
We prove the case of $3g-3+l\geq0$ (with the other cases requiring few modifications). Consider the universal morphism $\ForMap: \M^r \rightarrow \frakM^r \times \calT$ that forgets all data but the source and target families. This gives rise to the following distinguished triangle of perfect obstruction theories whose dual is discussed more in \ref{equi_POT_for_M_re}
\begin{align*}
\xymatrix{
   \eF_{\ForMap}  \ar[r]  & \eF_{\M^r}  \ar[r]  & {\ForMap}^* \eT_{\frakM^r \times \calT} \ar[r] &        \eF_{\ForMap}[1].
}
\end{align*}
Here $\eF_{\ForMap}$ is the pullback of the perfect relative obstruction theory originally given by Behrend in \cite{Behrend_GW} and extend to the relative stable maps case by \cite{GraVakil_Loc}. \\

In the simple fixed locus the target curves are not degenerated. Hence, there are none of the complications which arise from admissibility conditions discussed in \cite[\S2.8]{GraVakil_Loc} and we have an isomorphism
\[
{\ForMap}^* \eT_{\frakM^r \times \calT} 
\cong
R{\bm{\pi}_{\fixedDR}}_* R\sheafHom\big(\Omega_{\bm{\pi}_{\fixedDR}}(\mbox{$\sum_i$}\bm{q}_i), \O_{\calC_{\fixedDR}} \big). 
\]
For each flag node we have an étale gerbe $\bm{\gamma}_{\Theta_i}:\Theta \rightarrow\fixedDR$ which defines a sub-stack corresponding to the $r$-twisted node $\bm{\vartheta}_i: \Theta_i \hookrightarrow \calC_{\fixedDR}$ and corresponding sub-stacks of the normalised components 
\[
\check{\bm{\vartheta}}_i: \Theta_i \hookrightarrow \calC_v
\hspace{1cm}
\mbox{and}
\hspace{1cm}
\hat{\bm{\vartheta}}_i: \Theta_i \hookrightarrow \calC_i. 
\]
There is the following exact sequence (which is the $\C^*$-equivariant and $r$-twisted version of the clutching morphism from \cite[Thm. 3.5]{Knudsen_Mgn}):
\[
\xymatrix{
0 \ar[r] &\bigoplus_i {\bm{\vartheta}_i}_*\big({\check{\bm{\vartheta}}_i}^* \Omega^{\mathtt{e}}_{\bm{\pi}_v} \otimes {\hat{\bm{\vartheta}}_i}^*\Omega^{\mathtt{e}}_{\bm{\pi}_i}\big)  \ar[r] & \Omega^{\mathtt{e}}_{\bm{\pi}_{\fixedDR}}(\mbox{$\sum_i$}\bm{q}_i) \ar[r] &  \bm{n}_*\Omega^{\mathtt{e}}_{\widetilde{\bm{\pi}}}(\mbox{$\sum_i$}\hat{\bm{q}}_i)  \ar[r] & 0 .
}
\]
Note that we have also tensored by $\O_{\calC_{\fixedDR}}(\mbox{$\sum_i$}\bm{q}_i)$ and used the notation $\hat{\bm{q}}_i:\fixedDR \rightarrow \widetilde{\calC}$ for the morphisms corresponding to $\bm{q}_i:\fixedDR \rightarrow \calC_{\fixedDR}$. We can now analyse the fixed and moving parts of ${\ForMap}^* \eL_{\frakM^r \times \calT}$ by applying the functor $R{\bm{\pi}_{\fixedDR}}_*R\sheafHom(\,-\,,\O_{\calC_{\fixedDR}}) $ to this exact sequence and analysing the result. \\

\begin{itemize}[leftmargin=1.5em]
\item[(i)] \textit{Contribution from ${\bm{\vartheta}_i}_*\big({\check{\bm{\vartheta}}_i}^* \Omega^{\mathtt{e}}_{\bm{\pi}_v} \otimes {\hat{\bm{\vartheta}}_i}^*\Omega^{\mathtt{e}}_{\bm{\pi}_i}\big)$}: \\
Considering this term using Serre-Grothendieck duality for twisted curves gives
\begin{align*}
&R{\bm{\pi}_{\fixedDR}}_*R\sheafHom\Big(
 {\bm{\vartheta}_i}_*\big({\check{\bm{\vartheta}}_i}^* \Omega^{\mathtt{e}}_{\bm{\pi}_v} \otimes {\hat{\bm{\vartheta}}_i}^*\Omega^{\mathtt{e}}_{\bm{\pi}_i}\big)
,~\O_{\calC_{\fixedDR}}\Big)\\
&\cong
R{\bm{\pi}_{\fixedDR}}_*R\sheafHom\Big(
 {\bm{\vartheta}_i}_*\big({\check{\bm{\vartheta}}_i}^* \Omega^{\mathtt{e}}_{\bm{\pi}_v} \otimes {\hat{\bm{\vartheta}}_i}^*\Omega^{\mathtt{e}}_{\bm{\pi}_i}\big) \otimes \omega_{\bm{\pi}_{\fixedDR}}
,~\omega_{\bm{\pi}_{\fixedDR}}[1]\Big) [-1]\\
&\cong
\Big[
{\bm{\gamma}_{\Theta_i}}_* \big({\check{\bm{\vartheta}}_i}^* \Omega^{\mathtt{e}}_{\bm{\pi}_v} \otimes {\hat{\bm{\vartheta}}_i}^*\Omega^{\mathtt{e}}_{\bm{\pi}_i}\big)  \Big]^{\vee} [-1].
\end{align*}
Denote by $\calE:={\check{\bm{\vartheta}}_i}^* \Omega_{\bm{\pi}_v} \otimes {\hat{\bm{\vartheta}}_i}^*\Omega_{\bm{\pi}_i}$ the line bundle where the $\C^*$-equivariant structure is forgotten. We claim that ${\bm{\gamma}_{\Theta_i}}_*\calE $ is a line bundle on $\fixedDR$. To see this note that ${\bm{\gamma}_{\Theta_i}}$ is relative dimension $0$ so we only need to show that $h^0$ has constant rank $1$ for all fibres. For a geometric point $\bullet\hookrightarrow \fixedDR$ the local picture of $\calC_{\fixedDR}$ at the flag node is  given by $V:=\left[ U/ \mu_r\right]$, for $U = \Spec\,\C[z,w]/zw$ with action of $\mu_r$ given by $(z,w)\mapsto (\xi_r z, \xi_r^{-1} w)$. Then the local picture of $\calE |_\bullet$ has local generator $dz\otimes dw$ which is $\mu_r$ invariant so $h^0$ has constant rank $1$.\\

\noindent
Since ${\bm{\gamma}_{\Theta_i}}_*\calE$ is a line bundle on the $\C^*$ fixed locus we have that 
\[
{\bm{\gamma}_{\Theta_i}}_* \big({\check{\bm{\vartheta}}_i}^* \Omega^{\mathtt{e}}_{\bm{\pi}_v} \otimes {\hat{\bm{\vartheta}}_i}^*\Omega^{\mathtt{e}}_{\bm{\pi}_i}\big) 
\cong 
\calV_c \otimes {\bm{\gamma}_{\Theta_i}}_*\calE
\]
where $\calV_c$ is a trivial bundle on $\fixedDR$ with $\C^*$-weight $c$. We can calculate the weight $c$ by considering the geometric point $\bullet\hookrightarrow \fixedDR$. The equivariant bundle has local generator $dz\otimes dw$ giving that $\C^*$ acts with weight $(-{r \part_i})^{-1}$. For $(g,l)=(0,1)$ this term does not exist and for $(g,l)=(0,2)$ the weight is $(-{r \part_1+ r\part_2)^{-1}}$.
\\

\noindent
Hence this contribution has no fixed part and taking the Euler class gives
\[
e\Big(\big[
{\bm{\gamma}_{\Theta_i}}_* \big({\check{\bm{\vartheta}}_i}^* \Omega^{\mathtt{e}}_{\bm{\pi}_v} \otimes {\hat{\bm{\vartheta}}_i}^*\Omega^{\mathtt{e}}_{\bm{\pi}_i}\big)  \big]^{\vee} [-1] \Big)
=
e\big(\calV_c \otimes {\bm{\gamma}_{\Theta_i}}_*\calE\big)
= e\big({\bm{\gamma}_{\Theta_i}}_*\calE\big) -\tfrac{t}{r \part_i}.
\]
To compute $e\big({\bm{\gamma}_{\Theta_i}}_*\calE\big)$ we observe that the local generator of the line bundle  $\calE \cong {\check{\bm{\vartheta}}_i}^*\O_{\calC_v}(-\check{\bm{\vartheta}}_i) \otimes  {\hat{\bm{\vartheta}}_i}^* \O_{\calC_i}(-\hat{\bm{\vartheta}}_i)$ is $\mu_r$ invariant which gives $\calE\cong {\bm{\gamma}_{\Theta_i}}^{\hspace{-0.3em}*} {\bm{\gamma}_{\Theta_i}}_*\calE$. The projection formula  now gives
\[
({\bm{\gamma}_{\Theta_i}}_*\calE)^{\otimes r} 
\cong  {\bm{\gamma}_{\Theta_i}}_*(\calE^{\otimes r})
\cong {\check{\bm{\vartheta}}_i}^*\O_{\calC_v}(-r\check{\bm{\vartheta}}_i) \otimes  {\hat{\bm{\vartheta}}_i}^* \O_{\calC_i}(-r\hat{\bm{\vartheta}}_i)
\cong {\check{\bm{z}}_i}^* \Omega_{\overline{\bm{\pi}}_v} \otimes {\hat{\bm{z}}_i}^*\Omega_{\overline{\bm{\pi}}_i}
\]
where we have used the notation $\check{\bm{z}}_i: \fixedDR \hookrightarrow \overline{\calC}_v$ and 
$\hat{\bm{z}}_i: \fixedDR \hookrightarrow \overline{\calC}_i$ for the morphisms corresponding to $\bm{z}_i: \fixedDR \hookrightarrow \overline{\calC}_{\fixedDR}$.  Since ${\hat{\bm{z}}_i}^*\Omega_{\overline{\bm{\pi}}_i}$ is trivial as a bundle (as is the case in \cite[\S 4]{GraberPand}) we have $e\big({\bm{\gamma}_{\Theta_i}}_*\calE\big) = \frac{1}{r}\psi_i$. This vanishes for $(g,l)=(0,1),(0,2)$. \\

\item[(ii)]  \textit{Contribution from  $
\bm{n}_*\Omega^{\mathtt{e}}_{\widetilde{\bm{\pi}}}(\mbox{$\sum_i$}\hat{\bm{q}}_i)
\cong {\bm{\iota}_v}_* \Omega^{\mathtt{e}}_{\bm{\pi}_v} \oplus \bigoplus_i {\bm{\iota}_i}_* \Omega^{\mathtt{e}}_{\bm{\pi}_i}(\hat{\bm{q}}_i) 
$}:\\
Considering this term using Serre-Grothendieck duality for twisted curves gives
\begin{align*}
&R{\bm{\pi}_{\fixedDR}}_*R\sheafHom\Big(
\bm{n}_*\Omega^{\mathtt{e}}_{\widetilde{\bm{\pi}}}(\mbox{$\sum_i$}\hat{\bm{q}}_i)
,~\O_{\calC_{\fixedDR}}\Big) [1]\\
&\cong
R{\bm{\pi}_{\fixedDR}}_*R\sheafHom\Big(
\bm{n}_*\Omega^{\mathtt{e}}_{\widetilde{\bm{\pi}}}(\mbox{$\sum_i$}\hat{\bm{q}}_i) \otimes \omega_{\bm{\pi}_{\fixedDR}}
,~\omega_{\bm{\pi}_{\fixedDR}}[1]\Big) \\
&\cong
R{\bm{\pi}_{\fixedDR}}_*R\sheafHom\Big(
{\bm{\iota}_v}_* \big(\Omega^{\mathtt{e}}_{\bm{\pi}_v}\otimes \omega_{\bm{\pi}_v}(\mbox{$\sum_i$}\check{\bm{\vartheta}}_i)\big)\oplus \bigoplus\nolimits_i {\bm{\iota}_i}_* \big(\Omega^{\mathtt{e}}_{\bm{\pi}_i}(\hat{\bm{q}}_i) \otimes \omega_{\bm{\pi}_i}(\hat{\bm{\vartheta}}_i)\big) ,~\omega_{\bm{\pi}_{\fixedDR}}[1]\Big)\\
&\cong
\Big[
R{\bm{\pi}_{v}}_* \big(\Omega^{\mathtt{e}}_{\bm{\pi}_v}(\mbox{$\sum_i$}\check{\bm{\vartheta}}_i)\otimes \omega_{\bm{\pi}_v}\big)
\Big]^{\vee}
\oplus \bigoplus\nolimits_i 
\Big[ R{\bm{\pi}_{i}}_* \big(\Omega^{\mathtt{e}}_{\bm{\pi}_i}( \hat{\bm{\vartheta}}_i+ \hat{\bm{q}}_i) \otimes \omega_{\bm{\pi}_i}\big) 
\Big]^{\vee}
\end{align*}
Considering the direct summands individually gives:\\

\begin{itemize}[leftmargin=1.5em]
\item[(a)]
The term $\big[
R{\bm{\pi}_{v}}_* \big(\Omega^{\mathtt{e}}_{\bm{\pi}_v}(\mbox{$\sum_i$}\check{\bm{\vartheta}}_i)\otimes \omega_{\bm{\pi}_v}\big)
\big]^{\vee} [-1]
\cong
R{\bm{\pi}_{v}}_*R\sheafHom\big(
\Omega^{\mathtt{e}}_{\bm{\pi}_v}(\mbox{$\sum_i$}\check{\bm{\vartheta}}_i)
,\,\O_{\calC_{\fixedDR}}\big)
$ has $\C^*$-weight 0 and was shown in \cite[\S 2]{AbraJarvis} and \cite[\S 3]{Twisted_and_Admissible} to be locally free of rank $3g-3+l$. \\

\item[(b)]
The term $\big[ 
R{\bm{\pi}_{i}}_* \big(\Omega^{\mathtt{e}}_{\bm{\pi}_i}( \hat{\bm{\vartheta}}_i+ \hat{\bm{q}}_i) \otimes \omega_{\bm{\pi}_i}\big) 
\big]^{\vee}[-1] 
$ is the pullback of a corresponding bundle on $\M$. Specifically,  there is an isomorphism
\begin{align*}
\big[ 
&R{\bm{\pi}_{i}}_* \big(\Omega^{\mathtt{e}}_{\bm{\pi}_i}( \hat{\bm{\vartheta}}_i+ \hat{\bm{q}}_i) \otimes \omega_{\bm{\pi}_i}\big) 
\big]^{\vee}[-1] \\
& \cong 
\big[ R{\overline{\bm{\pi}}_{i}}_* \big( \Omega^{\mathtt{e}}_{\overline{\bm{\pi}}_i}( \hat{\bm{z}}_i+ \hat{\bm{q}}_i) \otimes \omega_{\overline{\bm{\pi}}_i}\big) 
\big]^{\vee}[-1]\\
& \cong 
(\ForMapRSMtoSMfixed\circ\ForMapDRtoRSMfixed)^* \big[ R{{\bm{\pi}}_{\fixedSM,i}}_* \big( \Omega^{\mathtt{e}}_{{\bm{\pi}}_{\fixedSM,i}}( \hat{\bm{z}}_{\fixedSM,i}+ \hat{\bm{q}}_{\fixedSM,i}) \otimes \omega_{{\bm{\pi}}_{\fixedSM,i}}\big) 
\big]^{\vee}[-1]
\end{align*}
where we have label the morphisms and sheaves related to $\fixedSM$ using the by adding $\fixedSM$ to the subscripts. This is the pullback of the bundle studied in \cite[\S 4]{GraberPand} and \cite[\S 3.7]{GraVakil_Loc}. We can use that analysis without change to show that this term cancels with a term in part (iii) below. \\
\end{itemize}

\item[(iii)]\textit{Contribution from $\eF_{\ForMap}$}:\\
There is an isomorphism $\eF_{\ForMap} 
\cong (\ForMapRSMtoSMfixed\circ\ForMapDRtoRSMfixed)^* R{\overline{\bm{\pi}}_{\fixedSM}}_* \bm{f}_{\fixedSM}^* (\omega_{\P^1}^{\mathrm{log}})^\vee.$ Hence we can use the analysis for this term from \cite[\S 4]{GraberPand} with the relative changes from \cite[\S 3.7]{GraVakil_Loc}. The desired result now follows once we employee the isomorphism of Hodge bundles $(\ForMapRSMtoSMfixed\circ\ForMapDRtoRSMfixed)^* {\bm{\pi}_{\fixedSM,v}}_*\omega_{\bm{\pi}_{\fixedSM,v}} \cong \bm{b}^*\bm{\rho}_* \omega_{\bm{\rho}}$ discussed in \ref{sheaves_on_fixed_locus_pullbacks} part \ref{dualizing_sheaves_iso}. 
\end{itemize}
\end{proof}

\subsection{Analysis of the Contributions Relative to \texorpdfstring{$\ForMapDRtoRSM$}{nu}} \label{analysis_fix_mov_nu_subsection}

\begin{re}[\textit{Description of Perfect Relative Obstruction Theory for $\ForMapDRtoRSM$}]
Consider the following commutative diagram:
\begin{align}
\begin{array}{c}
\xymatrix@R=1.5em{
\fixedDR \ar@{=}[d] & \calC_{\fixedDR} \ar[l]_{\bm{\pi}_{\fixedDR}} \ar@{=}[d] \ar[r]^{\mathfrak{f}~~} & \mathsf{Tot}_{\calC_{\fixedDR}} \L \ar[r]^{~~\bm{\beta}_{\fixedDR}}  \ar[d]^{\bm{\tau}_{\fixedDR}} & \calC_{\fixedDR} \ar@{=}[d]\\
\fixedDR & \calC_{\fixedDR} \ar[l]_{\bm{\pi}_{\fixedDR}}  \ar[r]^{\mathfrak{f}'~~} & \mathsf{Tot}_{\calC_{\fixedDR}} \L^r\ar[r]^{~~\bm{\alpha}_{\fixedDR}}  & \calC_{\fixedDR} \\
}
\end{array} \label{nu_pot_restriction_diagram}
\end{align}
In this diagram $\bm{\alpha}_{\fixedDR}$, $\bm{\beta}_{\fixedDR}$ and $\bm{\tau}_{\fixedDR}$ are the morphisms from \ref{rth_power_maps_on_M} pulled back to ${\fixedDR}$ from  $\M^{1/r}$. The morphisms $\mathfrak{f}$ and $\mathfrak{f}'$ closed immersions defined by the universal sections $\bm{\sigma}$ and $\bm{\sigma}^r$ respectively. Using this notation we have (as described in \ref{equivariant_pot_nu_re_alt}) that the equivariant perfect relative obstruction theory for $\ForMapDRtoRSM$ restricted to $\fixedDR$ is
\[
L\fixedIncDR^* \eF_{\nu} = R{\bm{\pi}_{\fixedDR}}_* L\mathfrak{f}^*  \eT_{\bm{\tau}_{\fixedDR}}. 
\]\\
\end{re}

\vspace{-0.5cm}
\begin{re} \label{fixed_moving_nu_morphism_notation}
Consider the morphism $\bm{n}:\widetilde{\calC} \rightarrow  \calC_{\fixedDR}$ (discussed in \ref{normalisation_along_flag_nodes_DR}) normalising $\calC_{\fixedDR}$ at the distinguished nodes. This gives the distinguished triangle in $\mathrm{D}_b^{\mathtt{e}}( \calC_{\fixedDR})$ 

\[
\xymatrix{
 \O_{ \calC_{\fixedDR}} \ar[r] & \bm{n}_* \O_{\widetilde{\calC}}  \ar[r] & \bigoplus\limits_{i} (\bm{\vartheta}_i )_* \O_{\Theta_i}  \ar[r]&  \O_{ \calC_{\fixedDR}} [1]
}
\]
where all $\C^*$-linearisations are trivial. Now, taking the (derived) tensor product by $L\mathfrak{f}^*  \eT_{\bm{\tau}_{\fixedDR}}$ and pushing forward via $R{\bm{\pi}_{\fixedDR}}_*$  we obtain the following distinguished triangle 

\[
\xymatrix{
L\fixedIncDR^*\eF_{\ForMapDRtoRSM} \ar[r] & \calF_{v} \oplus \bigoplus\limits_{i} \calF_{i} \ar[r] & \bigoplus\limits_{i} \calF_{\Theta_i} \ar[r]&  L\fixedIncDR^*\eF_{\ForMapDRtoRSM}[1]
}
\]
where, using the notation from \ref{etale_gerbes_of_simple_locus_DR} and \ref{normalisation_along_flag_nodes_DR}, we have
\begin{enumerate}
\item $\calF_{v} = R{\bm{\pi}_{v}}_*  L \bm{\iota}_{v}^* L\mathfrak{f}^*  \eT_{\bm{\tau}_{\fixedDR}}$, 
\item $\calF_{i} = R{\bm{\pi}_{i}}_* L \bm{\iota}_{v}^* L\mathfrak{f}^*  \eT_{\bm{\tau}_{\fixedDR}}$, 
\item $\calF_{\Theta_i} = R(\bm{\gamma}_{\Theta_i})_*  L \bm{\vartheta}_i^* L\mathfrak{f}^*  \eT_{\bm{\tau}_{\fixedDR}}$. \\
\end{enumerate}
\end{re}

\vspace{0.1cm}
\begin{lemma} \label{distinguished_triangle_for_vertex_contributions}
Let $3g-3+l\geq 0$ and let $\mathrm{log}$ refer to the un-twisted markings on $\calC_v$ corresponding to the flag nodes. There is a distinguished triangle in $\mathrm{D}_b^{\mathtt{e}}(\fixedDR)$:
\[
\xymatrix{
\calF_{v} \ar[r] & \calV_{\frac{1}{r}}  \otimes R{\bm{\pi}_v}_*\L_{v} \ar[r] & \calV_{1} \otimes  R{\bm{\pi}_v}_*\omega_{\bm{\pi}_v}^{\mathrm{log}} \ar[r]&  \calF_{v}[1]
}
\]
where $R{\bm{\pi}_v}_*\L_{v}$ and $R{\bm{\pi}_v}_*\omega_{\bm{\pi}_v}^{\mathrm{log}}$ have trivial $\C^*$-linearisations and $\calV_c$ is a trivial bundle on $\fixedDR$ with weight $c$. 
\end{lemma}
\begin{proof}
Consider the following diagram which is a restriction of the diagram (\ref{nu_pot_restriction_diagram}) to $\calC_v$ (with appropriately chosen labels):
\[
\xymatrix@R=1.5em{
\fixedDR \ar@{=}[d] & \calC_{v} \ar[l]_{\bm{\pi}_{v}} \ar@{=}[d] \ar[r]^{\mathfrak{f}_v~~} & \mathsf{Tot}_{\calC_{v}} \L_v \ar[r]^{~~\bm{\beta}_v}  \ar[d]^{\bm{\tau}_v} & \calC_{v} \ar@{=}[d]\\
\fixedDR & \calC_{v} \ar[l]_{\bm{\pi}_{v}}  \ar[r]^{\mathfrak{f}'_v~~} & \mathsf{Tot}_{\calC_{v}} \L_v^r\ar[r]^{~~\bm{\alpha}_v}  & \calC_{v} \\
}
\]
With this notation we have an isomorphism
\[
\calF_{v} = R{\bm{\pi}_{v}}_*  L \bm{\iota}_{v}^* L\mathfrak{f}^*  \eT_{\bm{\tau}_{\fixedDR}}
\cong
R{\bm{\pi}_{v}}_*  L\mathfrak{f}_v^*  \eT_{\bm{\tau}_v}
\]
and the distinguished triangle of equivariant tangent complexes arising from $\bm{\tau}_{v}$
\[
\xymatrix{
L\mathfrak{f}_v^*  \eT_{\bm{\tau}_v}  \ar[r] & \mathfrak{f}_v^*  \eT_{\bm{\beta}_v} \ar[r] & \mathfrak{f}_v^* \bm{\tau}_v^* \eT_{\bm{\alpha}_v} \ar[r]&  L\mathfrak{f}_v^*  \eT_{\bm{\tau}_v}[1].
}
\]
Since $\calC_v$ is a $\C^*$-fixed locus, any equivariant line bundle will be of the form $\calE \otimes\calU_{c}$ where $\calE$ is a line bundle with trivial $\C^*$-linearisation and $\calU_{c}$ is a trivial bundle on $\calC_v$ with weight $c$. Moreover, since $\fixedDR$ is also $\C^*$-fixed we have $\calU_{c}\cong \bm{\pi}_v^* \calV_{c}$ where $\calV_c$ is a trivial bundle on $\fixedDR$ with weight $c$.\\

Forgetting the $\C^*$-linearisations,  $\mathfrak{f}_v^*  \eT_{\bm{\beta}_v} $ and $\mathfrak{f}_v^* \bm{\tau}_v^* \eT_{\bm{\alpha}_v}$ correspond to the bundles $\L_v$ and $\omega_{\bm{\pi}_v}^{\mathrm{log}}$ respectively. Hence, we need to calculate the weight of the $\C^*$-linearisation. \\

We can calculate the weights by taking a geometric point in $\fixedDR$ given by
\[
\zeta = \Big(~
C
,~ f: C \rightarrow \P^1
,~ L
,~ e:L^r\overset{\sim}{\rightarrow} R_f
,~ \sigma:\O_C\rightarrow L 
~\Big)
\]
and considering the flag node of $C$ where $C_v$ meets a $C_i$. The weights must be the same as the weighs of the corresponding sheaves at the orbifold point of $C_i$. This has local picture given by $V:=\left[ U/ \mu_r\right]$, for $U = \Spec\, \C[z]$ where the action of $\mu_r$ is given by  $z\mapsto \xi_r z$. Let $\Psi$ be the local generator of $(\L_\zeta)|_{V}$ with isomorphisms
\begin{align}
\Psi^r  \, \C[z]
\cong
\frac{ dz }{z^{r \part_i+1} } \, \C[z]
\cong 
R_f\big|_V. 
\end{align}
Locally $\C^*$ acts on $V$ via $z\mapsto c^{-\frac{1}{r\part_i}} z$, so the $\C^*$-action on the local generator is $c\cdot \Psi = c^{\frac{1}{r}} \Psi$. Locally, the total space morphism $\bm{\beta}_v$ is given by
\[
A \rightarrow A[\Psi]
\]
Hence, we have the weight corresponding to $\bm{\beta}_v$ is $\frac{1}{r}$. Similarly we can calculate the weight for $\bm{\alpha}_v$ to be 1. 
\end{proof}
\mbox{}

\begin{corollary}\label{vertex_contributions_corollary}
In the case $3g-3+l\geq 0$, the fixed part of $\calF_{v}$ is zero and the moving part has equivariant Euler class
\[
e\big(\calF_v^{\,\mov}\big) = 
\frac{
c_{\frac{r}{t}}\big(\bm{b}^* R\bm{\rho}_* \L \big)
\left(\frac{t}{r}\right)^{m  -g+1-l -  \sum_i\left\lfloor\part_i/r\right\rfloor+\sum_{i: r|\part_i} 1}
}
{
c_{\frac{1}{t}}\big( \bm{b}^* \bm{\rho}_* \omega_{\bm{\rho}} \big) ~ t^{g-1+l}
}
\]
where $c_{s}$ is the Chern polynomial, $\bm{b}: \fixedDR \rightarrow \Mbar^{_{1/r, \mathtt{a}}}_{g,l}$ is from corollary \ref{degree_of_map_to_spin_curves} and where $\bm{\rho}$ and $\L$ are the the universal curve  and $r$-th root bundle for $\Mbar^{_{1/r, \mathtt{a}}}_{g,l}$.
\end{corollary}
\begin{proof}
We take the equivariant Euler class of the distinguished triangle from lemma \ref{distinguished_triangle_for_vertex_contributions}. We begin with the term $\calV_{\frac{1}{r}}  \otimes R{\bm{\pi}_v}_*\L_{v}$. As discussed in lemma \ref{sheaves_on_fixed_locus_pullbacks} part \ref{dist_triangle_forgetting_log_L}  there is a distinguished triangle
\[
\xymatrix{
\bm{b}^* R\bm{\rho}_* \L \ar[r] & R{\bm{\pi}_v}_* \L_v \ar[r] & \bigoplus\limits_{i:\, r|\part_i} \O_{\fixedDR} \ar[r] & \bm{b}^* R\bm{\rho}_* \L[1].
}
\]
This gives that  $e(\calV_{\frac{1}{r}}  \otimes R{\bm{\pi}_v}_*\L_{v}) = \left(\frac{t}{r}\right)^{\sum_{i:\, r|\part_i} 1} e(\calV_{\frac{1}{r}}  \otimes \bm{b}^* R\bm{\rho}_* \L)$. Now, using standard methods (for example \cite[Prop. 5]{Behrend_GW})  one can show $R\bm{\rho}_* \L$ is quasi-isomorphic to a two term sequence of bundles $
R\bm{\rho}_* \L
~ \cong~
 \big[\,E_{0} \longrightarrow E_1\,\big]$ in $D_b\big(\Mbar^{_{1/r, \mathtt{a}}}_{g,l}\big)$.\\

Pulling back and tensoring by $\calV_{\frac{1}{r}}$ gives a quasi-isomorphim in $D_b^{\mathtt{e}}(\fixedDR)$
\[
\calV_{\frac{1}{r}}  \otimes \bm{b}^*R\bm{\rho}_* \L
~ \cong ~
 \big[\, \calV_{\frac{1}{r}} \otimes F_{0} \longrightarrow \calV_{\frac{1}{r}} \otimes  F_1\,\big]. 
\]
where we have defined $F_n:=\bm{b}^*E_n$. Then taking the equivariant Euler class gives
\[
e\big(\calV_{\frac{1}{r}}  \otimes \bm{b}^*R\bm{\rho}_* \L \big)
= 
\frac{\sum_{k=0}^{\mathrm{rk}\,F_0} \left(\frac{t}{r}\right)^{\mathrm{rk}\,F_0-k} c_k(F_0)}{\sum_{k=0}^{\mathrm{rk}\,F_1} \left(\frac{t}{r}\right)^{\mathrm{rk}\,F_1- k} c_k(F_1)}
= 
\left(\frac{t}{r}\right)^{\mathrm{rk}\,F_0-\mathrm{rk}\,F_1}   c_{\frac{r}{t}}(\bm{b}^*R\bm{\rho}_* \L ).
\]
The degree of $\L$ is $m - l -  \sum_i\left\lfloor\part_i/r\right\rfloor$, hence we can calculate $\mathrm{rk}\,F_0-\mathrm{rk}\,F_1$ using Riemann-Roch for twisted curves. \\

The contribution from $\calV_{\frac{1}{r}}  \otimes R{\bm{\pi}_v}_*\L_{v}^r $ is calculated in a similar manner. We have a distinguished triangle discussed in \ref{sheaves_on_fixed_locus_pullbacks} part \ref{dist_triangle_forgetting_log_L^r} 
\[
\xymatrix{
\bm{b}^* R\bm{\rho}_* \omega_{\bm{\rho}} \ar[r] & R{\bm{\pi}_v}_* \L_v^r \ar[r] & \bigoplus\limits_{i} \O_{\fixedDR} \ar[r] & \bm{b}^* R\bm{\rho}_* \L[1].
}
\]
However this case is simpler and it is well known that there is an isomorphism  $
 R\bm{\rho}_* \omega_{\bm{\rho}} 
 ~ \cong ~
 \big[\bm{\rho}_* \omega_{\bm{\rho}} \rightarrow \O_{\Mbar^{_{1/r, \mathtt{a}}}_{g,l}}\big]$
 where each term in the complex is a vector bundle. The calculation now proceeds in the same way as the previous case. 
\end{proof}
\mbox{}

\vspace{0.5cm}
\begin{lemma}\label{distinguished_triangle_for_flag_contributions}
The contributions from flag nodes have no fixed part.
\begin{enumerate}
\item In the case $g=0$ and $l=2$, there is a single flag node and 
\[
e( \calF_{\Theta}) = 
\left\{
\begin{array}{cl}
\tfrac{1}{r}, & \mbox{if $r\mid\part_1$; } \vspace{0.1cm}\\
\tfrac{1}{t}, & \mbox{if $r\nmid\part_1$.}
\end{array}
\right.
\]
\item In the case $3g-3+l\geq0$, the moving part has equivariant Euler class
\[
e\big(\mbox{$\bigoplus_{i}$} \calF_{\Theta_i}\big)
= 
t^{-l} \left(\tfrac{t}{r}\right)^{\sum_{i:\, r|\part_i} 1}.
\]
\end{enumerate}
\end{lemma}
\begin{proof}
This proof is similar to the proof of \ref{distinguished_triangle_for_vertex_contributions}. Consider the case where we have $3g-3+l\geq0$. There is the following diagram which is a restriction of (\ref{nu_pot_restriction_diagram}) to $\Theta_i$ (with appropriately chosen labels):
\[
\xymatrix@R=1.5em{
\fixedDR \ar@{=}[d] & \Theta_i \ar[l]_{\bm{\gamma}_{ \Theta_i}} \ar@{=}[d] \ar[r]^{\mathfrak{f}_{ \Theta_i}~~} & \mathsf{Tot}_{\Theta_i} \bm{\vartheta}_i^* \L \ar[r]^{~~\bm{\beta}_{ \Theta_i}}  \ar[d]^{\bm{\tau}_{ \Theta_i}} & \Theta_i \ar@{=}[d]\\
\fixedDR & \Theta_i \ar[l]_{\bm{\gamma}_{ \Theta_i}}  \ar[r]^{\mathfrak{f}'_{ \Theta_i}~~} & \mathsf{Tot}_{\Theta_i} \bm{\vartheta}_i^* \L^r \ar[r]^{~~\bm{\alpha}_{ \Theta_i}}  & \Theta_i \\
}
\]
With this notation we have an isomorphism
\[
\calF_{\Theta_i} = R(\bm{\gamma}_{\Theta_i})_*  L \bm{\vartheta}_i^* L\mathfrak{f}^*  \eT_{\bm{\tau}_{\fixedDR}}
\cong
R(\bm{\gamma}_{\Theta_i})_*  L\mathfrak{f}_{\Theta_i}^*  \eT_{\bm{\tau}_{\Theta_i}}
\]
and the canonical distinguished triangle on $\Theta_i$ arising from $\bm{\tau}_{\Theta_i}$
\[
\xymatrix{
L\mathfrak{f}_{\Theta_i}^*  \eT_{\bm{\tau}_{\Theta_i}}  \ar[r] & \mathfrak{f}_{\Theta_i}^*  \eT_{\bm{\beta}_{\Theta_i}} \ar[r] & \mathfrak{f}_{\Theta_i}^* \bm{\tau}_{\Theta_i}^* \eT_{\bm{\alpha}_{\Theta_i}} \ar[r]&  L\mathfrak{f}_{\Theta_i}^*  \eT_{\bm{\tau}_{\Theta_i}}[1].
}
\]\\
$\Theta_i$ is a $\C^*$-fixed locus (as in the proof of \ref{distinguished_triangle_for_vertex_contributions}) and we have equivariant isomorphisms
\begin{align*}
R(\bm{\gamma}_{\Theta_i})_* \mathfrak{f}_{\Theta_i}^*  \eT_{\bm{\beta}_{\Theta_i}}
 &\cong
\calV_{c_1}   \otimes R(\bm{\gamma}_{\Theta_i})_* \bm{\vartheta}_i^* \L  
\hspace{0.5cm}\mbox{and}\\
R(\bm{\gamma}_{\Theta_i})_*\mathfrak{f}_{\Theta_i}^* \bm{\tau}_{\Theta_i}^* \eT_{\bm{\alpha}_{\Theta_i}}
&\cong
\calV_{c_2}   \otimes R(\bm{\gamma}_{\Theta_i})_* \bm{\vartheta}_i^* \L^r
\end{align*}
for some weights $c_1,c_2\in\mathbb{Q}$. The weights are calculated to be $\tfrac{1}{r}$ and $1$ respectively by using the same method as the proof of \ref{distinguished_triangle_for_vertex_contributions} . The result follows from lemma \ref{direct_image_of_line_bundles_at_flag_nodes} which shows the Chern polynomials are $c_s\Big(R(\bm{\gamma}_{\Theta_i})_* \bm{\vartheta}_i^* \L^r \Big) = 1$ and 
\[
c_s\Big( R(\bm{\gamma}_{\Theta_i})_* \bm{\vartheta}_i^* \L \Big)  = \left\{\begin{array}{ll}
1 & \mbox{if $r|\part_i$;} \\
0 & \mbox{otherwise.} \\
\end{array}\right.\]\\
In the case $(g,l)=(0,2)$, the proof is the same except there is only one node connecting two components where the stable map is non-constant.
\end{proof}
\mbox{}

\begin{lemma}\label{edge_contributions_lemma}
The fixed part of $\calF_i$ is zero. In the cases $3g-3+l\geq 0$ and $(g,l)=(0,2)$ the moving part has equivariant Euler class
\[
e\big(\calF_i^{\,\mov}\big) = 
\left( \frac{\left\lfloor \frac{\part_i}{r} \right\rfloor!}{\part_i^{\left\lfloor \frac{\part_i}{r} \right\rfloor}} t^{\left\lfloor \frac{\part_i}{r} \right\rfloor} \right)
\left(
 \frac{\part_i!}{\part_i^{\part_i}} t^{\part_i} 
\right)^{-1}.
\] 
For $(g,l)=(0,1)$ the moving part contribution is the above formula multiplied by $t$. 
\end{lemma}
\begin{proof}
Consider the following diagram which contains a restriction of the diagram (\ref{nu_pot_restriction_diagram}) to $\calC_i$ (with appropriately chosen labels):
\[
\xymatrix@R=1.5em{
\fixedDR \ar@{=}[d] & \calC_i \ar[l]_{\bm{\pi}_i} \ar@{=}[d] \ar[r]^{\mathfrak{f}_i~~} & \mathsf{Tot}_{\calC_i} \L_i \ar[r]^{~~\bm{\beta}_i}  \ar[d]^{\bm{\tau}_i} & \calC_i \ar@{=}[d]\\
\fixedDR & \calC_i \ar[l]_{\bm{\pi}_i}  \ar[r]^{\mathfrak{f}'_i~~} & \mathsf{Tot}_{\calC_i} \L_i^r\ar[r]^{~~\bm{\alpha}_i}  & \calC_i \\
}
\]
With this notation we have an isomorphism
\[
\calF_i = R{\bm{\pi}_i}_*  L \bm{\iota}_i^* L\mathfrak{f}^*  \eT_{\bm{\tau}_{\fixedDR}}
\cong
R{\bm{\pi}_i}_*  L\mathfrak{f}_i^*  \eT_{\bm{\tau}_i}
\]
and the canonical distinguished triangle on $\fixedDR$ arising from $\bm{\tau}_i$:
\[
\xymatrix{
\calF_{i} \ar[r] & R{\bm{\pi}_{i}}_* \mathfrak{f}_i^*  \eT_{\bm{\beta}_i} \ar[r] & R{\bm{\pi}_{i}}_* {\mathfrak{f}'_i}^* \eT_{\bm{\alpha}_i} \ar[r]&  \calF_{i}[1].
}
\]
In lemma \ref{edge_bundles_trivial_without_linearisations} it is shown that, after forgetting the $\C^*$-linearisation, $R{\bm{\pi}_{i}}_* \mathfrak{f}_i^*  \eT_{\bm{\beta}_i}$ and $R{\bm{\pi}_{i}}_* {\mathfrak{f}'_i}^* \eT_{\bm{\alpha}_i}$ are trivial bundles. Hence we may calculate their weights by considering a geometric point in $\fixedDR$. \\

\vspace{0.0cm}
Without loss of generality we consider a geometric point $\bullet \hookrightarrow \fixedDR$ with $\L|_{\bullet} = L$. Then we can form the following diagrams 
\[
\begin{array}{c}
\xymatrix@=1.8em{
\bullet \ar[d] & \P^1[r] \ar[l]_{\mathfrak{p}} \ar[d] \ar[r]^{s_{\mathfrak{b}}~~} & \mathsf{Tot}_{\P^1[r]} L\ar[r]^{~~\mathfrak{b}}  \ar[d] &  \P^1[r]  \ar[d]\\ 
\fixedDR & \calC_i \ar[l]_{\bm{\pi}_i}  \ar[r]^{\mathfrak{f}_i~~} & \mathsf{Tot}_{\calC_i} \L_i\ar[r]^{~~\bm{\beta}_i}  & \calC_i \\
}
\end{array}
\hspace{0.3cm}
\begin{array}{c}
\xymatrix@=1.8em{
\bullet \ar[d] & \P^1[r] \ar[l]_{\mathfrak{p}} \ar[d] \ar[r]^{s_{\mathfrak{a}}~~} & \mathsf{Tot}_{\P^1[r]} L^r\ar[r]^{~~\mathfrak{a}}  \ar[d] &  \P^1[r]  \ar[d]\\ 
\fixedDR & \calC_i \ar[l]_{\bm{\pi}_i}  \ar[r]^{\mathfrak{f}'_i~~} & \mathsf{Tot}_{\calC_i} \L_i^r\ar[r]^{~~\bm{\alpha}_i}  & \calC_i \\
}
\end{array}
\]
by taking the left and right squares of each diagram to be cartesian. (Here $\P^1[r]$ is the standard notation for $\P^1$ with a $r$-orbifold point at $0$.) These give isomorphisms 
\[
\big(R{\bm{\pi}_{i}}_* \mathfrak{f}_i^*  \eT_{\bm{\beta}_i} \big)\big|_{\bullet}
\cong 
R{\mathfrak{p}}_* s_{\mathfrak{b}}^*  \eT_{\mathfrak{b}} 
\hspace{0.5cm}
\mbox{and}
\hspace{0.5cm}
\big(R{\bm{\pi}_{i}}_* \mathfrak{f}_i^* \bm{\tau}_i^* \eT_{\bm{\alpha}_i}\big)\big|_{\bullet}
\cong 
R{\mathfrak{p}}_* s_{\mathfrak{a}}^*  \eT_{\mathfrak{a}} .
\]\\

\vspace{-0.4cm}
We also have a natural morphism $\gamma^* \gamma_*L\rightarrow L$ which induces a morphism on the total spaces $\mathfrak{h}: \mathsf{Tot}_{\P^1[r]} L\rightarrow  \mathsf{Tot}_{\P^1[r]} \gamma^* \gamma_*L$ as well as inducing an isomorphism on global sections $H^0(\P^1[r], \gamma^* \gamma_*L) \cong H^0(\P^1[r],L)$. Hence there exist a commuting diagram of the following form:
\[
\begin{array}{c}
\xymatrix@R=1.8em{
\mathsf{Tot}_{\P^1[r]} L  \ar[d]_{\mathfrak{b}} &  \mathsf{Tot}_{\P^1[r]} \gamma^* \gamma_*L \ar[d]_{\mathfrak{b}'} \ar[l]^{\mathfrak{h}\hspace{1em}} \\
\P^1[r] \ar@{=}[r]\ar@/_/[u]_{s_{\mathfrak{b}}}  & \P^1[r]\ar@/_/[u]_{s_{\mathfrak{b}'}} 
}
\end{array}
\]\\

\vspace{-0.2cm}
The morphism $\mathfrak{h}$, of the total spaces, induces the following distinguished triangle
\[
\xymatrix{
R{\mathfrak{p}}_* s_{\mathfrak{b}'}^*\eT_{\mathfrak{h}} \ar[r]& R{\mathfrak{p}}_* s_{\mathfrak{b}'}^*  \eT_{\mathfrak{b}'} \ar[r] & R{\mathfrak{p}}_* s_{\mathfrak{b}}^*  \eT_{\mathfrak{b}} \ar[r]& R{\mathfrak{p}}_* s_{\mathfrak{b}'}^*\eT_{\mathfrak{h}} [1].
}
\]
After forgetting the $\C^*$-linearisation the morphism $R{\mathfrak{p}}_* s_{\mathfrak{b}'}^*  \eT_{\mathfrak{b}'}\rightarrow R{\mathfrak{p}}_* s_{\mathfrak{b}}^*  \eT_{\mathfrak{b}}$ becomes the isomorphism $H^0(\P^1[r], \gamma^* \gamma_*L) \cong H^0(\P^1[r],L)$ which shows that $R{\mathfrak{p}}_* s_{\mathfrak{b}'}^*\mathbb{T}_{\mathfrak{h}}\cong 0$. So we  must have $R{\mathfrak{p}}_* s_{\mathfrak{b}'}^*\eT_{\mathfrak{h}}\cong 0$, showing $R{\mathfrak{p}}_* s_{\mathfrak{b}'}^*  \eT_{\mathfrak{b}'}\rightarrow R{\mathfrak{p}}_* s_{\mathfrak{b}}^*  \eT_{\mathfrak{b}}$ is an isomorphism. 

In the cases $3g-3+l\geq 0$ and $(g,l)=(0,2)$, consider the morphism $\gamma: \P^1[r]\rightarrow \P^1$ forgetting the stack structure. We can form the following diagrams where the top squares are cartesian: 
\[
\begin{array}{c}
\xymatrix@R=1.8em{
\mathsf{Tot}_{\P^1[r]} \gamma^* \gamma_*L \ar[r] \ar[d]_{\mathfrak{b}'} &  \mathsf{Tot}_{\P^1} \O_{\P^1}\hspace{-0.25em}\left( \left\lfloor \frac{\part_i}{r}\right\rfloor \right)\ar[d]_{\overline{\mathfrak{b}'}} \\
\P^1[r] \ar[r]^{\gamma}\ar[d]_{\mathfrak{p}}\ar@/_/[u]_{s_{\mathfrak{b}'}}  & \P^1\ar[d]_{\overline{\mathfrak{p}}} \ar@/_/[u]_{\overline{s}_{\mathfrak{b}'}} \\
\bullet \ar@{=}[r]  &  \bullet
}
\end{array}
\hspace{1cm}
\begin{array}{c}
\xymatrix@R=1.8em{
\mathsf{Tot}_{\P^1[r]} L^r \ar[r] \ar[d]_{\mathfrak{a}} &  \mathsf{Tot}_{\P^1} \O_{\P^1} \hspace{-0.25em}\left( \part_i \right) \ar[d]_{\overline{\mathfrak{a}}} \\
\P^1[r] \ar[r]^{\gamma}\ar[d]_{\mathfrak{p}}\ar@/_/[u]_{s_{\mathfrak{a}}}  & \P^1\ar[d]_{\overline{\mathfrak{p}}} \ar@/_/[u]_{\overline{s}_{\mathfrak{a}}} \\
\bullet \ar@{=}[r]  &  \bullet
}
\end{array}
\]\\

\vspace{-0.2cm}
Since $\gamma$ is flat we have isomorphisms $\eT_{\mathfrak{b}'} \cong \gamma^* \eT_{\overline{\mathfrak{b}'}}$ and $\eT_{\mathfrak{a}} \cong \gamma^* \eT_{\overline{\mathfrak{a}}}$. Furthermore these induce isomorphisms 
\[
\big(R{\bm{\pi}_{i}}_* \mathfrak{f}_i^*  \eT_{\bm{\beta}_i} \big)\big|_{\bullet}
\cong 
R{\overline{\mathfrak{p}}}_* \overline{s}_{\mathfrak{b}'}^*  \eT_{\overline{\mathfrak{b}'}} 
\hspace{0.5cm}
\mbox{and}
\hspace{0.5cm}
\big(R{\bm{\pi}_{i}}_* \mathfrak{f}_i^* \bm{\tau}_i^* \eT_{\bm{\alpha}_i}\big)\big|_{\bullet}
\cong 
R{\overline{\mathfrak{p}}}_* \overline{s}_{\mathfrak{a}}^*  \eT_{\overline{\mathfrak{a}}} .
\]
Here, $\C^*$ acts on $\P^1$ by $c\cdot [x_0 :x_1] = [x_0 : c^{1/\part_i} x_1]$. Hence we have:
\begin{enumerate}
\item $R{\overline{\mathfrak{p}}}_* \overline{s}_{\mathfrak{b}'}^*  \eT_{\overline{\mathfrak{b}'}} $ is the sum of trivial bundles with weights $ 0, \frac{t}{\part_i},2\frac{t}{\part_i}, \ldots,  \left\lfloor \frac{\part_i}{r} \right\rfloor\frac{ t}{\part_i}$.
\item $R{\overline{\mathfrak{p}}}_* \overline{s}_{\mathfrak{a}}^*  \eT_{\overline{\mathfrak{a}}}$ is the sum of trivial bundles with weights $ 0, \frac{t}{\part_i},2\frac{t}{\part_i}, \ldots,  \part_i\frac{ t}{\part_i}$.
\end{enumerate}
Taking equivariant Euler classes of the parts with non-zero weights completes the proof for the moving parts. The case $(g,l)=(0,1)$ is the same except the bundles $\O_{\P^1}\hspace{-0.25em}\left( \left\lfloor \frac{\part_i}{r}\right\rfloor \right)$ and $\O_{\P^1}\hspace{-0.25em}\left( \part_i \right)$ and replaced with $\O_{\P^1}\hspace{-0.25em}\left( \left\lfloor \frac{\part_i-1}{r}\right\rfloor \right)$ and $\O_{\P^1}\hspace{-0.25em}\left( \part_i-1 \right)$. \\

Lastly, the restriction of the morphism $R{\bm{\pi}_{i}}_* \mathfrak{f}_i^*  \eT_{\bm{\beta}_i} \longrightarrow R{\bm{\pi}_{i}}_* \mathfrak{f}_i^* \bm{\tau}_i^* \eT_{\bm{\alpha}_i}$ to the weight zero part is a nowhere-zero morphism between trivial bundles. Hence it is an isomorphism. The desired result now follows by taking equivariant Euler classes.
\end{proof}
\mbox{}

\subsection{Results from the Analysis of the Fixed and Moving Parts}

\begin{corollary} \label{final_localisation_formula_result}
The restriction $L\fixedIncDR^*\F_{\M^{1/r}}$ has fixed part  
\[
\big(L\fixedIncDR^*\F_{\M^{1/r}}\big)^{\fix} \cong \eT_{\fixedDR} \cong (\Omega_{\fixedDR})^\vee
\] 
which is locally free. Moreover, setting $\mathrm{N}^{\mathrm{vir}}_{\fixedDR}:= (L\fixedIncDR^*\F_{\M^{1/r}})^{\mov}$ we have
\begin{enumerate}
\item If $g=0$ and $l=1$ then
\begin{flalign*}
\hspace{4.5em}
\frac{1}{e\big(\mathrm{N}^{\mathrm{vir}}_{\fixedDR}\big) } 
&= 
r^{-1}
\left(\tfrac{t}{r}\right)^{-m}  
\left(  \frac{\part_1 \left(\frac{\part_1}{r}\right)^{\left\lfloor \frac{\part_1}{r} \right\rfloor}}{\left\lfloor \frac{\part_1}{r} \right\rfloor!}\right) 
\frac{1}{\part_1^2}.&
\end{flalign*}
\item If $g=0$ and $l=2$ then 
\begin{flalign*}
\hspace{4.5em}
\frac{1}{e\big(\mathrm{N}^{\mathrm{vir}}_{\fixedDR}\big) } 
&= 
\left(\tfrac{t}{r}\right)^{ -m}  
\left( \prod_{i =1}^2 \frac{\part_i \left(\frac{\part_i}{r}\right)^{\left\lfloor \frac{\part_i}{r} \right\rfloor}}{\left\lfloor \frac{\part_i}{r} \right\rfloor!}\right) 
\frac{1}{\left( \part_1 + \part_2 \right)  }. &
\end{flalign*}
\item If $3g-3+l\geq 0$ then
\begin{flalign*}
\hspace{4.5em}
\frac{1}{e\big(\mathrm{N}^{\mathrm{vir}}_{\fixedDR}\big) } 
&= 
r^{l+2g-2}
\left(\tfrac{t}{r}\right)^{3g-3+l -m}  
\left( \prod_{i =1}^l \frac{\part_i \left(\frac{\part_i}{r}\right)^{\left\lfloor \frac{\part_i}{r} \right\rfloor}}{\left\lfloor \frac{\part_i}{r} \right\rfloor!}\right) 
\frac{c_{\frac{r}{t}}(-\bm{b}^*R\bm{\rho}_* \L )}{\prod_{i=1}^{l}(1 - \frac{\part_i}{r}\psi_i)}.&
\end{flalign*}
where $c_{s}$ is the Chern polynomial, $\bm{b}: \fixedDR \rightarrow \Mbar^{_{1/r, \mathtt{a}}}_{g,l}$ is from corollary \ref{degree_of_map_to_spin_curves}, while $\bm{\rho}$ and $\L$ are the the universal curve  and $r$-th root bundle for $\Mbar^{_{1/r, \mathtt{a}}}_{g,l}$.
\end{enumerate}
\end{corollary}
\begin{proof}
Take the distinguished triangle of perfect obstruction theories from corollary \ref{POT_for_M1/r_main_corollary}. Restricting to $\fixedDR$ and taking either either the fixed or the moving part gives
\[
\xymatrix@C=1.25em{
 \big(L\fixedIncDR^* \eF_{\ForMapDRtoRSM}\big)^{\fixmov} \ar[r] &  \big(L\fixedIncDR^* \eF_{\M^{1/r}} \big)^{\fixmov}  \ar[r] & \big(L\fixedIncDR^*L\ForMapDRtoRSM^*\eF_{\M^r}\big)^{\fixmov}   \ar[r]& \big(L\fixedIncDR^* \eF_{\ForMapDRtoRSM}\big)^{\fixmov}  [1].
}
\]
We have from lemmas \ref{distinguished_triangle_for_vertex_contributions}, \ref{distinguished_triangle_for_flag_contributions} and \ref{edge_contributions_lemma} that $ (L\fixedIncDR^* \eF_{\ForMapDRtoRSM})^{\fix} \cong 0$. So we have an isomorphim $ (L\fixedIncDR^* \eF_{\M^{1/r}} )^{\fix} \cong (L\fixedIncDR^*L\ForMapDRtoRSM^*\F_{\M^{r}} )^{\fix}$ which is shown to be locally free in lemma \ref{fixed_moving_parts_pullback_of_RSM}. So, by theorem \ref{localisation_theorem_GraberPand} the fixed part $(L\fixedIncDR^*L\ForMapDRtoRSM^*\F_{\M^{r}})^{\fix}$ is a perfect obstruction theory for $\fixedDR$.   The isomorphism $(L\fixedIncDR^*\F_{\M^{1/r}} )^{\fix} \cong (\Omega_{\fixedDR})^\vee$ follows from \cite[Prop. 5.5]{Intrinsic}. \\

Taking the moving part of the above distinguished triangle and using the notation from \ref{fixed_moving_nu_morphism_notation} we have
\begin{align*}
\frac{1}{e\big(\mathrm{N}^{\mathrm{vir}}_{\fixedDR}\big) } 
&=
\frac{1}{e\Big(\big(L\fixedIncDR^* L\ForMapDRtoRSM^*\F_{\M}\big)^{\mov}\Big) e\Big(\big(L\fixedIncDR^*\F_{\ForMapDRtoRSM}\big)^{\mov}\Big)}\\
&=
\frac{\prod_i e\Big(\calF_{\Theta_i}^{\,\mov}\Big) }{e\Big(\big(L\fixedIncDR^* L\ForMapDRtoRSM^*\F_{\M}\big)^{\mov}\Big) e\Big(\calF_{v}^{\,\mov}\Big) \prod_i e\Big(\calF_{i}^{\,\mov}\Big)}.
\end{align*}
These contributions are calculated in lemmas \ref{fixed_moving_parts_pullback_of_RSM}, \ref{distinguished_triangle_for_flag_contributions}, \ref{edge_contributions_lemma} and corollary \ref{vertex_contributions_corollary}. In the cases where $3g-3+l<0$ we have used:
\begin{enumerate}[leftmargin=1.5em]
\item[] \textit{For $(g,l)=(0,1)$:} $m= \frac{1}{r}(\part_1-1) = \left\lfloor \frac{\part_1}{r} \right\rfloor $.
\item[] \textit{For $(g,l)=(0,2)$:} $m= \frac{1}{r}(\part_1+\part_2) = 
\left\{\begin{array}{cl}
\left\lfloor \frac{\part_1}{r} \right\rfloor + \left\lfloor \frac{\part_2}{r} \right\rfloor,  & \mbox{if $r\mid \part_1$}; \vspace{0.1cm}\\
1+\left\lfloor \frac{\part_1}{r} \right\rfloor + \left\lfloor \frac{\part_2}{r} \right\rfloor, & \mbox{if $r\nmid \part_1$.}
\end{array}\right.
$
\end{enumerate}
In the cases where $3g-3+l\geq0$ we have used the fact that
\[
c_{s}\big((\bm{\rho}_* \omega_{\bm{\rho}})^\vee\big)  c_{s}\big(\bm{\rho}_* \omega_{\bm{\rho}}\big) =1,
\]
a proof of which can be found in \cite[Prop. 5.16]{ACG_Geometry}.
\end{proof}
\mbox{}

\begin{corollary}\label{cotangent_F1/r->F_is_zero}
The morphism of simple fixed loci $\ForMapDRtoRSMfixed: \fixedDR \rightarrow \fixedRSM$ which forgets the $r$-th root section is étale.
\end{corollary}
\begin{proof}
Consider the morphism of equivariant cotangent complexes $\eT_{\fixedDR} \rightarrow \ForMapDRtoRSMfixed^*\eT_{\fixedRSM}$ arising from $\ForMapDRtoRSMfixed: \fixedDR \rightarrow \fixedRSM$. Taking the fixed part gives a morphism $\eT_{\fixedDR} \rightarrow (\ForMapDRtoRSMfixed^*\eT_{\fixedRSM})^\fix$. Applying this to the 
morphism of perfect obstruction theories from corollary \ref{POT_for_M1/r_main_corollary} gives the left square in the following commuting diagram:
\[
\xymatrix@R=1.5em{
\eT_{\fixedDR}\ar[r] \ar[d]_{u_1} &  \big(  {\ForMapDRtoRSMfixed}^*  \eT_{\fixedRSM} \big)^\fix\ar[d]_{u_2} \ar[r]& {\ForMapDRtoRSMfixed}^*     \eT_{\fixedRSM}  \ar[d]_{u_3}\\
 \big(L\fixedIncDR^*\eF_{\M^{1/r}} \big)^\fix\ar[r]^{v_1} &  \big(  {\ForMapDRtoRSMfixed}^* L{(\fixedIncRSM)}^* \eF_{\M^r} \big)^\fix \ar[r]^{v_2}& {\ForMapDRtoRSMfixed}^* \big(  L{(\fixedIncRSM)}^* \eF_{\M^r} \big)^\fix
}
\]
The right square arises from the analysis of $L\fixedIncDR^*L\ForMapDRtoRSM^*\eF_{\M^r}\cong {\ForMapDRtoRSMfixed}^*  L{(\fixedIncRSM)}^* \eF_{\M^r} $ from lemma \ref{fixed_moving_parts_pullback_of_RSM} which also holds for the fixed part of $L{(\fixedIncRSM)}^* \eE_{\M^r}$. Indeed, this gives an isomorphism
\[
v_2:  \big(  {\ForMapDRtoRSMfixed}^* L{(\fixedIncRSM)}^* \eF_{\M^r} \big)^\fix
 \overset{\sim}{\longrightarrow}
 {\ForMapDRtoRSMfixed}^* \big(  L{(\fixedIncRSM)}^* \eF_{\M^r} \big)^\fix 
\]
Now, using both corollary \ref{final_localisation_formula_result} and \cite[Prop 5.5]{Intrinsic} we have that $\fixedDR$ and $\fixedRSM$ are smooth. Hence, $u_1$ and $u_3$ are isomorphisms.  Moreover, corollary \ref{final_localisation_formula_result} shows that $v_1$ is an isomorphism, showing that $\eT_{\fixedDR} \rightarrow \ForMapDRtoRSMfixed^*\eT_{\fixedRSM}$ is an isomorphism. 
\end{proof}
\mbox{}

\section{Proving Theorem \ref{rHurwitz_theorem} - The \texorpdfstring{$r$}{r}-ELSV formula} 

\subsection{Choice of Equivariant Lift}
We use arguments based on \cite[\S 4.2]{FantechiPand}. 

\begin{re}\label{Sym^m_P^m_identification}
Recall the \textit{branch-type morphism} from \cite{Leigh_Ram} discussed earlier in \ref{branch_type_morphism_re}. It is a morphism of stacks
\[
\mathtt{br} :\M^{\frac{1}{r}}  \longrightarrow \Sym^{m} \P^1 
\]
commuting with the branch morphism of \cite{FantechiPand} via the diagram
\[
\xymatrix@R=1.5em{
\M^{\frac{1}{r}} \ar[r]^{\mathtt{br}\hspace{0.75em}}\ar[d] &\Sym^{m} \P^1 \ar[d]^{\Delta}\\
 \M \ar[r]^{br} & \Sym^{rm} \P^1  
}
\]
where $\Delta$ is defined by $\sum_i x_i \mapsto \sum_i rx_i$. We make the identification $\Sym^{m} \P^1  \cong \P^m$ and extend the $\C^*$-action so that $\mathtt{br}$ is equivariant with $\C^*$ acting on $\P^m$ via
\[
c \cdot [ y_0: y_1 : \cdots : y_m]  = [y_0: c y_1 :\cdots: c^m y_m].
\]\\
\end{re}

\vspace{-0.4cm}
\begin{re}
Let $H$ be the hyperplane class in $\P^m$. Then, with the identification in \ref{Sym^m_P^m_identification} the class of a point in $\Sym^m\P^1$ corresponds to the class $c_1(\O(H))^m$. Hence, after a choice of equivariant lift for $\O(H)$, we may apply the localisation formula to the integral:
\[
\int_{\left[\M^{\frac{1}{r}}\right]^{\mathrm{vir}}} \mathtt{br}^*\big(c_1(\O(H))^m\big) .
\] \\
\end{re}

\vspace{-0.4cm}
\begin{re}\label{hyperplane_linearisations}
Let $y_0, \ldots, y_m$ be projective coordinates on $\P^m$ and define the point $h_n \in \P^m$ by the ideal 
\[
\big<\{y_0,\ldots ,y_m\} \setminus \{y_n\}\big>
\] 
Then we have $h_n \in (\P^m)^{\C^*}$ and corresponds to the point $[(m-k)\cdot(0) + k\cdot(\infty)]$ in $\Sym^{m}\P^1$. Moreover we know from the discussion in \ref{simple_non-simple_fixed_loci_of_M1/r} that
\[
\mathtt{br}\Big(\big(\M^{\frac{1}{r}}\big)^{\C*}\Big) \subseteq \{h_0, \ldots, h_m\}.
\]

Denote by $\calH_n$ the unique $\C^*$-linearisation of $\O(H)$ having weight $0$ at $h_n$. We apply the localisation formula to 
\[
\int_{\left[\M^{\frac{1}{r}}\right]^{\mathrm{vir}}} \mathtt{br}^*\left( \prod_{n=1}^m c_1(\calH_n )\right).
\] \\
\end{re}

\vspace{-0.4cm}
\begin{remark}
Different choices of the $\C^*$-linearisations of $\O(H)$ will lead to different \textit{equivariant} integrands. However, they will all have the same non-equivariant limit so we follow \cite[\S 4.2]{FantechiPand} and choose the equivariant lift with the simplest localisation. \\
\end{remark}

\subsection{Application of the Localisation Formula}

\begin{re}
Recall from \ref{simple_non-simple_fixed_loci_of_M1/r} that the fixed loci is decomposed as 
\[
\big(\M^{\frac{1}{r}}\big)^{\C^*} = \fixedDR \sqcup \mbox{$\bigsqcup_{n=1}^{m}$} \nonsimpDR_n
\]

where $\fixedDR := (\M^{\frac{1}{r}})^{\C^*} \cap \mathtt{br}^{-1}(h_0)$ and $\nonsimpDR_n := (\M^{\frac{1}{r}})^{\C^*} \cap \mathtt{br}^{-1}(h_n)$. 
\end{re}
\mbox{}

\begin{lemma}\label{weights_vanishing_pullback_classes}
Using the linearisations in \ref{hyperplane_linearisations} we have
\begin{enumerate}
\item $\big[\,\nonsimpDR_n \big]^{\mathrm{vir}} \cap \mathtt{br}^*\Big( \prod_{n=1}^m c_1(\calH_n )\Big) = 0$ 
\item $\big[\,\fixedDR\,\big]^{\mathrm{vir}} \cap \mathtt{br}^*\Big( \prod_{n=1}^m c_1(\calH_n )\Big) = \big[\,\fixedDR\,\big]^{\mathrm{vir}} \cdot m! t^{m}$ 
\end{enumerate}
\end{lemma}

\begin{proof}
For both, the proof is the same as that in \cite[\S 4.2]{FantechiPand}. For part 1 we have that $\mathtt{br}^* \calH_n$ restricted to $\nonsimpDR_n$   is the trivial bundle with trivial linearisation (i.e weight 0). Hence, the given intersection vanishes. For part 2 we have that $\mathtt{br}^* \calH_n$ restricted to $\fixedDR$ is the trivial bundle with weight $n$. 
\end{proof}
\mbox{}

\begin{corollary}[\textit{Theorem \ref{rHurwitz_theorem}}]
When $3g-3+l\geq 0$ there is an equality:
\[
\int_{\left[\M^{\frac{1}{r}}\right]^{\mathrm{vir}}} \mathtt{br}^*[ p_1+ \cdots + p_{m}]
=
 m!\, r^{m+ l+ 2g-2} 
 \left(
 \prod_{i=1}^{l} \frac{\left(\frac{\part_i}{r}\right)^{\lfloor \frac{\part_i}{r} \rfloor}}{\lfloor \frac{\part_i}{r} \rfloor!}
 \right)
~ \int_{\Mbar^{\frac{1}{r}, \mathtt{a}}_{g,l}} 
\frac{c(-R\bm{\rho}_* \mathcal{L})}{\prod_{j=1}^{l}(1 - \frac{\part_i}{r}\psi_j)}
\]
where $\bm{\rho}$ and $\L$ are the universal curve and $r$-th root of $\Mbar^{\frac{1}{r}, \mathtt{a}}_{g,l}$, while $\mathtt{a}= (a_1,\ldots, a_{l})$ is a vector with $a_i\in \{0,\ldots, r-1\}$ defined by $\part_i = \left\lfloor \frac{\part_i}{r} \right\rfloor r + (r-1 -a_i)$  and $l=l(\part)$.\\

In the special cases where $3g-3+l<0$ we interpret this formula by defining:
\begin{flalign*}
\hspace{1.5em}
(i)~~&\int_{\Mbar^{\frac{1}{r}, \mathtt{a}}_{0,1}} 
\frac{c(-R\bm{\rho}_* \mathcal{L})}{(1 - \frac{\part_1}{r}\psi_1)}
=
\frac{1}{\part_1^2}, ~~\mbox{when $r|(\part_1 -1)$ and $0$ otherwise},&\\
(ii)~~&\int_{\Mbar^{\frac{1}{r}, \mathtt{a}}_{0,2}} 
\frac{c(-R\bm{\rho}_* \mathcal{L})}{\prod_{j=1}^{2}(1 - \frac{\part_j}{r}\psi_j)}
=
\frac{1}{\part_1+\part_2},~~\mbox{when $r|(\part_1 +\part_2)$ and $0$ otherwise}.&
\end{flalign*}
\end{corollary}
\begin{proof}
Considering the equivariant intersection with lifts from \ref{hyperplane_linearisations} we have
\[
\int_{\left[\M^{\frac{1}{r}}\right]^{\mathrm{vir}}} \mathtt{br}^*\left( \prod_{n=1}^m c_1(\calH_n )\right) 
=
 \big[\fixedDR\big]^{\mathrm{vir}}  \cap \frac{m!\, t^m}{e\left(\mathrm{N}^{\mathrm{vir}}_{\fixedDR}\right)}
\] 
from lemma \ref{weights_vanishing_pullback_classes}. For $3g-3+l\geq 0$ we apply corollary \ref{final_localisation_formula_result} to show the intersection is equal to
\[
m!\, r^{m+l+2g-2}
\left(
 \prod_{i=1}^{l} \frac{\part_i \left(\frac{\part_i}{r}\right)^{\lfloor \frac{\part_i}{r} \rfloor}}{\lfloor \frac{\part_i}{r} \rfloor!}
 \right)
\left(\big[\fixedDR\big] \cap   \bm{b}^*\frac{ \left(\frac{t}{r}\right)^{3g-3+l}c_{\frac{r}{t}}\big([R\bm{\rho}_* \L ]^\vee\big)}{\prod_{i=1}^{l}\left(1 - \left(\frac{r}{t}\right)\frac{\part_i }{r}c_1(\bm{\sigma}_i^*\omega_{\bm{\rho}})\right)}
\right)
\]
where $\bm{b}:\fixedDR\rightarrow \Mbar^{\frac{1}{r}, \mathtt{a}}_{g,l(\part)}$ is the degree $(\part_1\cdots\part_l)^{-1}$ morphism defined in lemma \ref{degree_of_map_to_spin_curves}. The desired result follows from pushing forward via $\bm{b}$ and taking the non-equivariant limit.  For $3g-3+l<0$ we apply the special cases of corollary \ref{final_localisation_formula_result}. 
\end{proof}


\section{Proving Theorem \ref{POT_theorem} - Equivariant Perfect Obstruction Theory} \label{equi_pot_section}

\notationConvention\\

\subsection{Background on Equivariant Perfect Obstruction Theories}\label{background_equivariant_POT_subsection}
\begin{re}[\textit{Equivariant Cotangent Complex}]\label{equivariant_cotangent_complex}
For a stack $\calX$ with a $\C^*$-action we denote the derived category of equivariant coherent sheaves by  $\mathrm{D}^{\mathtt{e}}_b(\calX)$.  For an equivariant morphism $\bm{\eta}:\calX\rightarrow \calY$ there is an \textit{equivariant cotangent complex} $\eL_\eta \in \mathrm{D}^{\mathtt{e}}_b(\calX)$ (originally defined by Illusie in \cite[Ch. VII 2.2]{Illusie_II}) with the following properties:
\begin{enumerate}
\item Without the $\C^*$-linearisation, $\eL_\eta$ is the cotangent complex $\bbL_\eta \in \mathrm{D}_b(\calX)$.
\item If the following is a commuting diagram of equivariant morphisms 
\begin{align*}
\xymatrix@R=1.5em{
\calX \ar[r]^{\bm{\eta}}\ar[d]_{\bm{\zeta}} & \calY \ar[d]^{\bm{\zeta'}}\\
\calW \ar[r]^{\bm{\eta'}} & \calZ
} 
\end{align*}
then there is a natural morphism $L\bm{\eta}^*\eL_{\bm{\zeta'}} \rightarrow \eL_{\bm{\zeta}}$ in $\mathrm{D}^{\mathtt{e}}_b(\calX)$ which is an isomorphism when the diagram is cartesian and one of $\bm{\zeta'}$ or $\bm{\eta'}$ is flat. Moreover, after forgetting the $\C^*$-linearisations this morphism is the natural morphism $L\bm{\eta}^*\bbL_{\bm{\zeta'}} \rightarrow \bbL_{\bm{\zeta}}$ in $\mathrm{D}_b(\calX)$.
\item The equivariant cotangent complex is compatible with composition. That is, if the following is a commuting diagram of equivariant morphisms 
\begin{align*}
\xymatrix@R=1.5em{
\calX \ar[r]^{\bm{\eta}}\ar[d]_{\bm{\zeta}} & \calY \ar[d]^{\bm{\zeta'}}\ar[r]^{\bm{\varkappa}}& \calU \ar[d]^{\bm{\zeta''}}\\
\calW \ar[r]^{\bm{\eta'}} & \calZ \ar[r]^{\bm{\varkappa'}} & \calV
} 
\end{align*}
then there is the following commuting diagram in $\mathrm{D}^{\mathtt{e}}_b(\calX)$:
\begin{align*}
\xymatrix@=1em{
L(\bm{\varkappa}\circ \bm{\eta})^*\eL_{\bm{\zeta''}} \ar[rr] \ar[rd]&& \eL_{\zeta} \\
& L\bm{\eta}^* \eL_{\bm{\zeta'}} \ar[ru]&
} 
\end{align*}
\item If $\bm{\eta}:\calX\rightarrow \calY$ and $\bm{\zeta}:\calY\rightarrow \calZ$ are equivariant morphisms then there is a distinguished triangle in $\mathrm{D}^{\mathtt{e}}_b(\calX)$ given by
\begin{align*}
\xymatrix{
L\bm{\eta}^*\eL_{\bm{\zeta}} \ar[r] & \eL_{\bm{\zeta}\circ \bm{\eta}}  \ar[r] & \eL_{\bm{\eta}}  \ar[r]& Lg^*\eL_{\bm{\zeta}}[1].
}
\end{align*}
Again, after forgetting the $\C^*$-linearisations this is the usual distinguished triangle of the cotangent complex. \\
\end{enumerate}
\end{re}

\begin{re}\label{equivariant_compatibility_square_re}
If we have commuting diagrams of equivariant morphisms 
\begin{align*}
\begin{array}{c}
\xymatrix@R=1.5em{
\calX \ar[r]^{\bm{\eta}}\ar[d]_{\bm{\zeta}} & \calY \ar[d]^{\bm{\zeta'}}\ar@{=}[r]& \calY \ar[d]\\
\calW \ar[r]^{\bm{\eta'}} & \calZ \ar[r] & \bullet
} 
\end{array}
\hspace{1cm}
\mbox{and}
\hspace{1cm}
\begin{array}{c}
\xymatrix@R=1.5em{
\calX \ar@{=}[r] \ar[d]_{\bm{\zeta}} & \calX \ar[d] \ar[r]^{\bm{\eta}}& \calY \ar[d]\\
\calW \ar[r] & \bullet \ar@{=}[r] & \bullet
} 
\end{array}
\end{align*}
then we can use the properties of \ref{equivariant_cotangent_complex} to show there is the following commuting diagram of morphisms in $\mathrm{D}^{\mathtt{e}}_b(\calX)$:
\[
\xymatrix@R=1.5em{
L\bm{\eta}^*\eL_{\calY} \ar[r]\ar[d] &  L\bm{\eta}^*\eL_{\bm{\zeta'}} \ar[d]\\
\eL_{\calX}  \ar[r] & \eL_{\bm{\zeta}}
}
\]\\
\end{re}

\vspace{-0.7cm}
\begin{definition}[\textit{Equivariant Perfect Obstruction Theory, \cite{GraberPand}}]
Let $\bm{\eta}:\calX\rightarrow \calY$ be an equivariant morphism of stacks. An equivariant perfect obstruction theory for $\bm{\eta}$ is a morphism
\[
\phi^{\mathtt{e}}_{\bm{\eta}}:\eE_{\bm{\eta}} \longrightarrow \eL_{\bm{\eta}} 
\]
in $\mathrm{D}^{\mathtt{e}}_b(\calX)$ such that the associated morphism $\phi_{\bm{\eta}}:\E_{\bm{\eta}} \longrightarrow \bbL_{\bm{\eta}}$ in  $\mathrm{D}_b(\calX)$ (constructed by forgetting the $\C^*$-linearisations) is a perfect obstruction theory for $\bm{\eta}$.\\
\end{definition}

\subsection{Equivariant Perfect Obstruction Theory For \texorpdfstring{$\M$}{M}}

\begin{re}[\textit{Equivariant Perfect Relative Obstruction Theory for $\ForMap$}] \label{equi_POT_for_ForMap_re}
Denote the morphisms which forget all data but the source and target families by
\[
\ForMapSM: \M \rightarrow \frakM \times \calT 
\hspace{0.5cm}
\mbox{and}
\hspace{0.5cm}
\ForMap: \M^r \rightarrow \frakM^r \times \calT 
\]
where $\ForMapSM$ is discussed more in \ref{universal_objects_of_moduli_relative_stable_maps}. These morphism fit in the diagram
\begin{align}
\begin{array}{c}
\xymatrix@R=1.5em{
\M^r \ar[r]\ar[d]^{\ForMap} &\widetilde{\M} \ar[r]\ar[d]&  \ar[d]^{\ForMapSM}  \M\\
\frakM^r \times \calT \ar@{=}[r]^{} &\frakM^r \times \calT \ar[r] & \frakM \times \calT
} 
\end{array}\label{stable_maps_forgetful_maps_diagram}
\end{align}
where $\widetilde{\M} := \M  \times_{( \frakM\times \calT)} (\frakM^r \times \calT)$. We note that the morphism  $\frakM^r \times \calT \rightarrow \frakM\times \calT$ is flat and the morphism $\M^r\rightarrow\widetilde{\M}$ is étale (as described in \cite[Lem. 2.1.2]{Leigh_Ram}).\\

We give $\frakM\times \calT$ and $\frakM^r \times \calT$ the trivial $\C^*$-action (as discussed in \ref{action_on_M1/r_and_related_spaces}) which makes every morphism in the diagram of (\ref{stable_maps_forgetful_maps_diagram}) a $\C^*$-equivariant morphism. There is an equivariant perfect relative obstruction theory $\ephi_{{\ForMapSM}}: \eE_{\ForMapSM} \rightarrow \eL_{\ForMapSM}$ for $\ForMapSM$ constructed in \cite{li2,GraVakil_Loc}. This pulls back via the top row of (\ref{stable_maps_forgetful_maps_diagram}) to give an equivariant perfect relative obstruction theory $\ephi_{{\ForMap}}: \eE_{\ForMap} \rightarrow \eL_{\ForMap}$ for $\ForMap$. \\
\end{re}

\begin{re}[\textit{Equivariant Perfect Obstruction Theory for $\M^r$}] \label{equi_POT_for_M_re}
Using the properties of the equivariant cotangent complex given in \ref{equivariant_cotangent_complex} and the construction of \ref{equi_POT_for_ForMap_re} with \cite[\S2.8]{GraVakil_Loc} there is a natural equivariant perfect obstruction theory for $\M^r$ which fits into the following commuting diagram with distinguished triangles for rows:
\begin{align}
\begin{array}{c}
\xymatrix@R=1.5em{
  {\ForMap}^* \eL_{\frakM^r \times \calT}   \ar[r] \ar@{=}[d] & \eE_{\M^r}  \ar[r] \ar[d]^{\ephi_{\M^r} } & \eE_{\ForMap} \ar[d]^{\ephi_{{\ForMap}}} \ar[r] &     {\ForMap}^* \eL_{\frakM^r \times \calT}[1]     \ar@{=}[d]\\
 {\ForMap}^* \eL_{\frakM^r \times \calT}  \ar[r] & \eL_{\M^r}  \ar[r] & \eL_{{\ForMap}} \ar[r] &   {\ForMap}^* \eL_{\frakM^r \times \calT}[1]   
} 
\end{array}
\label{equi_POT_commuting_diagram_M}
\end{align}
Moreover, this construction gives the following commutative diagram relating $\eE_{\M}$ and $\eE_{\M^r}$ via the natural forgetful morphism $\ForMapRSMtoSM: \M^r\rightarrow \M$:
\begin{align}
\begin{array}{c}
\xymatrix@R=1.5em{
  {\ForMapRSMtoSM}^* \eE_{\M}  \ar[r] \ar[d]\ar[d]^{{\ForMapRSMtoSM}^*\ephi_{\M}} & \eE_{\M^r}  \ar[r] \ar[d]^{\ephi_{\M^r} } & \eL_ {\ForMapRSMtoSM} \ar@{=}[d]\ar[r] &      {\ForMapRSMtoSM}^* \eE_{\M}  [1]    \ar[d]^{{\ForMapRSMtoSM}^*\ephi_{\M}[1]}\\
  {\ForMapRSMtoSM}^* \eL_{\M}  \ar[r] & \eL_{\M^r}  \ar[r] & \eL_ {\ForMapRSMtoSM}  \ar[r] &     {\ForMapRSMtoSM}^* \eL_{\M}        [1]   
} 
\end{array}
\label{Mbar_M_POT_comparision_diagram}
\end{align}\\
\end{re}

\subsection{Equivariant Perfect Relative Obstruction Theory for \texorpdfstring{$\ForMapDRtoRSM$}{nu}} \label{equi_pot_for_nu_subsection}

\begin{re}[\textit{Equivariant Perfect Relative Obstruction Theories for Total Spaces}]\label{equi_pot_for_total_spaces}
Let $\calF$ be a line bundle on $\calC^r$ and recall the spaces $\mathsf{Tot}\,\bm{\pi}_* \calF$ and $\mathsf{Tot}\,\calF$ from 
definition \ref{totalspace_definition}. If $\bm{\psi}: \calC_{\mathsf{Tot}\,\bm{\pi}_*\calF} \rightarrow \mathsf{Tot}\,\bm{\pi}_* \calF$ is the universal curve for $\mathsf{Tot}\,\bm{\pi}_* \calF$, then there is a natural evaluation morphism $\mathfrak{e} : \calC_{\mathsf{Tot}\,\bm{\pi}_*\calF} \rightarrow \mathsf{Tot}\,\calF$ defined by
 \begin{align*}
					\mathfrak{e} ~:~ \Big(~
					\zeta
					,~~
					\sigma: \O_{C} \longrightarrow \calF_{\zeta} ~\Big)
					\longmapsto
					\Big(~
					\zeta
					,~~
					{\gamma_{\Theta}}_*\vartheta^*\sigma: \O_{S} \longrightarrow {\gamma_{\Theta}}_*\vartheta^*\calF_{\zeta} ~\Big)
\end{align*}
which leads to the following commutative diagram where the left square is cartesian and all morphisms are equivariant:
\begin{align}\label{main_cone_squares}\begin{array}{c}
\xymatrix@R=1.5em{
\mathsf{Tot}\,\bm{\pi}_*\calF \ar[d]^{\bm{\varepsilon}} & \calC_{\mathsf{Tot}\,\bm{\pi}_*\calF} \ar[l]_{\bm{\psi}} \ar[r]^{\mathfrak{e}}\ar[d]^{\bm{\widehat{\varepsilon}}}& \mathsf{Tot}\,\calF \ar[d]^{\bm{\check{\varepsilon}}} \\
\M^r & \calC^r \ar[l]^{\bm{\pi}} \ar@{=}[r] &\calC^r
}\end{array}
\end{align}\\

\vspace{-0.2cm}
Since $\bm{\psi}$ is Gorenstein and Deligne-Mumford-type, there is a natural equivariant morphism arising from Grothendieck duality for Deligne-Mumford stacks
\[
R\bm{\psi}_* (\bm{\psi}^* \eL_{\bm{\varepsilon}} \otimes \omega_{\bm{\psi}} )[1] \longrightarrow \eL_{\bm{\varepsilon}}.
\]
Combining this with the inverse of the isomorphism $\bm{\psi}_* \eL_{\bm{\varepsilon}} \overset{\sim}{\rightarrow} \eL_{\widehat{\bm{\varepsilon}}}$ (arising from $\bm{\pi}$ being flat) and the morphism $L\mathfrak{e}^*\eL_{\check{\bm{\varepsilon}}} \rightarrow \eL_{\widehat{\bm{\varepsilon}}}$ gives the following morphism in $\mathrm{D}_b^{\mathtt{e}}(\M^r)$
\[
R\bm{\psi}_* (L\mathfrak{e}^*\eL_{\check{\bm{\varepsilon}}} \otimes \omega_{\bm{\psi}} )[1] \longrightarrow \eL_{\bm{\varepsilon}}.
\]
We denote this morphism by $\ephi_{\bm{\varepsilon}}: \eE_{\bm{\varepsilon}} \rightarrow \eL_{\bm{\varepsilon}}$. The associated morphism $\phi_{\bm{\varepsilon}}$ in $\mathrm{D}_b(\M^r)$ was shown to be a perfect relative obstruction theory for $\bm{\varepsilon}$ in \cite[Prop. 2.5]{ChangLi}.\\
\end{re}

\vspace{0.01cm}
\begin{re}[\textit{Equivariant Relative Obstruction Theory for \texorpdfstring{$\bm{\tau}$}{tau}}]
Consider the following diagram of equivariant morphisms by combining the $r$-th power maps from \ref{rth_power_maps_on_M} with the construction in \ref{equi_pot_for_total_spaces}:
\begin{align}
\begin{tikzpicture}[baseline=(current  bounding  box.center), node distance=3cm, auto,
  f->/.style={->,preaction={draw=white, -,line width=3pt}},
  d/.style={double distance=1pt},
  fd/.style={double distance=1pt,preaction={draw=white, -,line width=3pt}}]
  \node (11) {$\mathsf{Tot}\,\bm{\pi}_*\L^r$};
  \node [right of=11] (12) {$\calC_{\mathsf{Tot}\,\bm{\pi}_*\L^r}$};
  \node [right of=12] (13) {$\mathsf{Tot}\,\L^r$};
  \node [below of=11, node distance=2.2cm] (21) {$\M^r$};
  \node [right of=21] (22) {$\calC^r$};
  \node [right of=22] (23) {$\calC^r$};
  \node (11b) [right of=11, above of=11, node distance=1.05cm] {$\mathsf{Tot}\,\pi_*\L$};
  \node [right of=11b] (12b) {$\calC_{\mathsf{Tot}\,\bm{\pi}_*\L}$};
  \node [right of=12b] (13b) {$\mathsf{Tot}\,\L$};
  \node [below of=11b, node distance=2.2cm] (21b) {$\M^r$};
  \node [right of=21b] (22b) {$\calC^r$};
  \node [right of=22b] (23b) {$\calC^r$};
\draw[->]  (12b) -> (11b) node[pos=0.5, above]{${\bm{\psi}}$}; \draw[->]  (12b) -> (13b) node[pos=0.5]{$\mathfrak{e}$};  
  \draw[->]  (22b) -> (21b) node[pos=0.65, above]{$\bm{\pi^r}$}; \draw[d]  (22b) -> (23b) ; 
 \draw[->] (11b) -> (21b) node[pos=0.8]{$\bm{\beta}$};  \draw[->] (12b) -> (22b) node[pos=0.8]{$\bm{\widehat{\beta}}$}; \draw[->] (13b) -> (23b) node[pos=0.5]{$\bm{\check{\beta}}$}; 
  \draw[f->]  (12) -> (11) node[pos=0.4, above]{${\bm{\varphi}}$}; \draw[f->]  (12) -> (13) node[pos=0.5]{$\mathfrak{e}'$};  
  \draw[f->]  (22) -> (21) node[pos=0.4, above]{$\bm{\pi^r}$}; \draw[fd]  (22) -> (23); 
  \draw[f->] (11) -> (21) node[pos=0.5, align=right, left]{$\bm{\alpha}$};  \draw[f->] (12) -> (22) node[pos=0.2]{$\bm{\widehat{\alpha}}$}; \draw[f->] (13) -> (23) node[pos=0.2]{$\bm{\check{\alpha}}$};
   \draw[->]  (11b) -> (11) node[pos=0.7, align=left, above]{$\bm{\tau}$}; \draw[->] (12b) -> (12) node[pos=0.7, align=left, above]{$\bm{\widehat{\tau}}$}; \draw[->] (13b) -> (13) node[pos=0.7, align=left, above]{$\bm{\check{\tau}}$}; 
  \draw[d] (21b) -> (21) ; \draw[d] (22b) -> (22) ;  \draw[d] (23b) -> (23) ; 
\end{tikzpicture} 
\label{main_pot_diagram}
\end{align}
Here $\tau$ and $\check{\tau}$ are the $r$-th power maps from \ref{rth_power_maps_on_M}, $\mathfrak{e}$ and $\mathfrak{e'}$ are the evaluatation maps of \ref{equi_pot_for_total_spaces}. The morphisms $\bm{\varphi}$, $\bm{\psi}$, $\widehat{\bm{\alpha}}$, $\widehat{\bm{\beta}}$ and $\widehat{\bm{\tau}}$ are all defined by taking appropriate cartesian diagrams\\

As in \ref{equi_pot_for_total_spaces} we have two natural morphisms 
\[
R{\bm{\psi}}_* ({\bm{\psi}}^* \eL_{\bm{\tau}} \otimes \omega_{\bm{\psi}} )[1] \longrightarrow \eL_{\bm{\tau}}
\hspace{1cm}\mbox{and}\hspace{1cm}
L\mathfrak{e}^*\eL_{\bm{\check{\tau}}}  \longrightarrow \eL_{\bm{\widehat{\tau}}} \cong  {\bm{\psi}}^* \eL_{{\bm{\tau}}}.
\]
which combine to obtain a morphism in $\mathrm{D}_b^{\mathtt{e}}(\M^r)$ of the form
\[
\ephi_{{\bm{\tau}}}: R{\bm{\psi}}_* (L\mathfrak{e}^*\eL_{\bm{\check{\tau}}} \otimes \omega_{\bm{\psi}} )[1] \longrightarrow \eL_{\bm{\tau}}.
\]
The morphism $\ephi_{{\bm{\tau}}}$ fits into the following commutative diagram with distinguished triangles as rows:
\begin{align}
\begin{array}{c}
\xymatrix@R=1.8em{
L\bm{\tau}^*\eE_{\bm{\alpha}} \ar[r]\ar[d]^{L\bm{\tau}^*\ephi_{{\bm{\alpha}}}} & \eE_{\bm{\beta}}  \ar[d]^{\ephi_{{\bm{\beta}}}} \ar[r] & \eE_{\bm{\tau}} \ar[d]^{\ephi_{{\bm{\tau}}}}\ar[r] & L\bm{\tau}^*\eE_{\bm{\alpha}}[1] \ar[d]^{L\bm{\tau}^*\ephi_{{\bm{\alpha}}}[1]}\\
L\bm{\tau}^*\eL_{\bm{\alpha}} \ar[r] &  \eL_{\bm{\beta}}\ar[r]& \eL_{\bm{\tau}} \ar[r] & L\bm{\tau}^*\eL_{\bm{\alpha}}[1]
 }  
\end{array}\label{alpha_beta_tau_pot_dist_triangle_diagram}
\end{align}
Moreover, it was shown in \cite[Lemma 4.1.1]{Leigh_Ram} that the associated morphism $\ephi_{{\bm{\tau}}}$ in $\mathrm{D}_b(\M^r)$ is a relative obstruction theory. \\
\end{re}

\begin{re}[\textit{Equivariant Perfect Relative Obstruction Theory for \texorpdfstring{$\ForMapDRtoRSM$}{nu}}] \label{equivariant_pot_nu_re}
Recall the following diagram of equivariant morphisms (\ref{M1/r_defining_diagram}) from  definition \ref{M1/r_main_definition}:
\begin{align}
\begin{array}{c}
\xymatrix@R=1.5em{
		\M^{\frac{1}{r}} \ar[r]^{\bm{i}\hspace{0.5em}}\ar[d]_{\ForMapDRtoRSM} &\mathsf{Tot}\,\bm{\pi}_*\L \ar[d]^{\bm{\tau}}\\
		\M^{r}\ar[r]_{\bm{i'}\hspace{0.5em}} &\mathsf{Tot}\,\bm{\pi}_*\L^r
} 
\end{array}\label{M1/r_defining_square_repeat}
\end{align}
There is a natural morphism $L\bm{i}^*\eL_{\bm{\tau}} \rightarrow \eL_{\ForMapDRtoRSM}$ in $\mathrm{D}_b^{\mathtt{e}}(\M^{1/r})$ which gives rise to the following diagram defining the morphism $\ephi_{\nu}: \eE_{\ForMapDRtoRSM}  \rightarrow \eL_{\ForMapDRtoRSM}$: 
\begin{align}
\begin{array}{c}
\xymatrix@R=1.5em{
L\bm{i}^*\eE_{\bm{\tau}}  \ar@{=}[d]_{}\ar[r]^{L\bm{i}^*\ephi_{\tau}} &  L\bm{i}^*\eL_{\bm{\tau}}  \ar[d]\\
\eE_{\ForMapDRtoRSM} \ar[r]^{\ephi_{\nu}} &  \eL_{\ForMapDRtoRSM}
 }
 \end{array}\label{nu_relative_pot_defining_diagram}
\end{align}
It was shown in \cite[Thm. 4.1.4]{Leigh_Ram} that the associated morphism $\phi_{\nu}: \E_{\ForMapDRtoRSM}  \rightarrow \bbL_{\ForMapDRtoRSM}$ is a perfect relative obstruction theory. \\
\end{re}

\begin{re}[\textit{Alternative Description for the Equivariant Perfect Relative Obstruction Theory for \texorpdfstring{$\ForMapDRtoRSM$}{nu}}] \label{equivariant_pot_nu_re_alt}

The equivariant perfect relative obstruction theory for $\ForMapDRtoRSM$ from \ref{equivariant_pot_nu_re} can also be constructed by considering the following commutative diagram:
\begin{align}
\begin{array}{c}
\xymatrix@R=1.5em{
\M^{\frac{1}{r}} \ar@{=}[d] & \calC^{\frac{1}{r}} \ar[l]_{\bm{\pi}^{\frac{1}{r}}} \ar@{=}[d] \ar[r]^{\widehat{\mathfrak{f}}~~} & \mathsf{Tot}_{\calC^{\frac{1}{r}}} \L \ar[r]^{~~\bm{\beta}^{\frac{1}{r}}}  \ar[d]^{\bm{\tau}^{\frac{1}{r}}} & \calC^{\frac{1}{r}} \ar@{=}[d]\\
\M^{\frac{1}{r}} & \calC^{\frac{1}{r}} \ar[l]_{\bm{\pi}^{\frac{1}{r}}}  \ar[r]^{\widehat{\mathfrak{f}}'~~} & \mathsf{Tot}_{\calC^{\frac{1}{r}}} \L^r\ar[r]^{~~\bm{\alpha}^{\frac{1}{r}}}  & \calC^{\frac{1}{r}} \\
}
 \end{array}
\end{align}
In this diagram $\bm{\alpha}^{1/r}$, $\bm{\beta}^{1/r}$ and $\bm{\tau}^{1/r}$ are the morphisms from \ref{rth_power_maps_on_M} pulled back to $\M^{1/r}$ from $\M^r$. The morphisms $\widehat{\mathfrak{f}}$ and $\widehat{\mathfrak{f}}'$ closed immersions defined by the universal sections $\bm{\sigma}$ and $\bm{\sigma}^r$ respectively. Using this notation 
\[
\eE_{\nu} = R{\pi}^{_{\frac{1}{r}}}_* (L\widehat{\mathfrak{f}}^* \eT_{\bm{\tau}^{_{\frac{1}{r}}}} \otimes \omega_{{\pi}^{_{\frac{1}{r}}}} )[1]
\]
(or $\eF_{\nu} = R\bm{\pi}^{_{1/r}}_* L\widehat{\mathfrak{f}}^*  \eT_{\bm{\tau}^{_{1/r}}}$ in the (derived) dual language of \ref{localisation_formula_section}). This was shown for the non-equivariant case in \cite[Lem. 4.1.2]{Leigh_Ram} and those methods extend to the equivariant case. \\
\end{re}

\subsection{Perfect Obstruction Theory for \texorpdfstring{$\M^{1/r}$}{M1/r}}\label{POT_for_M1/r_subsec}

\begin{re}
Using the notation of (\ref{main_pot_diagram}) and (\ref{M1/r_defining_square_repeat}) there is the following commuting diagram
\[
\xymatrix@R=1.5em{
L\bm{i'}^*\bm{\alpha}^*\eL_{\M^r} \ar[r]\ar[rd]_{\cong}& L\bm{i'}^*\eL_{\Tot \bm{\pi}_*\L^r} \ar[d]\\
 &  \eL_{\M^r}
 }
\]
 in $\mathrm{D}^{\mathtt{e}}_b(\M)$ which arises from the properties of the equivariant cotangent complex given in \ref{equivariant_cotangent_complex}. The isomorphism follows from $\bm{\alpha} \circ \bm{i'}^* = \mathrm{id}_{\M^r}$. We can extend this diagram to the following diagram with distinguished triangles as rows:
\begin{align}
\begin{array}{c}
\xymatrix@R=1.5em{
L\bm{i'}^*\bm{\alpha}^*\eL_{\M^r} \ar[r]\ar[d]_{\cong}& L\bm{i'}^*\eL_{\Tot \bm{\pi}_*\L^r} \ar[d] \ar[r] & L\bm{i'}^*\eL_{\bm{\alpha}}\ar[d]\ar[r] & L\bm{i'}^*\bm{\alpha}^*\eL_{\M^r}[1] \ar[d]_{\cong}\\
 \eL_{\M^r} \ar@{=}[r] &  \eL_{\M^r} \ar[r]& 0 \ar[r] & \eL_{\M^r}[1]
 } 
 \end{array}\label{alpha_section_composition_cotangent_diagram_full}
\end{align}
This shows that the pullback by $L\bm{i'}^*$ of the natural morphism $\eL_{\bm{\alpha}} \rightarrow \bm{\alpha}^*\eL_{\M^r}[1]$ is the zero morphism in $\mathrm{D}^{\mathtt{e}}_b(\M^r)$.\\
\end{re}

\begin{lemma} \label{POT_double_square_commuting diagram}
There exists a commuting diagram in $\mathrm{D}^{\mathtt{e}}_b(\M^{1/r})$ of the form
\[
\xymatrix@R=2em{
L\bm{i}^* \eE_{\bm{\tau}}[-1] \ar[r]\ar[d]^{L\bm{i}^*\ephi_{\bm{\tau}}[-1]} & \calE \ar[d]_{}\ar[r] & L\ForMapDRtoRSM^* \eE_{\M^r} \ar[d]^{L\bm{\ForMapDRtoRSM}^*\ephi_{\M^r}}\\
L\bm{i}^* \eL_{\bm{\tau}}[-1] \ar[r]^{v_1\hspace{1.5em}} & L\ForMapDRtoRSM^* L\bm{i'}^*\eL_{\Tot \bm{\pi}_*\L^r}  \ar[r]^{\hspace{1.5em}v_2} & L\ForMapDRtoRSM^* \eL_{\M^r}
 }
\]
where $\calE\in \mathrm{D}^{\mathtt{e}}_b(\M^{1/r})$, $v_1$ is the pullback by $L\bm{i}^*$ of the (shifted) connecting morphism $\eL_{\bm{\tau}}[-1] \rightarrow L\bm{\tau}^*\eL_{\Tot \bm{\pi}_*\L^r}$ arising from the cotangent complex distinguished triangle for $\bm{\tau}$ and $v_2$ is the pullback by $L\ForMapDRtoRSM^*$ of the canonical morphism $L\bm{i'}^*\eL_{\Tot \bm{\pi}_*\L^r}\rightarrow \eL_{\M^r}$ arising from $\bm{i'}$. 
\end{lemma}

\begin{proof}
We have the following two diagrams
\[
\xymatrix@R=2em{
 \eE_{\bm{\tau}}[-1] \ar[r]\ar[d]^{\ephi_{\bm{\tau}}[-1]}& L\bm{\tau}^* \eE_{\bm{\alpha}}   \ar[d]^{L\bm{\tau}^* \ephi_{\bm{\alpha}}}  \\
 \eL_{\bm{\tau}}[-1] \ar[r] & L\bm{\tau}^* \eL_{\bm{\alpha}}   
 }
 \hspace{1.8cm}
 \xymatrix@R=2em{
L\bm{i'}^* \eE_{\bm{\alpha}}\ar[r]^{0~~~}\ar[d]^{L\bm{i'}^* \ephi_{\bm{\alpha}}}& L\bm{i'}^*\bm{\alpha}^* \eE_{\M^r}[1]   \ar[d]^{L\bm{i'}^* \ephi_{\M^r}[1]}  \\
L\bm{i'}^* \eL_{\bm{\alpha}} \ar[r]^{0~~~} & L\bm{i'}^*\bm{\alpha}^* \eL_{\M^r}[1]   
 }
\]
where the left comes from (\ref{alpha_beta_tau_pot_dist_triangle_diagram}) and the $0$ in the bottom of the right diagram comes from the comment following (\ref{alpha_section_composition_cotangent_diagram_full}). We pull back the left by $L\bm{i}^*$ and the right by  $L\ForMapDRtoRSM^*$ to get the following diagram:
\begin{align*}
\begin{tikzpicture}[baseline=(current  bounding  box.center), node distance=3.5cm, auto,
  f->/.style={->,preaction={draw=white, -,line width=3pt}},
  d/.style={double distance=1pt},
  Dash/.style={dashed,->},
  fd/.style={double distance=1pt,preaction={draw=white, -,line width=3pt}}]
  \node (11) {$L\bm{i}^* \eL_{\bm{\tau}}[-1]$};
  \node [right of=11] (12){$L\ForMapDRtoRSM^*L\bm{i'}^*\eL_{\bm{\alpha}}$};
  \node [right of=12] (13){0};
  \node [below of=11, node distance=2.2cm] (21) {0}; 
  \node [right of=21] (22)  {$L\ForMapDRtoRSM^*L\bm{i'}^*\bm{\alpha}^*\eL_{\M^r}[1]$};
  \node [right of=22] (23)  {$L\ForMapDRtoRSM^*\eL_{\M^r}[1]$};
  \node (11b) [above right = 0.5cm and 0.0cm of 11]  {$L\bm{i}^* \eE_{\bm{\tau}}[-1]$};
  \node [right of=11b] (12b) {$L\ForMapDRtoRSM^*L\bm{i'}^*\eE_{\bm{\alpha}}$};
  \node [right of=12b] (13b) {0};
  \node [below of=11b, node distance=2.2cm] (21b){0}; 
  \node [right of=21b] (22b)  {$L\ForMapDRtoRSM^*L\bm{i'}^*\bm{\alpha}^*\eE_{\M^r}[1]$};
  \node [right of=22b] (23b)  {$L\ForMapDRtoRSM^*\eE_{\M^r}[1]$};
    \draw[->]  (11b) -> (12b) node[pos=0.7, above]{};     \draw[->]  (12b) -> (13b) node[pos=0.7, above]{};
  \draw[->]  (21b) -> (22b) node[pos=0.2, above]{};   \draw[->]  (22b) -> (23b) node[pos=0.7, above]{\footnotesize$\cong$};
 \draw[->] (11b) -> (21b) node[pos=0.7]{};  \draw[->] (12b) -> (22b) node[pos=0.2]{\footnotesize$0$}; \draw[->] (13b) -> (23b); 
  \draw[f->]  (11) -> (12) node[pos=0.4, above]{};    \draw[f->]  (12) -> (13) node[pos=0.4, above]{};  
  \draw[f->]  (21) -> (22) node[pos=0.6, above]{};   \draw[f->]  (22) -> (23) node[pos=0.7, above]{\footnotesize$\cong$};
  \draw[f->] (11) -> (21) node[pos=0.25, align=right, left]{};  \draw[f->] (12) -> (22) node[pos=0.2]{\footnotesize$0$}; \draw[f->] (13) -> (23); 
  \draw[->]  (11b) -> (11) node[pos=0.7, align=left, above]{}; \draw[->] (12b) -> (12) node[pos=0.7, align=left, above]{}; \draw[->] (13b) -> (13) node[pos=0.7, align=left, above]{}; 
  \draw[->] (21b) -> (21) ; \draw[->] (22b) -> (22) ; \draw[->] (23b) -> (23) ; 
\end{tikzpicture} 
\end{align*}
The desired diagram (shifted by $1$) is constructed by taking the cone of the vertical morphisms.
\end{proof}

\begin{remark}
All the discussion and results in section \ref{equi_pot_section} of this article still hold if we replace $\P^1$ with a smooth curve $X$ and drop the word \textit{equivariant}. In particular this is true for theorem \ref{POT_theorem}. We have have used  $\P^1$ and the equivariant language to be consistent with the other sections of the article.   \\
\end{remark}

\begin{corollary}[\textit{Theorem \ref{POT_theorem}}] \label{POT_for_M1/r_main_corollary}
There is perfect a perfect obstruction theory
\[
\ephi_{\M^{\frac{1}{r}}}:\eE_{\M^{\frac{1}{r}}}\longrightarrow  \eL_{\M^{\frac{1}{r}}}
\]
fitting into the following commutative diagram with distinguished triangles as rows:
\begin{align*}
\xymatrix@R=2.5em{
L\ForMapDRtoRSM^*\eE_{\M^r} \ar[r]\ar[d]^{L\ForMapDRtoRSM^*\ephi_{\M^r}} & \eE_{\M^{\frac{1}{r}}}  \ar[d]^{\ephi_{\M^{1/r}}} \ar[r] & \eE_{\ForMapDRtoRSM} \ar[d]^{\ephi_{\ForMapDRtoRSM} }\ar[r] & L\ForMapDRtoRSM^*\eE_{\M^r}[1] \ar[d]^{L\ForMapDRtoRSM^*\ephi_{\M^r}[1]}\\
L\ForMapDRtoRSM^*\eL_{\M^r} \ar[r] &  \eL_{\M^{\frac{1}{r}}}\ar[r]& \eL_{\ForMapDRtoRSM} \ar[r] & L\ForMapDRtoRSM^*\eL_{\M^r}[1]
 } 
\end{align*}
\end{corollary}
\begin{proof} 
We have the following commutative diagrams
\[
\xymatrix@R=2.5em{
L\bm{i}^*\E_{\bm{\tau}}  \ar@{=}[d]_{}\ar[r]^{L\bm{i}^*\phi_{\tau}} &  L\bm{i}^*\bbL_{\bm{\tau}}  \ar[d]\\
\E_{\ForMapDRtoRSM} \ar[r]^{\phi_{\nu}} &  \bbL_{\ForMapDRtoRSM}
 }
 \hspace{1.5cm}
 \xymatrix@R=2.5em{
L\bm{i}^*\bbL_{\bm{\tau}}[-1]  \ar[d]_{}\ar[r]^{} &  L\bm{i}^*L\bm{\tau}^*\bbL_{\Tot \bm{\pi}_*\L^r} \ar[d]\\
\bbL_{\ForMapDRtoRSM}[-1] \ar[r] &  L\ForMapDRtoRSM^*\bbL_{\M^r}
 }
\]
where the left comes from (\ref{nu_relative_pot_defining_diagram}). The right comes from taking the cone of the diagram in \ref{equivariant_compatibility_square_re} when the construction is applied to square (\ref{M1/r_defining_square_repeat}). Note that the top and right morphisms of the right diagram are the bottom morphisms from lemma \ref{POT_double_square_commuting diagram}. So, combining these  with the diagram from lemma \ref{POT_double_square_commuting diagram} gives the following commuting diagram (which defines the dashed line):

\begin{align*}
\begin{tikzpicture}[baseline=(current  bounding  box.center), node distance=3.5cm, auto,
  f->/.style={->,preaction={draw=white, -,line width=3pt}},
  d/.style={double distance=1pt},
  Dash/.style={dashed,->},
  fd/.style={double distance=1pt,preaction={draw=white, -,line width=3pt}}]
  \node (11) {$L\bm{i}^*\bbL_{\bm{\tau}}[-1]$ };
  \node [right of=11] (12){$L\bm{i}^*L\bm{\tau}^*\bbL_{\Tot \bm{\pi}_*\L^r}$ };
  \node [below of=11, node distance=2cm] (21) {$\bbL_{\ForMapDRtoRSM}[-1]$ };
  \node [right of=21] (22) {$L\ForMapDRtoRSM^*\bbL_{\M^r}$};
  \node (11b) [above right = 0.75cm and -0.2cm of 11]  {$L\bm{i}^*\E_{\bm{\tau}}[-1]$};
  \node [right of=11b] (12b) {$\calE$};
  \node [below of=11b, node distance=2cm] (21b){$\E_{\ForMapDRtoRSM}[-1]$};
  \node [right of=21b] (22b)  {$L\ForMapDRtoRSM^*\E_{\M^r}$};
    \draw[->]  (11b) -> (12b) node[pos=0.7, above]{};
  \draw[Dash]  (21b) -> (22b) node[pos=0.7, above]{}; 
 \draw[d] (11b) -> (21b) node[pos=0.7]{};  \draw[->] (12b) -> (22b) node[pos=0.7]{}; 
  \draw[f->]  (11) -> (12) node[pos=0.4, above]{};  
  \draw[f->]  (21) -> (22) node[pos=0.4, above]{}; 
  \draw[f->] (11) -> (21) node[pos=0.25, align=right, left]{};  \draw[f->] (12) -> (22) node[pos=0.25]{}; 
  \draw[->]  (11b) -> (11) node[pos=0.7, align=left, above]{}; \draw[->] (12b) -> (12) node[pos=0.7, align=left, above]{}; 
  \draw[->] (21b) -> (21) ; \draw[->] (22b) -> (22) ; 
\end{tikzpicture} 
\end{align*}
The desired morphism and diagram comes from taking the cones of the horizontal arrows in the bottom square.\\

To see that the associated object $\phi_{\M^{1/r}}$ in $\mathrm{D}_b(\M^{1/r})$ is a perfect obstruction theory for $\M^{1/r}$ we recall from \ref{equi_POT_for_M_re} and \ref{equivariant_pot_nu_re} that $\phi_{\M^r}$ and $\phi_{\ForMapDRtoRSM}$ are perfect obstruction theories. We have that $\E_{\ForMapDRtoRSM}[-1]$ is perfect in $[0,1]$ and $\E_{\M^r}$ is perfect in $[-1,0]$ so the cone $\mathrm{Cone}(\E_{\ForMapDRtoRSM}[-1]\rightarrow \E_{\M^r})$ is perfect in $[-1,0]$. To see that $\phi_{\M^{1/r}}$ is an obstruction theory we consider the following long exact sequence of the cohomology of the cones:
\[\xymatrix@R=0.75em{
 & \calH^{\SmNeg1}(\mathrm{cone}(\ForMapDRtoRSM^*\phi_{\M^r})) \ar[r] &\calH^{\SmNeg1}(\mathrm{cone}(\phi_{\M^{1/r}})) \ar[r]&\calH^{\SmNeg1}(\mathrm{cone}( \phi_{\ForMapDRtoRSM}))\\
\ar[r] & \calH^{0}(\mathrm{cone}(\ForMapDRtoRSM^*\phi_{\M^r})) \ar[r] &\calH^{0}(\mathrm{cone}(\phi_{\M^{1/r}})) \ar[r]&\calH^{0}(\mathrm{cone}(\phi_{\ForMapDRtoRSM} )).
}
\]
Which shows that $\calH^{\SmNeg1}(\mathrm{cone}(\phi_{\M^{1/r}})) =\calH^{0}(\mathrm{cone}(\phi_{\M^{1/r}})) =0$ and making $\phi_{\M^{1/r}}$ an obstruction theory for $\M^{1/r}$.
\end{proof}


\end{document}